\documentclass[pdflatex,sn-mathphys-num]{sn-jnl}
\setcitestyle{numbers,square}

\usepackage{graphicx}%
\usepackage{multirow}%
\usepackage{amsmath,amssymb,amsfonts}%
\usepackage{amsthm}%
\usepackage{mathrsfs}%
\usepackage[title]{appendix}%
\usepackage{xcolor}%
\usepackage{textcomp}%
\usepackage{manyfoot}%
\usepackage{booktabs}%
\usepackage{algorithm}%
\usepackage{algorithmicx}%
\usepackage{algpseudocode}%
\usepackage{cancel}
\usepackage{listings}%

\usepackage{tikz}

\usetikzlibrary{
    arrows.meta,
    positioning,
    shapes.geometric,
    calc
}

\definecolor{darkorange}{rgb}{1.0, 0.45, 0.0}

\usepackage{comment}

\usepackage{multirow}

\usepackage{subcaption}

\usepackage{soul}
\soulregister\cite7
\soulregister\ref7
\soulregister\eqref7
\soulregister\pageref7

\usepackage{cases}

\usepackage{mathabx}

\usepackage{bm}

\theoremstyle{thmstyleone}%
\newtheorem{theorem}{Theorem}[section]
\newtheorem{proposition}[theorem]{Proposition}%

\newtheorem{lemma}[theorem]{Lemma}%

\theoremstyle{thmstyletwo}%
\newtheorem{remark}{Remark}%

\theoremstyle{thmstylethree}%

\begin{document}

\title[A novel L-shaped refinement chain cuts method for two-stage stochastic programs]{A novel L-shaped refinement chain cuts method for two-stage stochastic programs}



\author[1]{\fnm{Mike} \sur{Hewitt}}\email{mhewitt3@luc.edu}

\author*[2]{\fnm{Francesca} \sur{Maggioni}}\email{francesca.maggioni@unibg.it}

\author[2]{\fnm{Andrea} \sur{Spinelli}}\email{andrea.spinelli@unibg.it}


\affil[1]{\orgdiv{Department of Information System and Supply Chain Management}, \orgname{Quinlan School of Business, Loyola University}, \orgaddress{\street{16 E Pearson St}, \city{Chicago}, \postcode{60611}, \state{IL}, \country{United States}}}

\affil[2]{\orgdiv{Department of Management, Information and Production Engineering}, \orgname{University of Bergamo}, \orgaddress{\street{Viale G. Marconi 5}, \city{Dalmine}, \postcode{24044}, \country{Italy}}}

\abstract{
This paper introduces the \textit{L-shaped refinement chain cuts method}, a novel decomposition approach for solving two-stage stochastic programs. The proposed method generalizes the classical L-shaped framework by exploiting refinement chains of scenarios, where sequences of increasingly refined scenario partitions are considered. At each refinement level, the full scenario set is partitioned into subgroups, and one subproblem is solved for each subgroup rather than for each individual scenario as in the classical L-shaped method. The proposed framework provides a unified decomposition scheme in which the classical multi-cut and single-cut L-shaped formulations arise as special cases. Theoretical convergence properties to the optimal solution of the original two-stage stochastic program are established at each refinement level. Furthermore, the relationships between consecutive refinement levels are characterized in terms of Benders cuts, leading to the development of an iterative refinement-based solution algorithm. The effectiveness of the proposed method is evaluated on a two-stage stochastic fixed-charge multicommodity network design problem under a mean-risk formulation. Computational experiments validate the benefits of the proposed approach and highlight its applicability to large-scale risk-averse stochastic optimization problems.
}




\maketitle
 
\section{Introduction}\label{sec1}

Stochastic programming \cite{BirLou2011,shapiro2021lectures} is a widely adopted framework for optimization under uncertainty. Over the years, it has been successfully applied across a broad range of domains, including finance \cite{Consigli2017}, healthcare \cite{Ozaltn2011}, transportation \cite{SpiBezJabMag2025} and logistics planning \cite{SIDDIG2025925}, to name a few. A fundamental assumption of stochastic programming is that the probability distribution of the uncertain parameters is known to the decision-maker. In practice, the distribution is represented by a finite set of scenarios, each corresponding to a possible realization of the uncertain parameters, yielding a deterministic equivalent of the original stochastic optimization problem. However, as the number of scenarios increases, the resulting optimization problem becomes increasingly difficult to solve, thereby motivating the development of decomposition approaches \cite{Ruszczyski1997}. A complementary perspective, particularly when solving the problem to optimality becomes computationally challenging, is to derive tractable bounds on its objective value. In this context, refinement chain approaches \cite{MagPfl2016} construct a hierarchy of scenario partitions through successive aggregation. Each level of the refinement chain corresponds to a partition of the scenario set into subgroups, each associated with a subproblem. Their optimal values are then combined to provide a lower bound on the original stochastic problem.

In this work, we introduce a novel decomposition method, called the \emph{L-shaped refinement chain cuts method}. This approach generalizes the classical L-shaped framework \cite{BirLou1988,SlyWet1969} by exploiting the structure of refinement chains of scenarios. Rather than solving one subproblem for each individual scenario, as in the classical L-shaped method, we solve a sequence of subproblems associated with groups of scenarios defined through a refinement chain. In what follows, we review the relevant literature on the L-shaped method and refinement chain approaches, while introducing the main concepts that motivate the proposed decomposition framework.

The L-shaped method extends the Benders decomposition approach of \cite{Benders1962} to two-stage stochastic programming. The method iteratively solves a master problem and a collection of single-scenario subproblems to generate feasibility and optimality cuts that progressively strengthen the master problem \cite{RahCraGenRei2017}. Depending on the level of scenario aggregation, the method admits either a multi-cut formulation, which generates one optimality cut per scenario \cite{BirLou1988}, or a single-cut formulation, which considers a single aggregated optimality cut per iteration \cite{SlyWet1969}.

Since its introduction, Benders decomposition has been successfully applied to a broad range of optimization problems, including nonlinear, integer, and bilevel formulations \cite{RahCraGenRei2017}. When applied to large-scale problems, particularly in the stochastic setting, a straightforward application of the classical Benders decomposition may exhibit limited computational efficiency. The main drawbacks include time-consuming iterations, weak feasibility and optimality cuts, and potential computational instability, to name a few. To mitigate these issues, various enhancement and acceleration techniques have been proposed. According to \cite{RahCraGenRei2017}, four main directions can be identified: (i) cut generation and (ii) solution generation approaches, aimed at reducing the number of iterations and producing high-quality Benders cuts and solutions (see, e.g., \cite{HosTurn2025,KalSah2025}); (iii) solution procedures, referring to the algorithms used to solve the master problem and the subproblems, with the goal of improving convergence speed and reducing computational time per iteration (see, e.g., \cite{ReiCorGenSor2009}); (iv) decomposition strategies which focus on how the problem is partitioned to obtain the initial master problem and the subproblems (see, e.g., \cite{CraHewMagRei2021}). However, most of these improvements and developments have been proposed for the classical Benders decomposition, whereas relatively few studies have specifically focused on the L-shaped method (see, e.g., \cite{BerCor-WriPau2025,CraHewMagRei2021,luo2025cutting,Ram-PicLjuMor2023}). In this setting, the scenario-based representation of uncertainty introduces an additional source of computational complexity, as the number of scenario subproblems and associated cuts can grow substantially with the number of scenarios. For very large-scale stochastic programs, obtaining an optimal solution may therefore remain computationally challenging even when decomposition methods are employed. This has led to the development of bounding and approximation techniques aimed at providing tractable estimates of the optimal objective value \cite{NarMagWal2023}. 

According to the classification proposed in \cite{Edirisinghe1999}, bounding approaches for stochastic programs can be divided into two main categories. The first category constructs bounds through \textit{functional approximations} of the recourse value functions, exploiting properties such as convexity, concavity, and piecewise linearity \cite{Birge1988bounds,EdirisingheZiemba1992,Wallace1987}. The second category derives bounds based on \textit{distributional approximations}, whereby the original probability distribution is substituted with a simpler discrete representation while preserving specified moments \cite{Edirisinghe1996,EdirisingheZiemba1994}. This category includes barycentric discretization methods \cite{FraKuhSch2011,Kuhn2005} as well as stochastic dominance-based guaranteed bounds for possibly non-Markovian processes \cite{MagPfl2019}.

Beyond functional and distributional approximation techniques, an alternative class of bounding approaches is based on \textit{scenario grouping decomposition}. This idea was first introduced in \cite{Birge1982} through the concept of pair subproblems and was subsequently extended in \cite{MaggioniAlleviBertocchi2014,MaggioniAlleviBertocchi2016}, where lower bounds on the optimal objective value are obtained by solving collections of subproblems defined over subsets of scenarios. Building on these developments, \cite{MagPfl2016} introduced the \textit{refinement chain framework}, which considers sequences of increasingly refined scenario partitions, thereby generating a hierarchical representation of the uncertainty. Such a framework generates monotonic non-decreasing sequences of lower bounds for stochastic programs with concave risk functionals. By operating directly on scenario partitions rather than on functional or distributional approximations, refinement chain-based approaches provide a scalable alternative for large-scale stochastic programs while preserving explicit monotonicity guarantees. Within the literature on scenario grouping decomposition, several related developments have been proposed. In particular, hierarchies of progressively refined scenario partitions have been used to derive increasingly tighter optimization bounds while improving computational scalability \cite{SandikciKongSchaefer2013,SandikciOzaltin2017}. Other extensions include the derivation of dual bounds through scenario set partitions \cite{BakirBolandDandurandErera2020}, the incorporation of dynamic risk measures \cite{MAHMUTOGULLARI2018595,MahmutogullariCavusAkturk2020}, and applications to distributionally robust optimization \cite{BayMagFacYan2024}. More recently, \cite{MagCavFac2026} proposed an optimization-driven framework for constructing optimal refinement chains induced by the dualization of non-anticipativity constraints.

Scenario grouping and aggregation techniques have also been extensively investigated for improving the computational tractability of two-stage stochastic programs. In \cite{SonLue2015}, an \textit{Adaptive Partition Method} (APM) based on scenario aggregation is proposed to solve two-stage stochastic programs with discrete probability distributions. The authors derive necessary and sufficient conditions under which the optimal solution can be obtained by solving the problem over a suitable scenario partition. A \textit{Generalized} version of APM, referred to as GAPM, is introduced in \cite{Ram-PicMor2022}, where scenarios are initially fully aggregated and then iteratively disaggregated into subsets. Scenario aggregation through clustering techniques is studied in \cite{HewOrtRei2022}, where representative scenarios are selected from clusters to construct smaller approximations of the original two-stage stochastic program.

A related line of research investigates the integration of scenario partitioning and aggregation techniques within the L-shaped decomposition framework, with the goal of strengthening the master problem and mitigating some of the computational limitations of the classical approach. In \cite{TruNtaSch2010}, an adaptive multi-cut method is proposed in which the aggregation level of optimality cuts evolves dynamically during the solution process. Starting from a low level of cut aggregation, the algorithm progressively increases the aggregation level as more information about the expected recourse function becomes available, thereby reducing the size of the master problem. In \cite{CraHewMagRei2021}, the authors introduce a \textit{Partial Benders Decomposition} methodology that reduces the number of feasibility cuts through partial relaxation of the original problem, while strengthening optimality cuts by generating artificial scenario subproblems that are incorporated into the master problem. The integration of GAPM into the L-shaped decomposition method is studied in \cite{Ram-PicLjuMor2023}, where sufficient conditions are provided to ensure that feasibility and optimality cuts constructed with respect to a scenario partition lead to an optimal solution. More recently, in \cite{WuZhaZhaLiaZha2026}, scenarios are clustered using neural network and machine learning techniques, with the corresponding Benders cuts aggregated within each cluster as an algorithmic enhancement for solving stochastic facility location and production planning problems.

The existing literature highlights the benefits of scenario aggregation and refinement-based representations of uncertainty in solving two-stage stochastic programs. However, the use of these ideas within the L-shaped decomposition method has received only limited attention from a general methodological perspective. To address this gap, this paper introduces a novel decomposition framework that exploits the hierarchical structure of refinement chains while preserving the convergence properties of the classical L-shaped method.

\subsection{Contributions}
In this paper, we introduce a novel L-shaped refinement chain cuts method that generalizes the classical L-shaped decomposition approach through the use of refinement chains of scenarios, yielding a new family of Benders decomposition schemes for two-stage stochastic programs. The main contributions of the paper are threefold and can be summarized as follows. First, we develop an L-shaped decomposition framework associated with each level of a refinement chain of scenarios. The proposed approach generalizes the scenario-based decomposition of the classical L-shaped method to a subgroup-based decomposition, in which one dual subproblem is solved for each subgroup of scenarios at a given refinement level. As a result, fewer dual subproblems are solved, each corresponding to a scenario subgroup. For each refinement level, a single master problem is constructed and Benders cuts specific for each subgroup are generated. We show that both the multi-cut and single-cut variants of the L-shaped method are special cases of the proposed framework. Furthermore, we prove that, at every refinement level, the proposed method converges to the optimal solution of the original two-stage stochastic program. A key feature of the proposed L-shaped refinement chain cuts framework is that it does not impose any additional structural requirements on the problem, making it applicable to a broad class of two-stage stochastic programs with continuous, integer, or binary first-stage variables and continuous recourse.

Second, we investigate the theoretical relationship between consecutive levels of the refinement chain within the proposed L-shaped framework. In particular, we characterize how the dual feasible regions at two consecutive levels of the refinement chain are related, and how this relationship is reflected in the corresponding sets of extreme points and extreme rays. We show that the feasibility cuts of the classical L-shaped method remain valid at all refinement levels, whereas optimality cuts at a given level can be constructed as convex combinations of optimality cuts from the previous level using appropriately refinement-based weights. This result enables the design of an iterative algorithm that solves the master problem at a given level of the refinement chain while incorporating optimality cuts obtained at the previous level.

Third, we assess the effectiveness of the proposed method on a two-stage stochastic fixed-charge multicommodity network design problem. While most stochastic network design problems focus on the minimization of the expected costs (see, e.g., \cite{CraHewMagRei2021,HewReiWall2021}), in this work we adopt a mean-risk formulation combining expectation and \textit{Conditional Value-at-Risk} \cite{Shapiro2011}. To the best of our knowledge, this work and \cite{Cordeau2021} represent the only contributions in the literature to consider a mean-risk objective in a stochastic network design setting.

With respect to the existing literature on scenario partitioning and aggregation within L-shaped decomposition, the works most closely related to ours are \cite{Ram-PicLjuMor2023} and \cite{TruNtaSch2010}. Both approaches are restricted to decomposition schemes based on scenario partitions, whereas the proposed framework accommodates a more general class of refinement chains, including both partition-based and fixed-scenario constructions. Compared with \cite{Ram-PicLjuMor2023}, our approach preserves convergence to the optimal solution without imposing additional conditions on the structure of the dual subproblems. Moreover, it allows uncertainty in the second-stage objective coefficients. Compared with \cite{TruNtaSch2010}, the proposed approach replaces the scenario-based solution of dual subproblems with a subgroup-based decomposition, in which one dual subproblem is solved for each subgroup of scenarios at a given refinement level. Furthermore, we establish theoretical relationships among extreme points, extreme rays, and Benders cuts across consecutive refinement levels, which form the basis of an iterative solution algorithm capable of transferring information throughout the refinement chain. Overall, these features make the proposed L-shaped refinement chain cuts method a more general and flexible decomposition framework for solving two-stage stochastic programs.

The remainder of the paper is organized as follows. Section \ref{sec_basicfactsandnotation} introduces the notation and reviews the main concepts from the literature underlying the proposed methodology. Section \ref{sec_integrating_ref_chain_into_Lshaped} presents the novel L-shaped refinement chain cuts method. The section first introduces the method at a fixed refinement level and then establishes its theoretical properties across consecutive refinement levels. Section \ref{sec_implementation} presents the implementation framework together with an iterative algorithm based on the proposed approach. Section \ref{sec_numericalresults} introduces the two-stage stochastic fixed-charge multicommodity network design problem and discusses the results of the computational experiments. Finally, Section \ref{sec_conclusions} summarizes the main findings and outlines directions for future research.

\section{Basic facts and notation} \label{sec_basicfactsandnotation}
In this section, we introduce the notation and summarize the main concepts for the development of the proposed L-shaped refinement chain cuts method. Specifically, Section \ref{sec_lshaped} reviews the formulation of a two-stage stochastic program and the structure of the classical L-shaped decomposition method \cite{BirLou1988,SlyWet1969}, while Section \ref{sec_refchain} recalls the refinement chain construction and the associated probability dissection proposed in \cite{MagPfl2016}.

\subsection{L-shaped decomposition for two-stage stochastic programs} \label{sec_lshaped}

Let $\mathbb{R}^n$ and $\mathbb{R}^{m\times n}$ denote the space of $n$-dimensional real vectors and $m \times n$ real matrices, respectively. Let $\bm{x} \in \mathcal{X}\subseteq \mathbb{R}^n$ represent first-stage decision variables, with $\mathcal{X}$ a non-empty closed set. We consider a stochastic parameter $\bm \xi$ defined on a probability space $(\Omega, \mathcal{A},\mathbb{P})$. If $\bm \xi$ evolves as a discrete-time stochastic process over the finite support $\Omega$, the information structure can be described in the form of a scenario tree. Let $\mathcal{S}:=\{\bm \omega^1,\ldots,\bm \omega^{|\mathcal{S}|}\}$ be the set of possible scenarios $\bm \omega^s\in\Omega$, each occurring with probability $p^s:=\mathbb{P}(\bm \omega^s)$. With a slight abuse of notation, in the following we use $s$ and $\bm \omega^s$ interchangeably to denote the $s$-th scenario. For each scenario $s\in\mathcal{S}$, let $\bm  y^s\in\mathbb{R}^m$ denote the second-stage decision variables, constrained to satisfy the conditions $\bm W \bm y^s+\bm T^s \bm x=\bm h^s$ and $\bm y^s\geq \bm 0$. We assume a fixed recourse matrix $\bm W\in\mathbb{R}^{p\times m}$, while the technology matrix $\bm T^s\in\mathbb{R}^{p\times n}$ and the right-hand side vector $\bm h^s\in\mathbb{R}^{p}$ depend on scenario $s\in\mathcal{S}$. We associate fixed objective function coefficients $\bm c \in \mathbb{R}^n$ with the first-stage variables and scenario-dependent costs $\bm q^s \in \mathbb{R}^m$, $s \in \mathcal{S}$ with the second-stage variables. Finally, let $\bm \xi^s:=(\bm T^s,\bm h^s, \bm q^s)$ denote the realization of the stochastic parameter $\bm \xi$ under scenario $s \in \mathcal{S}$. 

Using the notation introduced above, the corresponding two-stage stochastic program can be formulated as follows (see, e.g., \cite{BirLou2011}):
\begin{subequations} \label{eq:two_stage_problem}
\begin{align}
    \min_{\bm x} & \quad \bm c^\top \bm x+\sum_{s\in\mathcal{S}}p^s\mathcal{Q}(\bm x,\bm \xi^s) \label{eq:obj_two_stage_problem} \\
    \text{s.t.} & \quad \bm x \in \mathcal{X} \label{eq:constr1_two_stage_problem},
\end{align}
\end{subequations}
where, for each $s\in\mathcal{S}$, $\mathcal{Q}(\bm x,\bm \xi^s)$ is the \textit{recourse} (or \textit{second-stage}) \textit{problem} associated with the realization $\bm  \xi^s$, defined as follows:

\newpage
\begin{subequations} \label{eq:second_stage_problem}
\begin{align}
    \mathcal{Q}(\bm x,\bm \xi^s):= \min_{\bm y^s} & \quad \bm q^{s\top} \bm y^s \label{eq:obj_second_stage_problem} \\
    \text{s.t.} & \quad \bm W \bm y^s= \bm h^s-\bm T^s \bm x \label{eq:constr1_second_stage_problem}\\
    & \quad \bm y^s\geq \bm 0 \label{eq:constr2_second_stage_problem}.
\end{align}
\end{subequations}

Extensively, the \textit{deterministic equivalent} of Problem \eqref{eq:two_stage_problem} can be formulated as:
\begin{subequations} \label{eq:DEP}
\begin{align}
    \min_{\bm x,\bm y^s} & \quad \bm c^\top \bm x+\sum_{s\in\mathcal{S}}p^s\bm q^{s\top} \bm y^s \label{eq:obj_DEP} \\
    \text{s.t.} & \quad \bm W \bm y^s=\bm h^s-\bm T^s\bm x & \forall s\in\mathcal{S} \label{eq:constr1_DEP}\\
    & \quad \bm y^s\geq \bm 0 & \forall s\in\mathcal{S} \label{eq:constr2_DEP}\\
    & \quad \bm x \in \mathcal{X}  \label{eq:constr3_DEP}.
\end{align}
\end{subequations}

Note that the optimization in the second-stage is performed with respect to the expected value $\mathbb{E}_{\mathbb{P}}[\mathcal{Q}(\bm x,\bm \xi)]:=\sum_{s \in \mathcal{S}} p^s \bm q^{s\top} \bm y^s$. However, in a more general setting, this term can be replaced by a coherent risk measure $\varrho_{\mathbb{P}}[\mathcal{Q}(\bm x,\bm \xi)]$ admitting a linear decomposition, such as the \textit{Conditional Value-at-Risk} (CVaR) or the \textit{mean-semideviation} risk measure (see, e.g., \cite{DenRus2024}). Such extensions are explored in the numerical experiments.

In what follows, we outline the L-shaped decomposition method proposed in \cite{BirLou1988,SlyWet1969} to solve Problem \eqref{eq:two_stage_problem}. We begin by rewriting such problem as:
\begin{subequations} \label{eq:DEP_ref}
\begin{align}
    \min_{\bm x,\theta^s} & \quad \bm c^\top \bm x+\sum_{s\in\mathcal{S}}p^s\theta^s \label{eq:obj_DEP_ref} \\
    \text{s.t.} & \quad  \mathcal{Q}(\bm x,\bm \xi^s)\leq \theta^s & \forall s\in\mathcal{S} \label{eq:constr1_DEP_ref}\\
    & \quad \theta^s \in \mathbb{R} & \forall s \in \mathcal{S}\\
    & \quad \bm x \in \mathcal{X} \label{eq:constr2_DEP_ref},
\end{align}
\end{subequations}
where, for each $s\in\mathcal{S}$, $\theta^s$ is an epigraph variable of the recourse problem $\mathcal{Q}(\bm x,\bm \xi^s)$.

If first-stage variables $\bm x\in\mathcal{X}$ are fixed and the realization $\bm \xi^s$ is observed, the right-hand side of \eqref{eq:constr1_second_stage_problem} becomes known. Thus, scenario subproblem \eqref{eq:second_stage_problem} can be dualized by means of dual variables $\bm \lambda\in\mathbb{R}^p$, yielding to:
\begin{subequations} \label{eq:scenario_subproblem_DUAL}
\begin{align}
    \mathcal{Q}_{dual}(\bm x,\bm \xi^s):=\max_{\bm \lambda} & \quad (\bm h^s-\bm T^s\bm x)^\top \bm \lambda \label{eq:obj_scenario_subproblem_DUAL} \\
    \text{s.t.} & \quad \bm W^\top \bm  \lambda \leq \bm q^s \label{eq:constr1_scenario_subproblem_DUAL}.
\end{align}
\end{subequations}

If the dual feasible region $\Lambda^s:=\{\bm \lambda \in \mathbb{R}^p:\bm W^\top \bm  \lambda \leq \bm q^s\}$, defined by \eqref{eq:constr1_scenario_subproblem_DUAL}, is not empty, as in relatively complete recourse case \cite{BirLou2011}, Problem \eqref{eq:scenario_subproblem_DUAL} can be either unbounded or feasible for any possible choice of $\bm x$ \cite{RahCraGenRei2017}.

Based on duality theory, in the case of unboundedness of the dual problem \eqref{eq:scenario_subproblem_DUAL}, its corresponding primal problem \eqref{eq:second_stage_problem} is infeasible. Therefore, if $\mathcal{R}(\Lambda^s)$ is the set of extreme rays of $\Lambda^s$, there exists a direction $\bm r \in \mathcal{R}(\Lambda^s)$ for which $(\bm h^s-\bm T^s\bm x)^\top \bm r > 0$. To avoid this situation and prevent movements in any potential direction of unboundedness, the following \emph{feasibility cut} is imposed in Problem \eqref{eq:DEP_ref}:
\begin{equation} \label{eq:feasibility_cut}
   (\bm h^s-\bm T^s\bm x)^\top \bm r \leq 0 \quad \forall \bm r \in \mathcal{R}(\Lambda^s), \forall s\in\mathcal{S}.
\end{equation}

On the other hand, in the case of feasibility of Problem \eqref{eq:scenario_subproblem_DUAL}, its optimal solution is one of the extreme points $\bm e \in \mathcal{E}(\Lambda^s)$ of the dual feasible region, being $\mathcal{E}(\Lambda^s)$ the set of extreme points of $\Lambda^s$. Thus, to bound the optimal value of Problem \eqref{eq:scenario_subproblem_DUAL}, and hence of Problem \eqref{eq:second_stage_problem} by strong duality, the following \emph{optimality cut} is added to Problem \eqref{eq:DEP_ref}:
\begin{equation} \label{eq:optimality_cut}
   (\bm h^s-\bm T^s\bm x)^\top \bm e \leq \theta^s \quad \forall \bm e \in \mathcal{E}(\Lambda^s), \forall s\in\mathcal{S}.
\end{equation}

Consequently, Problem \eqref{eq:DEP_ref} can be reformulated as follows:
\begin{subequations} \label{eq:master_problem}
\begin{align}
     \min_{\bm x,\theta^s} & \quad \bm c^\top \bm x + \sum_{s\in\mathcal{S}}p^s\theta^s \label{eq:obj_master_problem} \\
    \text{s.t.} & \quad (\bm h^s-\bm T^s\bm x)^\top \bm r \leq 0 & \forall \bm r \in \mathcal{R}(\Lambda^s), \forall s\in\mathcal{S} \label{eq:feasibility_cuts}\\
    & \quad (\bm h^s-\bm T^s\bm x)^\top \bm e \leq \theta^s & \forall \bm e \in \mathcal{E}(\Lambda^s), \forall s\in\mathcal{S} \label{eq:optimality_cuts}\\
    & \quad \theta^s \in \mathbb{R} & \forall s \in \mathcal{S}\\
    & \quad \bm x \in \mathcal{X} \label{eq:constr_master_problem}.
\end{align}
\end{subequations}
We will refer to Problem \eqref{eq:master_problem} as the \emph{Benders master problem}.

Note that solving Problem \eqref{eq:master_problem} requires enumerating all the extreme points $\bm e\in \mathcal{E}(\Lambda^s)$ and all the extreme rays $\bm r\in \mathcal{R}(\Lambda^s)$ for all $s\in\mathcal{S}$, which is generally not practical. To address this issue, the L-shaped method relaxes Problem \eqref{eq:master_problem} and solves it iteratively, considering only a subset of constraints \eqref{eq:feasibility_cuts} and \eqref{eq:optimality_cuts} and yielding a trial solution $\overline{\bm x}\in\mathcal{X}$ at each iteration. The dual subproblem \eqref{eq:scenario_subproblem_DUAL} is then solved with $\bm x=\overline{\bm x}$ in \eqref{eq:obj_scenario_subproblem_DUAL} for each scenario $s \in \mathcal{S}$ with the aim of identifying any possibly violated optimality and feasibility cuts. When these violated cuts are determined, they are added to the relaxation of the master problem and the process is repeated. Note that the procedure converges in a finite number of steps since both $|\mathcal{E}(\Lambda^s)|$ and $|\mathcal{R}(\Lambda^s)|$ are finite, for all $s\in\mathcal{S}$. This solution process is often referred to as the \emph{multi-cut} version of the L-shaped method, as one cut is constructed for each scenario \cite{BirLou1988}. Conversely, it is possible to aggregate all second-stage epigraph variables $\theta^s$ into a single variable $\Theta\in\mathbb{R}$, leading to the following \emph{single-cut} reformulation of the master problem \cite{SlyWet1969}:
\begin{subequations} \label{eq:master_problem_singlecut}
\begin{align}
     \min_{\bm x,\Theta} & \quad \bm c^\top \bm x + \Theta \label{eq:obj_master_problem_singlecut} \\
    \text{s.t.} & \quad (\bm h^s-\bm T^s\bm x)^\top \bm r \leq 0 & \forall \bm r \in \mathcal{R}(\Lambda^s), \forall s\in\mathcal{S} \label{eq:feasibility_cuts_singlecut}\\
    & \quad \sum_{s \in \mathcal{S}}p^s(\bm h^s-\bm T^s\bm x)^\top \bm e^s \leq \Theta & \forall \bm e^s \in \mathcal{E}(\Lambda^s), \forall s \in \mathcal{S} \label{eq:optimality_cuts_singlecut}\\
    & \quad \Theta \in \mathbb{R}\\
    & \quad \bm x \in \mathcal{X} \label{eq:constr_master_problem_singlecut}.
\end{align}
\end{subequations}
Note that the optimality cut \eqref{eq:optimality_cuts_singlecut} can be included in the master problem only when the trial first-stage solution $\overline{\bm x}$ yields feasible subproblems $\mathcal{Q}(\overline{\bm x},\bm \xi^s)$ for all scenarios $s \in \mathcal{S}$. If this condition is not satisfied, only the feasibility cuts \eqref{eq:feasibility_cuts_singlecut} are added, until a first-stage solution that ensures feasibility across all subproblems is found, allowing the generation of the optimality cut \eqref{eq:optimality_cuts_singlecut}.

When comparing the multi-cut and single-cut formulations, the main differences lie in how much dual information each approach retains and correspondingly in the resulting size of the Benders master problem \cite{TruNtaSch2010}. The single-cut version aggregates the dual information from all single-scenario subproblems into a single optimality cut, which results in a weaker informative power. In contrast, the multi-cut method preserves the full dual information by generating one cut per scenario, often resulting in faster convergence in terms of number of iterations. However, as the number of scenarios increases, the master problem in the multi-cut formulation grows significantly because of the large number of cuts \cite{Ram-PicLjuMor2023}.

\subsection{Refinement chain of probabilities} \label{sec_refchain}
As in \cite{MagPfl2016}, assume that the support $\Omega$ of the uncertain parameter $\bm \xi$ can be decomposed into finite sequences of subsets, forming a \textit{refinement chain} with $J$ levels. Each level $j=1,\ldots,J$ of the refinement chain is a collection $\big\{\Omega_{i}^{(j)}\big\}_{i=1}^{m_j}$ of $m_j$ subsets containing scenarios sampled from $\mathcal{S}$, as follows:
\begin{equation} \label{eq:refinement_chain_omega}
\begin{array}{clc}
     \text{level }J& & \Omega_1^{(J)}=\Omega\\
     \vdots & &\vdots\\
     \text{level }j & & \big(\Omega^{(j)}_{1},\Omega^{(j)}_{2},\ldots, \Omega^{(j)}_{m_j}\big)\\
     \vdots &&\vdots\\
     \text{level }1 &&\big(\Omega^{(1)}_{1},\Omega^{(1)}_{2},\ldots, \Omega^{(1)}_{m_1}\big).
\end{array}
\end{equation}

In the construction of the refinement chain \eqref{eq:refinement_chain_omega} two properties must hold:
\begin{itemize}
    \item[(R1)] at each level $j=1,\ldots,J$, the union of subsets $\big\{\Omega_{i}^{(j)}\big\}_{i=1}^{m_j}$ covers the whole support, i.e.:
    \begin{equation*}
    \Omega=\bigcup_{i=1}^{m_j}\Omega_{i}^{(j)} \quad \forall j=1,\ldots,J;
    \end{equation*}
    \item[(R2)] each set $\Omega_{i}^{(j)}$ at level $j=2,\ldots,J$ is the union of sets $\{\Omega_{k_1}^{(j-1)}, \ldots, \Omega_{k_{N_i}}^{(j-1)}\}$ from the more refined collection at level $j-1$, i.e.:
    \begin{equation*}
    \Omega_{i}^{(j)}=\bigcup_{\ell=1}^{N_i}  \Omega_{k_\ell}^{(j-1)} \quad \forall j=2,\ldots,J, \forall i=1,\ldots,m_j.
    \end{equation*}
\end{itemize}
The refinement chain \eqref{eq:refinement_chain_omega} of $\Omega$ is associated with a corresponding chain of dissections of the probability measure $\mathbb{P}$ defined as follows:
\begin{equation} \label{eq:refinement_chain_probabilities}
\begin{array}{clc}
     \text{level } J&&\mathbb{P}_1^{(J)}=\mathbb{P}\\
     \vdots&&\vdots\\
     \text{level } j &&\big(\mathbb{P}^{(j)}_{1},\mathbb{P}^{(j)}_{2},\ldots, \mathbb{P}^{(j)}_{m_j}\big)\\
     \vdots&&\vdots\\
     \text{level } 1 &&\big(\mathbb{P}^{(1)}_{1},\mathbb{P}^{(1)}_{2},\ldots, \mathbb{P}^{(1)}_{m_1}\big).
\end{array}
\end{equation}
The following properties on the chain of probabilities must be satisfied:
\begin{itemize}
    \item[(P1)] each probability measure $\mathbb{P}_{i}^{(j)}$ has support $\Omega_{i}^{(j)}$, $\forall j=1,\ldots,J$, $\forall i=1,\ldots,m_j$;
    
    \item[(P2)] at each level $j=1,\ldots,J$, the probability measure $\mathbb{P}$ can be decomposed into a convex combination of probabilities $\big\{\mathbb{P}_{i}^{(j)}\big\}_{i=1}^{m_j}$ with weights $\pi_{i}^{(j)}$, i.e.: 
    \begin{equation} \label{eq:intra_level_probability}
    \mathbb{P} = \sum_{i=1}^{m_j} \pi_{i}^{(j)} \mathbb{P}_{i}^{(j)},
    \end{equation}
    where $\pi_{i}^{(j)} \geq 0$ $\forall i=1,\ldots,m_j$ and  $\sum_{i=1}^{m_j} \pi_{i}^{(j)}=1$, $\forall j=1,\ldots,J$;
    
    \item[(P3)] for each level $j=2,\ldots,J$ and for each $i=1,\ldots,m_j$, the probability measure $\mathbb{P}_{i}^{(j)}$ can be decomposed into a convex combination of probability measures $\{\mathbb{P}_{k_1}^{(j-1)},\ldots,\mathbb{P}_{k_{N_i}}^{(j-1)}\}$ from the refined collection at level $j-1$ with weights $\{\pi_{k_1,i}^{(j-1,j)},\ldots,\pi_{k_{N_i},i}^{(j-1,j)}\}$, i.e.:
    \begin{equation} \label{eq:inter_level_probability}
    \mathbb{P}_{i}^{(j)} = \sum_{\ell=1}^{N_i} \pi_{k_\ell,i}^{(j-1,j)} \mathbb{P}_{k_\ell}^{(j-1)},
    \end{equation}
    with $\pi_{k_\ell,i}^{(j-1,j)} \geq 0$ $\forall \ell=1,\ldots,N_i$ and $\sum_{\ell=1}^{N_i} \pi_{k_\ell,i}^{(j-1,j)}=1$, $\forall j=2,\ldots,J$. 
\end{itemize}

Within the notation introduced so far, we will refer to $\pi_{i}^{(j)}$ and $\pi_{k_\ell,i}^{(j-1,j)}$ as \emph{intra-level} and \emph{inter-level} weights, respectively.

A refinement chain can be constructed in multiple ways, provided that properties (R1)--(R2) and (P1)--(P3) are satisfied. However, two natural approaches can be considered: (i) keeping one or more scenarios fixed across subsets, so they appear in all subsets at all refinement levels; (ii) forming disjoint partitions at each refinement level. In the following, we examine these two cases further.

\paragraph{Fixed scenarios}
Assume that at level $j\in\{1,\ldots,J\}$, the first $f\geq 1$ scenarios are fixed and included in all subset $\Omega_{i}^{(j)}$, $\forall i=1,\ldots,m_j$. Consequently, the probability measure $\mathbb{P}_{i}^{(j)}$ supported on $\Omega_{i}^{(j)}$ can be represented as the sum of Dirac measures $\delta_{\bm \omega^s}$ over the first $f$ scenarios and the remaining ones, as follows:
\begin{equation*}
    \mathbb{P}_{i}^{(j)}=\sum_{s=1}^fp^s \delta_{\bm \omega^{s}}+\sum_{s=f+1}^{| \Omega_{i}^{(j)} |} \frac{1-\sum_{k=1}^{f}{p^k}}{\sum_{k=f+1}^{| \Omega_{i}^{(j)}|} p^k}\hspace*{2pt}p^s\delta_{\bm \omega^{s}},
\end{equation*}
where $|\Omega_{i}^{(j)}|$ represents the number of scenarios in $\Omega_{i}^{(j)}$. Thus, by defining $p^s_{\Omega_{i}^{(j)}}:=\mathbb{P}_{i}^{(j)}(\omega^{s})$ as the probability of scenario $\bm \omega^{s}$ as an element of subset $\Omega_{i}^{(j)}$, we have that:
\begin{equation} \label{eq:prob_scenario_ffixed}
 p^s_{\Omega_{i}^{(j)}}= \begin{cases}
        p^{s} & \forall s=1,\ldots,f  \\
        \displaystyle\frac{1-\sum_{k=1}^{f}p^{k}}{\sum_{k=f+1}^{| \Omega_{i}^{(j)}|} p^{k}} \hspace*{2pt} p^{s} & \forall s=f+1,\ldots,| \Omega_{i}^{(j)}|.
    \end{cases}
\end{equation}
The intra-level weights $\pi_{i}^{(j)}$ in the convex combination \eqref{eq:intra_level_probability} are calculated as follows:
\begin{equation}  \label{eq:weights_scenario_ffixed}
    \pi_{i}^{(j)} = \frac{\sum_{k=f+1}^{|\Omega_{i}^{(j)}|} p^{k}}{1-\sum_{k=1}^{f}p^{k}}.
\end{equation}

\paragraph{Disjoint partitions} Suppose that at level $j\in\{1,\ldots,J\}$ all subsets $\Omega_{i}^{(j)}$ are disjoint. In this case, the probability measure $\mathbb{P}_{i}^{(j)}$ equals to:
\begin{equation*}
    \mathbb{P}_{i}^{(j)} = \sum_{s:\bm \omega^s\in\Omega_{i}^{(j)}} \frac{p^s}{\sum_{k:\bm \omega^k\in\Omega_{i}^{(j)}}p^{k}}\hspace*{2pt}\delta_{\bm \omega^{s}},
\end{equation*}
and accordingly the probability $p^s_{\Omega_{i}^{(j)}}$ is:
\begin{equation}  \label{eq:prob_scenario_disjoint}
    p^s_{\Omega_{i}^{(j)}}=\frac{p^s}{\sum_{k:\bm \omega^k\in\Omega_{i}^{(j)}}p^k}.
\end{equation}
For disjoint partitions, the intra-level weights $\pi_{i}^{(j)}$ are as follows:
\begin{equation} \label{eq:weights_scenario_disjoint}
    \pi_{i}^{(j)}=\sum_{k:\bm \omega^k\in\Omega_{i}^{(j)}} p^k.
\end{equation}

The following remarks hold for the two considered approaches (fixed scenarios and disjoint partitions), when examining two consecutive levels $j$ and $j-1$ of the refinement chain, with $j\in\{2,\ldots,J\}$.

\begin{remark}
    The inter-level weights $\pi_{k_\ell,i}^{(j-1,j)}$ can be expressed in closed form as a function of the corresponding intra-level weights $\pi_{k_\ell}^{(j-1)}$ and $\pi_i^{(j)}$, as established in the following lemma.
    \begin{lemma}\label{lemma_relation_phi_f_disjoint}
Fix a refinement level $j\in\{2,\ldots,J\}$. For every $i=1,\ldots,m_j$ and $\ell=1,\ldots,N_i$, it holds that:
\begin{equation}\label{eq:relation_phi_f_disjoint}
    \pi_{k_\ell,i}^{(j-1,j)} = \frac{\pi_{k_\ell}^{(j-1)}}{\pi_i^{(j)}}.
\end{equation}
\end{lemma}
        \begin{proof}
    See Appendix \ref{sec_app_A1}.
        \end{proof}
\end{remark}

\begin{remark}
    The relationship between the intra-level weights associated with two consecutive refinement levels can be characterized as follows.
\end{remark}
    \begin{lemma}\label{lemma_relations_intralevelwights_consecutivelevels}
        According to property (R2), suppose that $\Omega_{i}^{(j)}=\bigcup_{\ell=1}^{N_i} \Omega_{k_\ell}^{(j-1)}$, with $j\in\{2,\ldots,J\}$ and $i\in\{1,\ldots,m_j\}$. Then:
        \begin{equation}\label{eq_relations_intralevelwights_consecutivelevels}
            \pi_{i}^{(j)}=\sum_{\ell=1}^{N_i}\pi_{k_\ell}^{(j-1)}.
        \end{equation}
    \end{lemma}
\begin{proof}
    See Appendix \ref{sec_app_A2}.
\end{proof}

The construction of the refinement chain reviewed so far, together with the results on the L-shaped method discussed in Section \ref{sec_lshaped}, provides the theoretical foundation required for the development of the proposed L-shaped refinement chain cuts method presented in the following sections.

\section{The L-shaped refinement chain cuts method} \label{sec_integrating_ref_chain_into_Lshaped}
In this section, we present the proposed L-shaped refinement chain cuts method. Section \ref{sec_lshaped_fixedlevel} first introduces the method at a fixed refinement level. Then, Section \ref{sec_lshaped_consecutivelevels} establishes its theoretical properties across consecutive refinement levels by analyzing the relationships between dual feasible regions and between Benders cuts.

\subsection{The L-shaped refinement chain cuts method at a fixed refinement level} \label{sec_lshaped_fixedlevel}
As described in Section \ref{sec_lshaped}, in the classical L-shaped decomposition method a sequence of $|\mathcal{S}|$ single-scenario dual subproblems \eqref{eq:scenario_subproblem_DUAL} is iteratively solved to identify feasibility and optimality cuts \eqref{eq:feasibility_cuts}--\eqref{eq:optimality_cuts}. Suppose now that a refinement chain \eqref{eq:refinement_chain_omega} is provided and consider a given level $j\in\{1,\ldots,J\}$ composed by $m_j$ subsets \{$\Omega_1^{(j)},\ldots,\Omega_{m_j}^{(j)}\}$. We begin by reformulating Problem \eqref{eq:DEP_ref} so that, in the second stage, scenarios are not treated individually but are instead grouped within these subsets.

For each $i=1,\ldots,m_j$, we introduce $\theta_{i}^{(j)}$ as an epigraph variable of the recourse function $\mathcal{Q}(\bm x,\bm \xi_{i}^{(j)})$ defined over subset $\Omega_i^{(j)}$ of scenarios, where $\bm \xi^{(j)}_i:=(\bm T^{(j)}_i,\bm h^{(j)}_i, \bm q^{(j)}_i)$ denotes the realization of the uncertain parameters within $\Omega_{i}^{(j)}$.

The resulting \textit{two-stage stochastic problem at level $j$} can thus be written as:
\begin{subequations} \label{eq:DEP_ref_omega_i_j}
\begin{align}
    \min_{\bm x,\theta^{(j)}_i} & \quad \bm c^\top \bm x + \sum_{i=1}^{m_j}\pi_{i}^{(j)}\theta^{(j)}_i \label{eq:obj_DEP_ref_omega_i_j} \\
    \text{s.t.} & \quad  \mathcal{Q}(\bm x,\bm \xi^{(j)}_i)\leq \theta^{(j)}_i & \forall i=1,\ldots,m_j \label{eq:constr1_DEP_ref_omega_i_j} \\
    & \quad \theta_i^{(j)} \in \mathbb{R} & \forall i = 1,\ldots,m_j \label{eq:constr2_DEP_ref_omega_i_j} \\
    & \quad \bm x \in \mathcal{X} \label{eq:constr3_DEP_ref_omega_i_j},
\end{align}
\end{subequations}
where the recourse problem $\mathcal{Q}(\bm x,\bm\xi_i^{(j)})$ is defined as:
\begin{subequations} \label{eq:scenario_subproblem_omega_i_j}
\begin{align}
    \mathcal{Q}(\bm x,\bm \xi^{(j)}_i):=\min_{\bm y^s:\bm \omega^s\in\Omega_{i}^{(j)}} & \quad \sum_{s:\bm \omega^s\in\Omega_{i}^{(j)}}p^s_{\Omega_{i}^{(j)}}\bm q^{s\top} \bm y^s \label{eq:obj_scenario_subproblem_omega_i_j} \\
    \text{s.t.} & \quad \bm W\bm y^s=\bm h^s-\bm T^s\bm x & \forall s:\bm \omega^s\in\Omega_{i}^{(j)} \label{eq:constr1_scenario_subproblem_omega_i_j}\\
    & \quad \bm y^s\geq \bm 0 & \forall s:\bm \omega^s\in\Omega_{i}^{(j)} \label{eq:constr2_scenario_subproblem_omega_i_j}.
\end{align}
\end{subequations}
The weights $\pi_i^{(j)}$ in the objective function \eqref{eq:obj_DEP_ref_omega_i_j} correspond to the intra-level weights introduced in Section \ref{sec_refchain}.

For each scenario $\bm \omega^s\in\Omega_{i}^{(j)}$, we introduce a dual variable $\bm \lambda^s\in\mathbb{R}^p$ associated with constraint \eqref{eq:constr1_scenario_subproblem_omega_i_j}. The dual problem of Problem \eqref{eq:scenario_subproblem_omega_i_j} is then given by:
\begin{subequations} \label{eq:scenario_subproblem_DUAL_omega_i_j}
\begin{align}
    \mathcal{Q}_{dual}(\bm x,\bm \xi_i^{(j)}):=\max_{\bm \lambda^s:\bm \omega^s\in\Omega_{i}^{(j)}} & \quad \sum_{s:\bm\omega^s\in\Omega_{i}^{(j)}}(\bm h^s-\bm T^s\bm x)^\top \bm \lambda^s \label{eq:obj_scenario_subproblem_DUAL_omega_i_j} \\
    \text{s.t.} & \quad \bm W^\top \bm \lambda^s \leq p^s_{\Omega_{i}^{(j)}}\bm q^s  \qquad \forall s:\bm \omega^s\in\Omega_{i}^{(j)}  \label{eq:constr1_scenario_subproblem_DUAL_omega_i_j},
\end{align}
\end{subequations}
where constraints \eqref{eq:constr1_scenario_subproblem_DUAL_omega_i_j} define the dual feasible region $(\Lambda^{(j)}_i)^s :=\{\bm \lambda^s \in \mathbb{R}^p:\bm W^\top \bm \lambda^s \leq p^s_{\Omega_{i}^{(j)}}\bm q^s\}$ for each scenario $\bm \omega^s\in\Omega_{i}^{(j)}$. Considering the collection of all scenarios in $\Omega_{i}^{(j)}$, the dual feasible region $\Lambda^{(j)}_i$ of Problem \eqref{eq:scenario_subproblem_DUAL_omega_i_j} is then:
\begin{equation}\label{eq:feasible_region_dual_omega_i_j}
    \Lambda^{(j)}_i:=\bigtimes_{s:\bm \omega^s\in\Omega_{i}^{(j)}} (\Lambda^{(j)}_i)^s \subseteq\mathbb{R}^{p \cdot |\Omega_{i}^{(j)}|}.
\end{equation}
Let $\mathcal{E}(\Lambda^{(j)}_i)$ and $\mathcal{R}(\Lambda^{(j)}_i)$ denote the sets of extreme points and extreme rays of $\Lambda^{(j)}_i$, respectively, for each $i=1,\ldots,m_j$. Based on this construction, we define the \textit{Benders master problem at level $j$ of the refinement chain} as:
\begin{subequations} \label{eq:master_problem_omega_i_j}
\begin{align}
     \min_{\bm x,\theta^{(j)}_i} & \quad \bm c^\top \bm x + \sum_{i=1}^{m_j}\pi_{i}^{(j)}\theta^{(j)}_i \label{eq:obj_master_problem_omega_i_j} \\
    \text{s.t.} & \quad (\bm h^{(j)}_i-\bm T^{(j)}_i\bm x)^\top \bm r \leq 0 & \forall \bm r \in \mathcal{R}(\Lambda^{(j)}_i), \forall i=1,\ldots,m_j \label{eq:feasibility_cuts_omega_i_j}\\
    & \quad (\bm h^{(j)}_i-\bm T^{(j)}_i\bm x)^\top \bm e \leq \theta^{(j)}_i & \forall \bm e \in \mathcal{E}(\Lambda^{(j)}_i), \forall i=1,\ldots,m_j \label{eq:optimality_cuts_omega_i_j}\\
    & \quad \theta_i^{(j)} \in \mathbb{R} & \forall i = 1,\ldots,m_j \label{eq:constr1_master_problem_omega_i_j} \\
    & \quad \bm x \in \mathcal{X} \label{eq:constr2_master_problem_omega_i_j}.
\end{align}
\end{subequations}

Compared with the classical L-shaped formulation, the dual feasible region $\Lambda^{(j)}_i$ in \eqref{eq:feasible_region_dual_omega_i_j} is not fixed, as it depends on the composition of the scenario subset $\Omega_{i}^{(j)}$. Similarly, the dual feasible region $\Lambda^s$ in \eqref{eq:constr1_scenario_subproblem_DUAL} depends on the scenario $s\in\mathcal{S}$. Nevertheless, at each iteration of Problem \eqref{eq:master_problem_omega_i_j}, only $m_j$ dual subproblems \eqref{eq:scenario_subproblem_DUAL_omega_i_j} need to be solved, rather than one for each scenario, as in the classical multi-cut L-shaped method.

The following proposition shows that the multi-cut version of the L-shaped method \cite{BirLou1988} can be viewed as a special case of the proposed L-shaped refinement chain cuts method.

\begin{proposition} \label{prop:multicut}
Consider a refinement chain with $J$ levels based on disjoint scenario subsets. Then, the Benders master problem \eqref{eq:master_problem_omega_i_j} at the finest refinement level, i.e., $j=1$, coincides with the classical multi-cut Benders master problem \eqref{eq:master_problem}.
\end{proposition}

\begin{proof}
At the finest refinement level, each subset of the refinement chain contains exactly one scenario. Hence, $m_1=|\mathcal{S}|$, and for every scenario $s\in\mathcal{S}$, the corresponding subset is $\Omega_{s}^{(1)}=\{\bm\omega^s\}$. Therefore, the conditional probability satisfies $p^s_{\Omega_{s}^{(1)}}=1$, while the corresponding intra-level weight is $\pi_s^{(1)}=p^s$, according to \eqref{eq:prob_scenario_disjoint} and \eqref{eq:weights_scenario_disjoint}, respectively. Substituting into \eqref{eq:master_problem_omega_i_j} shows that, at level $j=1$, the Benders master problem \eqref{eq:master_problem_omega_i_j} coincides with the classical multi-cut Benders master problem \eqref{eq:master_problem}.
\end{proof}

The single-cut version of the L-shaped method can be also interpreted as a special case of the proposed approach, corresponding to the coarsest refinement level, i.e., $j=J$, of the refinement chain \eqref{eq:refinement_chain_omega}. However, the equivalence between the two formulations is not straightforward, as it requires proving that the optimality cuts \eqref{eq:optimality_cuts_singlecut} and \eqref{eq:optimality_cuts_omega_i_j} are equivalent. This result, together with the theoretical convergence properties of the proposed approach, is established in the next section.

\subsection{The L-shaped refinement chain cuts method across consecutive refinement levels} \label{sec_lshaped_consecutivelevels}
In this section, we investigate the theoretical relationships between the dual feasible regions and the Benders optimality and feasibility cuts associated with two consecutive refinement levels of the refinement chain \eqref{eq:refinement_chain_omega}. These results provide the foundation for establishing the convergence of the proposed L-shaped refinement chain cuts method, proved at the end of this section.

Section \ref{subsec_extpoint} analyzes the relationships between the dual feasible regions and their extreme points at consecutive refinement levels, whereas Section \ref{subsec_benders_cuts_theory} investigates the corresponding relationships between Benders optimality and feasibility cuts.

\subsubsection{Theoretical properties of dual feasible regions and extreme points} \label{subsec_extpoint}

According to property (R2), suppose that at level $j\in\{2,\ldots,J\}$ of the refinement chain, set $\Omega^{(j)}_i$, with $i\in\{1,\ldots,m_j\}$, is the union of $N_i$ subsets $\{\Omega_{k_\ell}^{(j-1)}\}_{\ell=1}^{N_i}$ from level $j-1$.
The set $\Omega^{(j)}_i$ and each subset $\Omega_{k_\ell}^{(j-1)}$ are associated with a recourse problem \eqref{eq:scenario_subproblem_omega_i_j}, having corresponding dual feasible regions $\Lambda^{(j)}_i \subseteq\mathbb{R}^{p \cdot |\Omega_{i}^{(j)}|}$ and $\Lambda^{(j-1)}_{k_\ell}\subseteq \mathbb{R}^{p \cdot |\Omega_{k_\ell}^{(j-1)}|}$, respectively. In general, each $\Omega_{k_\ell}^{(j-1)}$ may contain a different number of scenarios, so its cardinality $|\Omega_{k_\ell}^{(j-1)}|$ may vary across the $N_i$ subsets. However, since by construction $|\Omega_{i}^{(j)}|\geq |\Omega_{k_\ell}^{(j-1)}|$ for all $\ell=1,\ldots,N_i$ and therefore $p\cdot |\Omega_{i}^{(j)}|\geq p  \cdot |\Omega_{k_\ell}^{(j-1)}|$, we assume that every point in each $\Lambda^{(j-1)}_{k_\ell}$ has exactly $p \cdot |\Omega_{i}^{(j)}|$ components. This can be obtained by setting to zero all components corresponding to scenarios not present in $\Omega_{k_\ell}^{(j-1)}$.

In the next two lemmas, we begin by examining the connections between the dual feasible region $\Lambda^{(j)}_{i}$ at level $j$ with the collection of dual feasible regions $\{\Lambda^{(j-1)}_{k_\ell}\}_{\ell=1}^{N_i}$ at level $j-1$ of the refinement chain.

\begin{lemma} \label{lemma:feasibility_convex_comb}
Let $\{\bm \lambda_{k_\ell}^{(j-1)}\}_{\ell=1}^{N_i}$ be a collection of points such that $\bm \lambda_{k_\ell}^{(j-1)}\in\Lambda^{(j-1)}_{k_\ell}$ for all $\ell=1,\ldots,N_i$. Then:
\begin{equation*}
    \sum_{\ell=1}^{N_i}\pi_{k_\ell,i}^{(j-1,j)}\bm \lambda_{k_\ell}^{(j-1)},
\end{equation*}
belongs to $\Lambda^{(j)}_i$.
\end{lemma}

\begin{proof}
    See Appendix \ref{sec_app_A3}.
\end{proof}

This result implies that any convex combination of points in $\{\Lambda_{k_\ell}^{(j-1)}\}_{\ell=1}^{N_i}$ with coefficients given by the inter-level weights $\{\pi_{k_\ell,i}^{(j-1,j)}\}_{\ell=1}^{N_i}$ belongs to $\Lambda_i^{(j)}$. Conversely, in the two cases discussed in Section \ref{sec_refchain} for the construction of the refinement chain (fixed scenarios and disjoint partitions), each point in $\Lambda_i^{(j)}$ can be characterized as a convex combination of points in $\{\Lambda_{k_\ell}^{(j-1)}\}_{\ell=1}^{N_i}$, as shown in the following lemma.

\begin{lemma}\label{lemma:suff_cond}
    Suppose that the subsets $\{\Omega_{k_\ell}^{(j-1)}\}_{\ell=1}^{N_i}$ are either disjoint or with fixed scenarios appearing in all of them. Then, for each $\bm \lambda^{(j)}_i \in \Lambda^{(j)}_i$ there exists a collection of points $\{\bm \lambda_{k_\ell}^{(j-1)}\}_{\ell=1}^{N_i}$ with $\bm \lambda_{k_\ell}^{(j-1)}\in\Lambda^{(j-1)}_{k_\ell}$ for all $\ell=1,\ldots,N_i$ such that:
    \begin{equation} \label{eq_thesis_lemma}
    \sum_{\ell=1}^{N_i}\pi_{k_\ell,i}^{(j-1,j)}\bm \lambda_{k_\ell}^{(j-1)}=\bm \lambda^{(j)}_{i}.
\end{equation}
\end{lemma}

\begin{proof}
    See Appendix \ref{sec_app_A4}.
\end{proof}

In the construction of the optimality cuts \eqref{eq:optimality_cuts_omega_i_j}, the extreme points of the dual feasible region $\Lambda^{(j)}_i$ at level $j$ are considered. The following proposition establishes a connection between the extreme points of $\Lambda^{(j)}_i$ at level $j$ and those of the collection $\{\Lambda^{(j-1)}_{k_\ell}\}_{\ell=1}^{N_i}$ at level $j-1$.

\begin{proposition}\label{prop:extremepoints_omegaij}
Suppose that the subsets $\{\Omega_{k_\ell}^{(j-1)}\}_{\ell=1}^{N_i}$ are either disjoint or with fixed scenarios appearing in all of them. Let $\{\bm e_{k_\ell}^{(j-1)}\}_{\ell=1}^{N_i}$ be a collection of points such that $\bm e_{k_\ell}^{(j-1)}\in\Lambda^{(j-1)}_{k_\ell}$ for all $\ell=1,\ldots,N_i$, with corresponding objective function values in \eqref{eq:obj_scenario_subproblem_DUAL_omega_i_j} denoted by $\{z_{k_\ell}^{(j-1)}\}_{\ell=1}^{N_i}$. Then:
\begin{equation}\label{eq:extr_point}
    \bm {e}^{(j)}_{i}:=\sum_{\ell=1}^{N_i}\pi_{k_\ell,i}^{(j-1,j)}\bm e_{k_\ell}^{(j-1)},
\end{equation}
is an extreme point of $\Lambda^{(j)}_i$ if and only if each $\bm e_{k_\ell}^{(j-1)}$ is an extreme point of $\Lambda_{k_\ell}^{(j-1)}$ for all $\ell=1,\ldots,N_i$. In this case, the corresponding objective function value $z_i^{(j)}$ satisfies:
\begin{equation*}
    z_i^{(j)}=\sum_{\ell=1}^{N_i}\pi_{k_\ell,i}^{(j-1,j)}z_{k_\ell}^{(j-1)}.
\end{equation*}
\end{proposition}

\begin{proof}
    See Appendix \ref{sec_app_A5}.
\end{proof}

Thanks to Proposition \ref{prop:extremepoints_omegaij}, we can show that the single-cut version of the L-shaped approach \cite{SlyWet1969} can be interpreted as a special case of our L-shaped refinement chain cuts method. This equivalence is formally stated and proved in the following proposition.

\begin{proposition}
    Consider a refinement chain with $J$ levels, composed either of disjoint subsets or of subsets with fixed scenarios. Then, the single-cut optimality cut \eqref{eq:optimality_cuts_singlecut} coincides with the optimality cut \eqref{eq:optimality_cuts_omega_i_j} at the top level $J$.
\end{proposition}
\begin{proof}
    Assume that at level 1 of the refinement chain each subset contains a single scenario. Note that, in the case of a refinement chain with fixed scenarios, it is always possible to introduce an additional bottom level consisting of single-scenario subsets, so this assumption is without loss of generality.
    
    At level $J$, we have $\Omega_{1}^{J}=\Omega$. Consequently, in \eqref{eq:optimality_cuts_omega_i_j} it holds that $\bm h_{1}^{(J)}=\bm h$, $\bm T_{1}^{(J)}=\bm T$, and $\theta_{1}^{(J)}=\Theta$. Therefore, to establish the equivalence between \eqref{eq:optimality_cuts_singlecut} and \eqref{eq:optimality_cuts_omega_i_j}, it suffices to show that $\sum_{s\in\mathcal{S}}p^s\bm e^s$, where $\bm e^s$ are extreme points at the bottom level, is an extreme point of $\Lambda_{1}^{(J)}$.
    
    Since single scenarios appear at level 1 of the refinement chain, we have $m_1=|\mathcal{S}|$, and, from \eqref{eq:weights_scenario_disjoint}, it follows that $p^s=\pi_{s}^{(1)}$ for all $s\in\mathcal{S}$. Considering levels 1 and 2 and applying \eqref{eq:relation_phi_f_disjoint}, we obtain:
    \begin{equation*}
        \begin{split}
            \sum_{s \in \mathcal{S}} p^s\bm e^s & = \sum_{s=1}^{m_1} \pi_{s}^{(1)}\bm e^s\\
            & = \sum_{s_2=1}^{m_2}\sum_{s:\Omega_{s}^{(1)}\subseteq \Omega_{s_2}^{(2)}} \pi_{s}^{(1)}\bm e^{s}\\
            & = \sum_{s_2=1}^{m_2}\sum_{s:\Omega_{s}^{(1)}\subseteq \Omega_{s_2}^{(2)}} \pi_{s,s_2}^{(1,2)}\pi_{s_2}^{(2)}\bm e^s\\
            & = \sum_{s_2=1}^{m_2} \pi_{s_2}^{(2)} \sum_{s:\Omega_{s}^{(1)}\subseteq \Omega_{s_2}^{(2)}} \pi_{s,s_2}^{(1,2)}\bm e^s\\
            & = \sum_{s_2=1}^{m_2} \pi_{s_2}^{(2)} \bm e_{s_2}^{(2)},
        \end{split}
    \end{equation*}
    where, in the last equality, we applied Proposition \ref{prop:extremepoints_omegaij}, since $\bm e_{s_2}^{(2)}:=\sum_{s:\Omega_{s}^{(1)}\subseteq \Omega_{s_2}^{(2)}} \pi_{s,s_2}^{(1,2)}\bm e^s$ is an extreme point of $\Lambda_{s_2}^{(2)}$, with $s_2\in\{1,\ldots,m_2\}$. The same procedure can be applied recursively along the refinement chain up to level $J-1$. At level $J$, there is a single set $\Omega_{1}^{(J)}=\Omega$, whose unique intra-level weight satisfies $\pi_{1}^{(J)}=1$ by property (P2). Applying \eqref{eq:relation_phi_f_disjoint} yields:
    \begin{equation*}
        \sum_{s \in \mathcal{S}}p^s\bm e^s=\sum_{s_{J-1}=1}^{m_{J-1}}\pi_{s_{J-1}}^{(J-1)}\bm e_{s_{J-1}}^{(J-1)} = \sum_{s_{J-1}=1}^{m_{J-1}}\pi_{s_{J-1},1}^{(J-1,J)}\bm e_{s_{J-1}}^{(J-1)},
    \end{equation*}
    which, by Proposition \ref{prop:extremepoints_omegaij}, is an extreme point of $\Lambda_{1}^{(J)}$. This completes the proof.
\end{proof}

In the following proposition, we establish the relationship between optimal extreme points at level $j-1$ and those at level $j$ of the refinement chain.

\begin{proposition} \label{prop:optimal_extreme_points}
Suppose that the subsets $\{\Omega_{k_\ell}^{(j-1)}\}_{\ell=1}^{N_i}$ are either disjoint or with fixed scenarios appearing in all of them. Let $\{\bm e_{k_\ell}^{(j-1)}\}_{\ell=1}^{N_i}$ be a collection of extreme points such that $\bm e_{k_\ell}^{(j-1)}\in\mathcal{E}(\Lambda^{(j-1)}_{k_\ell})$ for all $\ell=1,\ldots,N_i$, with corresponding objective function values \eqref{eq:obj_scenario_subproblem_DUAL_omega_i_j} denoted by $\{z_{k_\ell}^{(j-1)}\}_{\ell=1}^{N_i}$. Then:
\begin{equation}\label{eq:extr_opt_point}
    \bm {e}^{(j)}_{i}:=\sum_{\ell=1}^{N_i}\pi_{k_\ell,i}^{(j-1,j)}\bm e_{k_\ell}^{(j-1)},
\end{equation}
is an optimal extreme point of problem \eqref{eq:scenario_subproblem_DUAL_omega_i_j} at level $j$ if and only if each $\bm e_{k_\ell}^{(j-1)}$ is an optimal extreme point of the corresponding problem at level $j-1$ for all $\ell=1,\ldots,N_i$.
\end{proposition}

\begin{proof}
    See Appendix \ref{sec_app_A6}.
\end{proof}

\subsubsection{Theoretical properties of Benders cuts}\label{subsec_benders_cuts_theory}
In this section, we investigate the relationship between consecutive levels $j-1$ and $j$ of the refinement chain in terms of Benders cuts. We begin by establishing the connection between the corresponding optimality cuts in the following proposition.

\begin{proposition} \label{prop:optimality_cuts_prop}
  Suppose that subsets $\{\Omega_{k_\ell}^{(j-1)}\}_{\ell=1}^{N_i}$ are either disjoint or with fixed scenarios appearing in all of them. Let $\{\bm e_{k_\ell}^{(j-1)}\}_{\ell=1}^{N_i}$ be a collection of extreme points such that $\bm e_{k_\ell}^{(j-1)}\in\mathcal{E}(\Lambda^{(j-1)}_{k_\ell})$ for all $\ell=1,\ldots,N_i$. Then:
\begin{equation}\label{eq:optimality_cuts_prop}
    (\bm h^{(j)}_{i}-\bm T^{(j)}_{i}\bm x)^\top \sum_{\ell=1}^{N_i}\pi_{k_\ell,i}^{(j-1,j)}\bm e^{(j-1)}_{k_\ell}\leq \theta_{i}^{(j)},
\end{equation}
is an optimality cut for Problem \eqref{eq:master_problem_omega_i_j} defined over the set $\Omega_{i}^{(j)}$.
\end{proposition}

\begin{proof}
    Consider the collection of optimality cuts $(\bm h^{(j-1)}_{k_\ell}-\bm T^{(j-1)}_{k_{\ell}}\bm x)^\top \bm e^{(j-1)}_{k_\ell}\leq \theta_{k_{\ell}}^{(j-1)}$, one for each subset $\Omega^{(j-1)}_{k_\ell}$, with $\ell=1,\ldots,N_i$. If $\bm x\in\mathcal{X}$ satisfies all these inequalities, then it also satisfies their convex combination with respect to the inter-level weights, i.e.:
    \begin{equation} \label{prop_proof_eq1}
        \sum_{\ell=1}^{N_i}\pi_{k_\ell,i}^{(j-1,j)}(\bm h^{(j-1)}_{k_\ell}-\bm T^{(j-1)}_{k_{\ell}}\bm x)^\top \bm e^{(j-1)}_{k_\ell}\leq \sum_{\ell=1}^{N_i}\pi_{k_\ell,i}^{(j-1,j)}\theta_{k_\ell}^{(j-1)}.
    \end{equation}
    Here, we have assumed that the components corresponding to scenarios not present in $\Omega_{k_\ell}^{(j-1)}$ are set to zero, so that $\bm h^{(j-1)}_{k_\ell}\in\mathbb{R}^{p\cdot |\Omega_{i}^{(j)}|}$ and $\bm T^{(j-1)}_{k_{\ell}}\in\mathbb{R}^{p\cdot |\Omega_{i}^{(j)}|\times n}$ for all $k_\ell=1,\ldots,N_i$. Therefore, the left hand-side of \eqref{prop_proof_eq1} can be re-written as:
    \begin{equation} \label{prop_proof_eq2}
        \sum_{\ell=1}^{N_i}\pi_{k_\ell,i}^{(j-1,j)}(\bm h^{(j)}_{i}-\bm T^{(j)}_{i}\bm x)^\top \bm e^{(j-1)}_{k_\ell}=(\bm h^{(j)}_{i}-\bm T^{(j)}_{i}\bm x)^\top\sum_{\ell=1}^{N_i}\pi_{k_\ell,i}^{(j-1,j)}\bm e^{(j-1)}_{k_\ell}.
    \end{equation}
    By Proposition \ref{prop:extremepoints_omegaij}, the summation in the right-hand side of \eqref{prop_proof_eq2} defines an extreme point of $\Lambda_{i}^{(j)}$.
    
    It remains to prove the relationship between $\theta_i^{(j)}$ and $\sum_{\ell=1}^{N_i}\pi_{k_\ell,i}^{(j-1,j)}\theta_{k_\ell}^{(j-1)}$, i.e., between the right-hand side of \eqref{eq:optimality_cuts_prop} and \eqref{prop_proof_eq1}, respectively. From \eqref{eq:relation_phi_f_disjoint}, we have that:
    \begin{equation*}
        \sum_{\ell=1}^{N_i}\pi_{k_\ell,i}^{(j-1,j)}\theta_{k_\ell}^{(j-1)}= \sum_{\ell=1}^{N_i} \frac{\pi_{k_\ell}^{(j-1)}}{\pi_i^{(j)}}\theta_{k_\ell}^{(j-1)} = \frac{1}{\pi_{i}^{(j)}}\sum_{\ell=1}^{N_i}\pi_{k_\ell}^{(j-1)}\theta_{k_{\ell}}^{(j-1)}.
    \end{equation*}
    
    Since $\Omega_{k_\ell}^{(j-1)} \subseteq \Omega_i^{(j)}$ for all $\ell = 1, \ldots, N_i$, it follows that $\theta_{k_\ell}^{(j-1)} \leq \theta_i^{(j)}$. Indeed, each $\theta$ represents an upper bound on the optimal value of the recourse problem $\mathcal{Q}(\cdot, \cdot)$, and the feasible region may shrink when moving from $\Omega_{k_\ell}^{(j-1)}$ to $\Omega_i^{(j)}$, as potentially more constraints \eqref{eq:constr1_scenario_subproblem_omega_i_j} are included. Therefore:
    \begin{equation} \label{prop_proof_eq3}
        \frac{1}{\pi_{i}^{(j)}}\sum_{\ell=1}^{N_i}\pi_{k_\ell}^{(j-1)}\theta_{k_{\ell}}^{(j-1)} \leq \frac{1}{\pi_{i}^{(j)}}\sum_{\ell=1}^{N_i}\pi_{k_\ell}^{(j-1)}\theta_{i}^{(j)}=\theta_{i}^{(j)}\frac{1}{\pi_{i}^{(j)}}\sum_{\ell=1}^{N_i}\pi_{k_\ell}^{(j-1)} = \theta_{i}^{(j)},
    \end{equation}
    where the last equality follows from \eqref{eq_relations_intralevelwights_consecutivelevels}. Combining \eqref{prop_proof_eq1}--\eqref{prop_proof_eq2} with \eqref{prop_proof_eq3}, we obtain that:
    \begin{equation*}
        (\bm h^{(j)}_{i}-\bm T^{(j)}_{i}\bm x)^\top\sum_{\ell=1}^{N_i}\pi_{k_\ell,i}^{(j-1,j)}\bm e^{(j-1)}_{k_\ell} \leq \sum_{\ell=1}^{N_i}\pi_{k_\ell,i}^{(j-1,j)}\theta_{k_\ell}^{(j-1)} \leq \theta_i^{(j)}.
    \end{equation*}
    This completes the proof.
\end{proof}

Proposition \ref{prop:optimality_cuts_prop} establishes how to construct an optimality cut at level $j$ from optimality cuts at level $j-1$. The following result further characterizes the relationship between consecutive levels.

\begin{proposition} \label{prop:optimality_cuts_prop_j_jminus1}
Suppose that subsets $\{\Omega_{k_\ell}^{(j-1)}\}_{\ell=1}^{N_i}$ are either disjoint or with fixed scenarios appearing in all of them. Then, the left-hand side of the optimality cut \eqref{eq:optimality_cuts_omega_i_j} at level $j$ can be decomposed into a convex combination of left-hand sides of the corresponding optimality cuts at level $j-1$, with weights $\pi_{k_\ell,i}^{(j-1,j)}$, i.e.:
\begin{equation*}
    (\bm h_i^{(j)}-\bm T_{i}^{(j)} \bm x)^\top \bm e_{i}^{(j)} = \sum_{\ell=1}^{N_i} \pi_{k_\ell,i}^{(j-1,j)} [(\bm h_{k_{\ell}}^{(j-1)}-\bm T_{k_{\ell}}^{(j-1)} \bm x)^\top \bm e_{k_\ell}^{(j-1)}].
\end{equation*}
\end{proposition}

\begin{proof}
Let $\bm e_{i}^{(j)}\in\mathcal{E}(\Lambda_{i}^{(j)})$ be an extreme point at level $j$ of the refinement chain. Thanks to Proposition \ref{prop:extremepoints_omegaij}, $\bm e_{i}^{(j)}$ can be decomposed as a convex combination of extreme points $\{\bm e_{k_\ell}^{(j-1)}\}_{\ell=1}^{N_i}$ at level $j-1$ with weights $\{\pi_{k_\ell,i}^{(j-1,j)}\}_{\ell=1}^{N_i}$. Thus, the left-hand side of the optimality cut \eqref{eq:optimality_cuts_omega_i_j} at level $j$ can be re-written as:
\begin{equation*}
\begin{split}
    (\bm h_i^{(j)}-\bm T_{i}^{(j)}\bm x)^\top \bm e_{i}^{(j)} &= (\bm h_i^{(j)}-\bm T_{i}^{(j)} \bm x)^\top \sum_{\ell=1}^{N_i}\pi_{k_\ell,i}^{(j-1,j)} \bm e_{k_\ell}^{(j-1)} \\
    &= \sum_{\ell=1}^{N_i} \pi_{k_\ell,i}^{(j-1,j)} [(\bm h_i^{(j)}- \bm T_{i}^{(j)} \bm x)^\top \bm e_{k_\ell}^{(j-1)}].
\end{split}
\end{equation*}
Since each $\bm e_{k_{\ell}}^{(j-1)}$ has zero components for all scenarios not contained in $\Omega_{i}^{(j)}$, the scalar product reduces to the terms associated with scenarios in $\Omega_{i}^{(j)}$, while all remaining terms vanish. Thus, the retained terms are the ones related to $\bm h_{k_{\ell}}^{(j-1)}$ and $\bm T_{k_{\ell}}^{(j-1)}$. Thus:
\begin{equation*}
    (\bm h_i^{(j)}-\bm T_{i}^{(j)} \bm x)^\top \bm e_{i}^{(j)} = \sum_{\ell=1}^{N_i} \pi_{k_\ell,i}^{(j-1,j)} [(\bm h_{k_{\ell}}^{(j-1)}-\bm T_{k_{\ell}}^{(j-1)} \bm x)^\top \bm e_{k_\ell}^{(j-1)}].
\end{equation*}
This concludes the proof.
\end{proof}

When considering the extreme rays and the corresponding feasibility cuts across consecutive levels, results analogous to those stated in Propositions \ref{prop:extremepoints_omegaij}--\ref{prop:optimal_extreme_points} and \ref{prop:optimality_cuts_prop}--\ref{prop:optimality_cuts_prop_j_jminus1} do not hold. Indeed, the extreme rays and feasibility cuts remain identical at each refinement level to those of the classical L-shaped method. This is formally established in the following result.

\begin{proposition} \label{prop:feasibility_cuts_prop}
Let $\bm r\in\mathbb{R}^p$ be an extreme ray of $\Lambda^s$, where $s:\bm\omega^s\in\Omega_i^{(j)}$, and let
\eqref{eq:feasibility_cuts} denote the associated set of feasibility cuts. Then, the extreme rays of $\Lambda_{i}^{(j)}$ take the following form: 
\begin{equation}\label{eq_extreme_ray_s_ij}
    \bm r^s:=(\bm b^s\otimes \bm I_p)\bm r,
\end{equation}
where $\bm b^s$ is the $s$-th standard basis vector of $\mathbb{R}^{|\Omega_{i}^{(j)}|}$, $\bm I_p$ is the $p\times p$ identity matrix, and $\otimes$ denotes the Kronecker product. Moreover, the feasibility cuts \eqref{eq:feasibility_cuts_omega_i_j} associated with $\bm r^s$ coincide with those in \eqref{eq:feasibility_cuts}.
\end{proposition}

\begin{proof}
First, observe that for each scenario $\bm\omega^s\in\Omega_i^{(j)}$, the dual feasible regions:
\begin{equation*}
    \Lambda^s=
    \left\{
        \bm\lambda\in\mathbb{R}^p:
        \bm W^\top\bm\lambda\leq\bm q^s
    \right\},
\end{equation*}
and
\begin{equation*}
    (\Lambda_i^{(j)})^s=
    \left\{
        \bm\lambda^s\in\mathbb{R}^p:
        \bm W^\top\bm\lambda^s
        \leq
        p^s_{\Omega_i^{(j)}}\bm q^s,
    \right\}
\end{equation*}
differ only by the positive scalar factor $p^s_{\Omega_i^{(j)}}$ multiplying the right-hand side $\bm q^s$. Therefore, they have the same recession cone, implying that every extreme ray $\bm r$ of $\Lambda^s$ is also an extreme ray of $(\Lambda_i^{(j)})^s$.

Next, by definition of $\bm r^s$ in \eqref{eq_extreme_ray_s_ij}:
\begin{equation*}
\bm r^s
=
\begin{bmatrix}
\bm 0 &
\cdots &
\bm 0 &
\bm r &
\bm 0 &
\cdots &
\bm 0
\end{bmatrix}^{\!\top},
\end{equation*}
that is, the block corresponding to scenario $s$ coincides with $\bm r$, while all remaining blocks are zero. Since $\Lambda_i^{(j)}$ is the Cartesian product of the feasible regions $(\Lambda_i^{(j)})^s$ (see \eqref{eq:feasible_region_dual_omega_i_j}), it follows that $\bm r^s$ is an extreme ray of $\Lambda_i^{(j)}$.

Finally, substituting $\bm r^s$ into the feasibility cut \eqref{eq:feasibility_cuts_omega_i_j}, all terms vanish except the one associated with scenario $\bm\omega^s$. Hence, the resulting inequality coincides exactly with the feasibility cut \eqref{eq:feasibility_cuts}. This completes the proof.
\end{proof}

We conclude this section by establishing the main convergence result for the proposed L-shaped refinement chain cuts method, building on the theoretical developments presented thus far.

\begin{theorem} \label{teo_convergence}
Consider a refinement chain composed by either disjoint subsets or with fixed scenarios appearing in all of them. For each level $j=1,\ldots,J$ of the refinement chain, the Benders master problem \eqref{eq:master_problem_omega_i_j} at level $j$ converges to the optimal solution of the original two-stage stochastic program \eqref{eq:DEP}.
\end{theorem}
\begin{proof}
According to Proposition \ref{prop:multicut}, the classical Benders multi-cut formulation can be regarded as a special case of the proposed framework, corresponding to level 1 of the refinement chain, where all subgroups are disjoint. In the case of fixed scenarios, a lower level composed of disjoint subgroups can also be introduced. By \cite{BirLou1988}, the multi-cut formulation \eqref{eq:master_problem} converges to the optimal solution of the original two-stage stochastic program \eqref{eq:DEP}.

Consider now the Benders master problem at level 2. By Proposition \ref{prop:optimality_cuts_prop_j_jminus1}, the optimality cuts at level 2 can be decomposed as convex combinations of optimality cuts obtained at level 1, i.e., the cuts characterizing the multi-cut formulation. Moreover, Proposition \ref{prop:optimal_extreme_points} establishes a one-to-one correspondence between the sets of optimal dual extreme points associated with levels 1 and 2. Finally, Proposition \ref{prop:feasibility_cuts_prop} ensures that directions of unboundedness are preserved between the two levels, so that feasibility cuts are also guaranteed. Therefore, convergence of the Benders master problem at level 2 follows from convergence at level 1.

The same argument can then be applied iteratively between any two consecutive levels of the refinement chain. Hence, for each level $j=1,\ldots,J$, the Benders master problem \eqref{eq:master_problem_omega_i_j} converges to the optimal solution of \eqref{eq:DEP}.
\end{proof}

Note that the convergence result of Theorem \ref{teo_convergence} can also be established by proceeding in the opposite direction along the refinement chain. In particular, starting from the single-cut formulation at the coarsest refinement level $j=J$, the relationships established between consecutive levels can be applied recursively up to the finest level $j=1$, yielding the same convergence result.

In Appendix \ref{sec_app_B}, we provide a numerical example illustrating the theoretical results developed thus far.

\section{An L-shaped refinement-based solution algorithm} \label{sec_implementation}

In this section, we present an algorithmic implementation of the L-shaped refinement chain cuts method between two consecutive levels $j-1$ and $j$ of the refinement chain, with $j\in\{2,\ldots,J\}$. The proposed approach considers a Benders master problem defined at level $j$ and iteratively strengthens it with feasibility and aggregated optimality cuts generated from the subproblems associated with level $j-1$. The overall procedure is summarized in Algorithm \ref{algo:refchaincuts_algo}.

As input, we consider a two-stage stochastic program of the form \eqref{eq:DEP} together with a refinement chain \eqref{eq:refinement_chain_omega} based on either disjoint or fixed scenarios. As in the classical L-shaped method, the explicit enumeration of all extreme points and extreme rays of the dual feasible regions $\Lambda_{i}^{(j)}$, for all $i=1,\ldots,m_j$, is computationally intractable. Therefore, the proposed methodology relies on an iterative relaxation procedure \cite{Geoffrion1970_b,Geoffrion1970_a}. For this reason, we introduce an absolute tolerance $\varepsilon_a$, a relative tolerance $\varepsilon_r$, and a maximum CPU time limit $T_{\max}$ for the stopping criteria. At termination, provided that the time limit has not been reached, the algorithm returns a solution that is optimal up to an absolute gap $\varepsilon_a$ or a relative gap $\varepsilon_r$, as guaranteed by Proposition \ref{prop:convergence_algorithm} at the end of this section. To evaluate the convergence of the procedure, we iteratively update a Lower Bound ($LB$) and an Upper Bound ($UB$) on the optimal objective function value of the two-stage stochastic program. These bounds are then used to compute the absolute and relative optimality gaps that define the stopping criteria.

The algorithm starts by initializing $LB$ and $UB$ to $-\infty$ and $+\infty$, respectively (line \ref{alg_ln1}). Similarly, the absolute and relative gaps, denoted by $gap_a$ and $gap_r$, are initialized to $+\infty$ (line \ref{alg_ln2}). The procedure then enters an iterative loop (lines \ref{alg_ln3}--\ref{alg_ln25}), which constitutes the core iterative structure of the algorithm and continues until one of the stopping criteria is satisfied. At each iteration, the Benders master problem \eqref{eq:DEP_ref_omega_i_j} defined at level $j$ is solved, yielding a first-stage solution $\bm x^{(j)}$ and an objective function value $z^{(j)}$ (line \ref{alg_ln4}). The lower bound is then updated in line \ref{alg_ln5}. Subsequently, the first-stage cost is computed as $obj_1=\bm c^\top \bm x^{(j)}$, while the second-stage cost $obj_2$ is initialized to zero (line \ref{alg_ln6}).

\begin{algorithm}[H] 
\resizebox{\textwidth}{!}{%
\begin{minipage}{\textwidth}
\caption{Algorithmic implementation of the L-shaped refinement chain cuts method between consecutive levels $j-1$ and $j$} \label{algo:refchaincuts_algo}
\textbf{Input:} Two-stage stochastic program \eqref{eq:DEP}, refinement chain \eqref{eq:refinement_chain_omega},\\ consecutive refinement levels $j-1$ and $j$, absolute tolerance $\varepsilon_{a}$, \\ relative tolerance $\varepsilon_{r}$, time limit $T_{\max}$\newline
\textbf{Output:} An optimal solution of the two-stage stochastic program
\begin{algorithmic}[1]
\State $LB \gets -\infty$, $UB \gets +\infty$ \label{alg_ln1} \Comment{\footnotesize Initialize lower and upper bounds} \normalsize
\State $gap_a \gets +\infty$, $gap_r \gets +\infty$ \label{alg_ln2} \Comment{\footnotesize Initialize absolute and relative gaps} \normalsize

\While{($gap_a>\varepsilon_a$ or $gap_r>\varepsilon_r$) and CPU time $<T_{\max}$} \label{alg_ln3}
\State Solve Benders master problem \eqref{eq:DEP_ref_omega_i_j} at level $j$, with first-stage solution $\bm x^{(j)}$ and \label{alg_ln4}
\Statex \hspace*{0.4cm} objective function value $z^{(j)}$
\State $LB \gets \max\{LB,z^{(j)}\}$ \label{alg_ln5} \Comment{\footnotesize Update the lower bound} \normalsize
\State $obj_1 \gets \bm c^\top \bm x^{(j)}$, $obj_2 \gets 0$ \label{alg_ln6} \Comment{\footnotesize First-stage and second-stage costs} \normalsize
\For{$k=1,\ldots,m_{j-1}$} \label{alg_ln7} \Comment{\footnotesize Loop at level $j-1$} \normalsize
\State Solve dual subproblem \eqref{eq:scenario_subproblem_DUAL_omega_i_j} $\mathcal{Q}_{dual}({\bm x}^{(j)},\bm \xi_{k}^{(j-1)})$ at level $j-1$, with objective \label{alg_ln8}
\Statex \hspace*{0.95cm} function value $obj_k^{(j-1)}$
\If{$\mathcal{Q}_{dual}({\bm x}^{(j)},\bm \xi_{k}^{(j-1)})$ is unbounded} \label{alg_ln9}
\State Obtain a dual extreme ray $\bm r_k^{(j-1)}\in\mathcal{R}(\Lambda_{k}^{(j-1)})$ \label{alg_ln10}
\State Generate a feasibility cut \eqref{eq:feasibility_cuts_omega_i_j} $(\bm h_{k}^{(j-1)}-\bm T_{k}^{(j-1)}\bm x)^\top \bm r_k^{(j-1)} \leq 0$ \label{alg_ln11}
\Else \label{alg_ln12}
\State Obtain a dual extreme point $\bm e_k^{(j-1)}\in\mathcal{E}(\Lambda_{k}^{(j-1)})$ \label{alg_ln13}
\State Generate an optimality cut \eqref{eq:optimality_cuts_omega_i_j} $(\bm h_{k}^{(j-1)}-\bm T_{k}^{(j-1)}\bm x)^\top \bm e_k^{(j-1)} \leq \theta_{k}^{(j-1)}$ \label{alg_ln14}
\EndIf \label{alg_ln15}
\State $obj_2 \gets obj_2+\pi_k^{(j-1)}obj_k^{(j-1)}$ \label{alg_ln16}
\EndFor \label{alg_ln17}
\State $UB \gets \min\{UB,obj_1+obj_2\}$ \label{alg_ln18} \Comment{\footnotesize Update the upper bound} \normalsize
\State Add the feasibility cuts to the Benders master problem \eqref{eq:DEP_ref_omega_i_j} at level $j$ \label{alg_ln19}
\For{$i=1,\ldots,m_j$} \label{alg_ln20} \Comment{\footnotesize Loop at level $j$} \normalsize
\State Using \eqref{eq:optimality_cuts_prop} aggregate the optimality cuts for each $k=1,\ldots,m_{j-1}$ such that \label{alg_ln21}
\Statex \hspace*{0.95cm} $\Omega_{i}^{(j)} \supseteq \Omega_{k}^{(j-1)}$ and $\mathcal{Q}_{dual}({\bm x}^{(j)},\bm \xi_{k}^{(j-1)})$ is feasible 
\State Add the aggregated optimality cuts to the Benders master problem \eqref{eq:DEP_ref_omega_i_j} at \label{alg_ln22}
\Statex \hspace*{0.95cm} level $j$ for subgroup $\Omega_i^{(j)}$
\EndFor \label{alg_ln23}
\State $gap_a \gets |UB-LB|$, $gap_r = gap_a/|UB|$ \label{alg_ln24} \Comment{\footnotesize Update the absolute and relative gaps} \normalsize
\EndWhile \label{alg_ln25}
\end{algorithmic}
\end{minipage}
}
\end{algorithm}

The algorithm then performs two consecutive iterative procedures. The first loop (lines \ref{alg_ln7}--\ref{alg_ln17}) operates at refinement level $j-1$. For each subset $\Omega_k^{(j-1)}$, with $k=1,\ldots,m_{j-1}$, the dual subproblem \eqref{eq:scenario_subproblem_DUAL_omega_i_j} $\mathcal{Q}_{dual}(\bm x^{(j)},\bm \xi_k^{(j-1)})$ is solved, obtaining an objective function value equal to $obj_k^{(j-1)}$ (line \ref{alg_ln8}). If the dual subproblem is unbounded, a dual extreme ray $\bm r_k^{(j-1)} \in \mathcal{R}(\Lambda_k^{(j-1)})$ is generated and a feasibility cut of the form \eqref{eq:feasibility_cuts_omega_i_j} is constructed (lines \ref{alg_ln9}--\ref{alg_ln11}). Otherwise, an extreme point $\bm e_k^{(j-1)} \in \mathcal{E}(\Lambda_k^{(j-1)})$ is obtained and an optimality cut of the form \eqref{eq:optimality_cuts_omega_i_j} is generated (lines \ref{alg_ln12}--\ref{alg_ln15}). The second-stage contribution is then updated using the intra-level weight $\pi_k^{(j-1)}$ associated with subset $\Omega_k^{(j-1)}$ (line \ref{alg_ln16}). In the unbounded case, the contribution is interpreted as $obj_k^{(j-1)}=+\infty$. At the end of this loop, the upper bound UB is updated as shown in line \ref{alg_ln18}. The generated feasibility cuts are then added to the Benders master problem at level $j$ (line \ref{alg_ln19}).

The second loop (lines \ref{alg_ln20}--\ref{alg_ln23}) establishes the connection between optimality cuts at levels $j-1$ and $j$ of the refinement chain. For each subset $\Omega_i^{(j)}$, with $i=1,\ldots,m_j$, the optimality cuts generated at level $j-1$ are aggregated using Proposition \ref{prop:optimality_cuts_prop}. More precisely, the aggregation is performed over all subsets $\Omega_k^{(j-1)}$ such that $\Omega_{i}^{(j)} \supseteq \Omega_{k}^{(j-1)}$, provided that the corresponding dual subproblem $\mathcal{Q}_{dual}(\bm x^{(j)},\bm \xi_k^{(j-1)})$ is feasible. The resulting aggregated optimality cuts are then added to the Benders master problem at level $j$ (lines \ref{alg_ln21}--\ref{alg_ln22}). Finally, the absolute and relative gaps are updated as reported in line \ref{alg_ln24}.

Schematically, the proposed methodology is illustrated in Figure \ref{fig:refchaincuts_flowchart}. The finite convergence of Algorithm \ref{algo:refchaincuts_algo} to an optimal solution of the two-stage stochastic program \eqref{eq:DEP} is established in the following proposition.

\begin{figure}[t!]
\centering

\resizebox{\textwidth}{!}{%

\begin{tikzpicture}[
    node distance=1.2cm and 1.7cm,
    >=Latex,
    box/.style={
        draw,
        align=center,
        minimum width=4.8cm,
        minimum height=1.1cm,
        font=\large
    },
    roundbox/.style={
        box,
        rounded corners=4pt,
        minimum width=3.4cm
    },
    decision/.style={
        draw,
        diamond,
        aspect=2.2,
        align=center,
        inner sep=1pt,
        font=\large
    },
    every path/.style={thick}
]

\node[roundbox] (start) {Start};

\node[box, below=0.6cm of start] (box1) {
Solve Benders master\\
problem \eqref{eq:DEP_ref_omega_i_j} at level $j$
};

\node[box, below=1.2cm of box1] (box2) {
Solve dual subproblem \eqref{eq:scenario_subproblem_DUAL_omega_i_j}\\
$\mathcal{Q}_{dual}(\bm x^{(j)},\bm \xi_k^{(j)})$ at level $j-1$\\
for each $k=1,\ldots,m_{j-1}$
};

\node[box, below=0.6cm of box2] (box3) {
Update $LB$ and $UB$ for\\
the stopping criteria
};

\node[box, below=0.6cm of box3] (box4) {
Generate the optimality and\\
feasibility cuts at level $j-1$
};

\node[box, below=0.6cm of box4] (box5) {
Aggregate the optimality cuts from\\
level $j-1$ to level $j$ using \eqref{eq:optimality_cuts_prop}
};

\node[decision, below=0.6cm of box5] (test) {
Stopping criteria\\
satisfied?
};

\node[roundbox, right=1.3cm of test, minimum width=5.0cm] (report) {
Report optimal\\
solution $\bm x^{(j)},z^{(j)}$
};

\node[box, left=1.2cm of box3] (box6) {
Add the feasibility and aggregated\\
optimality cuts to the Benders\\
master problem \eqref{eq:DEP_ref_omega_i_j} at level $j$
};

\draw[->] (start) -- (box1);

\draw[->] (box1) -- node[right, align=left] {
optimal first-stage $\bm x^{(j)}$\\
optimal objective value $z^{(j)}$
} (box2);

\draw[->] (box2) -- (box3);
\draw[->] (box3) -- (box4);
\draw[->] (box4) -- (box5);
\draw[->] (box5) -- (test);

\draw[->] (test) -- node[midway, above] {Yes} (report.west);

\draw[->] (test.west)
    -| node[pos=0.25, above] {No}
    (box6.south);

\draw[->] (box6.north) |- (box1.west);

\end{tikzpicture}

} 

\caption{Overview of the proposed refinement-based solution Algorithm \ref{algo:refchaincuts_algo}.} \label{fig:refchaincuts_flowchart}

\end{figure}

\begin{proposition} \label{prop:convergence_algorithm}
If $\varepsilon_a=\varepsilon_r=0$ and $T_{\max}=\infty$, Algorithm \ref{algo:refchaincuts_algo} terminates after finitely many iterations and returns an optimal solution of the original two-stage stochastic program \eqref{eq:DEP}.
\end{proposition}

\begin{proof}
Consider the Benders master problem \eqref{eq:DEP_ref_omega_i_j} at level $j$ of the refinement chain. At each iteration, Algorithm \ref{algo:refchaincuts_algo} solves a relaxation of this master problem and evaluates the corresponding dual recourse subproblems at level $j-1$. By Proposition \ref{prop:feasibility_cuts_prop}, every feasibility cut generated from a dual extreme ray at level $j-1$ is a valid feasibility cut for the original two-stage stochastic program and can therefore be added to the master problem at level $j$ without excluding any optimal solution.

Similarly, Proposition \ref{prop:optimality_cuts_prop} shows that the optimality cuts generated at level $j-1$ can be aggregated into valid optimality cuts for the master problem at level $j$. Moreover, Proposition \ref{prop:optimal_extreme_points} establishes a one-to-one correspondence between the optimal extreme points defining the recourse function at level $j$ and collections of optimal extreme points at level $j-1$. Consequently, every optimality cut required to characterize the recourse function at level $j$ can be generated through the aggregation procedure described in lines \ref{alg_ln20}--\ref{alg_ln23} of Algorithm \ref{algo:refchaincuts_algo}. Since each dual feasible region $\Lambda_k^{(j-1)}$, $k=1,\ldots,m_{j-1}$, is a polyhedron, it has finitely many extreme points and extreme rays. Therefore, only finitely many distinct feasibility and optimality cuts can be generated throughout the execution of the algorithm.

Assume by contradiction that Algorithm \ref{algo:refchaincuts_algo} does not terminate in a finite number of iterations. Since only finitely many distinct cuts exist, there must be an iteration after which no new feasibility or optimality cut is generated. At that point, all feasibility cuts and all optimality cuts necessary to describe the Benders master problem at level $j$ have already been added. Hence, the current master problem coincides with the master problem associated with level $j$.

By Theorem \ref{teo_convergence}, the Benders master problem at level $j$ converges to the optimal solution of the original two-stage stochastic program \eqref{eq:DEP}. Therefore, once all required cuts have been generated, the lower bound $LB$ provided by the master problem coincides with the upper bound $UB$ obtained by evaluating the recourse function at the current first-stage solution.

Since $\varepsilon_a=\varepsilon_r=0$, the stopping criteria are satisfied. Moreover, because $T_{\max}=\infty$, the algorithm cannot terminate due to the time limit. This contradicts the assumption that the algorithm continues indefinitely. Therefore, Algorithm \ref{algo:refchaincuts_algo} terminates after finitely many iterations and returns an optimal solution of the original two-stage stochastic program \eqref{eq:DEP}.
\end{proof}

\section{Numerical results} \label{sec_numericalresults}

In this section, we report and discuss the numerical results of the computational experiments conducted to validate the proposed methodology. To assess the effectiveness of the L-shaped refinement chain cuts approach, we apply it to a two-stage stochastic fixed-charge multicommodity network design problem. This class of problems commonly arises in several application domains, particularly in transportation and logistics, such as liner shipping \cite{ChrHelPisSacVil2020}, freight and passenger rail transportation \cite{ChoCrai2020}, logistics network design \cite{MelNicSald2009}, but also in manufacturing and telecommunications \cite{CraGenAkh2022}.

We begin in Section \ref{subsec_problemdefinition} by introducing the problem and describing its main characteristics, followed by its mathematical formulation in Section \ref{subsec_mathematical_formulation}. In Section \ref{subsec_experimental_setting} we present the details of the experimental setting considered in this study. Finally, Section \ref{subsec_results_discussion} reports and discusses the computational results.

\subsection{Problem definition} \label{subsec_problemdefinition}

We consider a directed graph $\mathcal{G}:=(\mathcal{N},\mathcal{A})$, where $\mathcal{N}$ and $\mathcal{A}$ are the set of nodes and arcs, respectively, and $\mathcal{K}$ is the set of commodities. For each node $i\in\mathcal{N}$, we define $\mathcal{N}^{+}(i)$ and $\mathcal{N}^{-}(i)$ the set of arcs  emanating from $i$ and terminating at $i$, respectively. Each commodity $k\in\mathcal{K}$ needs to be transported from an origin node $on^k\in\mathcal{N}$ to a destination node $dn^k\in\mathcal{N}$.

As in \cite{HewReiWall2021}, we suppose that the decision-maker can externally access additional resources beyond the design network. In this case, external providers are employed to transport some commodities directly from their origin to destination nodes, but at a higher unit price. No capacity constraints are imposed on these \textit{dummy} direct links. Formally, let $\mathcal{A}^A$ be the set of design arcs, and $\mathcal{A}^D$ be the set of dummy arcs associated with the commodities. For each commodity $k\in\mathcal{K}$, set $\mathcal{A}^D$ includes the direct arc from $on^k$ to $dn^k$. A fixed cost $f_{ij}\in\mathbb{R}^+$ is associated with the selection of arc $(i,j)\in\mathcal{A}^A$ in the network, while $c^k_{ij},\widetilde{c}^k_{ij}\in\mathbb{R}^+$ are the costs per unit of volume of commodity $k\in\mathcal{K}$ flowed on arc $(i,j)\in\mathcal{A}^A$ and $(i,j)\in\mathcal{A}^D$, respectively.

We assume that the capacity $u_{ij}(\omega)\in\mathbb{R}^+$ of each arc $(i,j)\in\mathcal{A}^A$ and the demand $d^k_{i}(\omega)\in\mathbb{R}^+$ of commodity $k\in\mathcal{K}$ at node $i\in\mathcal{N}$ are stochastic parameters, depending on the outcome $\omega\in\Omega$ of a random experiment. Collecting these parameters into the vectors $\bm u(\omega)$ and $\bm d(\omega)$, we denote the uncertain data by $\bm \xi(\omega)=\{\bm u(\omega),\bm d(\omega)\}$.
Let $v^{k}(\omega)$ be the volume of commodity $k\in\mathcal{K}$ that must be transported from $on^k$ to $dn^k$ once the outcome $\omega\in\Omega$ is observed. Then, the stochastic demand can be expressed as:
\begin{equation} \label{ed:stoch_network_demand}
    d^{k}_i(\omega):= \begin{cases}
        v^{k}(\omega)& \text{if } i=on^k  \\
         -v^{k}(\omega)& \text{if } i=dn^k  \\ 
         0 & \text{otherwise.}
    \end{cases}
\end{equation}

Let $x_{ij}\in\{0,1\}$ be the first-stage variable associated with arc $(i,j)\in\mathcal{A}^A$, indicating whether arc $(i,j)$ is installed in the network. In the second stage of the problem, once the uncertainty has been revealed, the network constructed at the first stage is used to flow the commodities to meet the observed demands. Let $y_{ij}^{k}(\omega)\geq 0$ be the second-stage volume of commodity $k\in\mathcal{K}$ that transits through arc $(i,j)\in\mathcal{A}^A$ if the outcome $\omega\in\Omega$ is observed. Similarly for $\widetilde{y}_{ij}^{k}(\omega)\geq 0$ for the dummy arc $(i,j)\in\mathcal{A}^D$.

The objective is to minimize the expected total cost, which includes both the fixed costs associated with the selection of the design arcs in the first stage and the operational transportation costs incurred in the second stage for routing the commodities through the designed network and, when necessary, through the external dummy arcs. To account for risk aversion, we consider a mean-risk formulation in which the second-stage cost is modeled as a convex combination of expectation and Conditional Value-at-Risk (CVaR), also known as Average Value-at-Risk (AVaR) (see, e.g., \cite{MAHMUTOGULLARI2018595,Schultz2011,Shapiro2011}).  An illustrative example of the proposed network design problem is provided in Appendix \ref{sec_app_C1}.

\newpage
\subsection{Mathematical formulation} \label{subsec_mathematical_formulation}

Extensively, the two-stage stochastic fixed-charge multicommodity network design problem can be formulated as follows:

\small
\begin{subequations} \label{eq:network_meanrisk_pb_DEP_approach1}
\begin{align}
    \min_{\bm x,\bm y, \bm{\widetilde{y}}} & \quad \sum_{(i,j)\in\mathcal{A}^A}f_{ij}x_{ij}+ (1-\beta)\text{ }\mathbb{E}_{\mathbb{P}}\bigg[\sum_{k\in\mathcal{K}} \sum_{(i,j)\in\mathcal{A}^A}c_{ij}^ky_{ij}^k(\omega)+\sum_{k\in\mathcal{K}} \sum_{(i,j)\in\mathcal{A}^D}\widetilde{c}_{ij}^k\widetilde{y}_{ij}^k(\omega) \bigg]+\nonumber\\
    & +\beta\text{ }\text{CVaR}_\alpha\bigg[\sum_{k\in\mathcal{K}} \sum_{(i,j)\in\mathcal{A}^A}c_{ij}^ky_{ij}^k(\omega)+\sum_{k\in\mathcal{K}} \sum_{(i,j)\in\mathcal{A}^D}\widetilde{c}_{ij}^k\widetilde{y}_{ij}^k(\omega) \bigg] \label{eq:network_meanrisk_objfunc_DEP_approach1}\\
    \text{s.t.} & \quad \sum_{j \in \mathcal{N}^+(i)}\left(y_{ij}^k(\omega)+\widetilde{y}_{ij}^k(\omega)\right)-\sum_{j \in \mathcal{N}^-(i)}\left(y_{ij}^k(\omega)+\widetilde{y}_{ij}^k(\omega)\right)=d_{i}^k(\omega) \hspace*{0.2cm} \forall i \in \mathcal{N},k \in \mathcal{K},\omega\in\Omega \label{eq:network_meanrisk_constr1_DEP_approach1}   \\
    & \quad \sum_{k\in\mathcal{K}} y_{ij}^k(\omega) \leq u_{ij}(\omega)x_{ij} \qquad \forall (i,j)\in\mathcal{A}^A,\omega\in\Omega \label{eq:network_meanrisk_constr2_DEP_approach1} \\
    & \quad x_{ij}\in\{0,1\} \qquad  \forall (i,j)\in\mathcal{A}^A \label{eq:network_meanrisk_constr3_DEP_approach1} \\
    & \quad y_{ij}^k(\omega) \geq 0 \qquad \forall (i,j)\in\mathcal{A}^A,k\in\mathcal{K},\omega\in\Omega \label{eq:network_meanrisk_constr4_DEP_approach1}\\
    & \quad \widetilde{y}_{ij}^k(\omega) \geq 0 \qquad \forall (i,j)\in\mathcal{A}^D,k\in\mathcal{K},\omega\in\Omega. \label{eq:network_meanrisk_constr5_DEP_approach1}
\end{align}
\end{subequations}

\normalsize
The objective function \eqref{eq:network_meanrisk_objfunc_DEP_approach1} aims to minimize the expected total cost, which consists of three components. The first term represents the fixed cost associated with the selection of the design arcs in the first stage. The second term corresponds to the expected operational cost, comprising the multicommodity flow costs over both the selected arcs and the dummy arcs, respectively. The third term accounts for the Conditional Value-at-Risk at confidence level $\alpha\in(0,1)$ of the same operational cost presented in the second term. The expectation and the CVaR are therefore defined consistently with respect to the same operational cost function, ensuring the coherence of the resulting mean-risk measure formulation \cite{DenRus2024}. The parameter $\beta\in[0,1]$ controls the trade-off between the expectation and the CVaR components of the objective function. Larger values of $\alpha$ or $\beta$ correspond to a more risk-averse decision-maker \cite{MAHMUTOGULLARI2018595}.

Constraints \eqref{eq:network_meanrisk_constr1_DEP_approach1} enforce flow conservation, ensuring that the demand \eqref{ed:stoch_network_demand} of each commodity $k\in\mathcal{K}$ is routed from its origin node $on^k$ to its destination node $dn^k$. Constraints \eqref{eq:network_meanrisk_constr2_DEP_approach1} restrict the total quantity of flow that can transit through arc $(i,j)\in\mathcal{A}^A$, without violating its capacity. Finally, constraints \eqref{eq:network_meanrisk_constr3_DEP_approach1}-\eqref{eq:network_meanrisk_constr5_DEP_approach1} define the domain of the decision variables of the problem.

In the case of a discrete set of scenarios $\mathcal{S}$, the objective function \eqref{eq:network_meanrisk_objfunc_DEP_approach1} can be expressed as a linear function \cite{RocUry2000}. The resulting formulation, together with its Benders reformulations, is reported in Appendix \ref{sec_app_C2}. Since recourse actions can always be performed through the set of dummy arcs, the problem remains feasible for every realization of the uncertainty. Consequently, only Benders optimality cuts need to be generated.

\subsection{Experimental setting} \label{subsec_experimental_setting}

We tested the proposed L-shaped refinement chain cuts method on a selection of benchmark instances from the multicommodity network design literature. In particular, we considered the Canad R-instances introduced in \cite{CRAINIC200173} and publicly available at \url{https://commalab.di.unipi.it/datasets/mmcf/#Canad}. We selected three classes of instances: R04 (small), R05 (medium), and R07 (large). The main characteristics of these instances are summarized in Table \ref{tab:instances}.

\begin{table}[h!]
\centering
\caption{Canad R-instances considered in the computational experiments.}
\label{tab:instances}

\begin{tabular}{@{}lllll@{}}
\toprule
Name & Size & Nodes $\lvert \mathcal{N} \rvert$ & Arcs $\lvert \mathcal{A}\rvert$ & Commodities $\lvert \mathcal{K}\rvert$\\
\midrule
R04 & small  & 10 & 60 & 10\\
R05 & medium & 10 & 60 & 25\\
R07 & large  & 10 & 82 & 10\\
\bottomrule
\end{tabular}

\end{table}

The scenarios associated with uncertain arc capacities and commodity demands were generated according to the methodology proposed in \cite{hoyland2003heuristic}. For the small and medium instances (R04 and R05), we considered scenario trees composed of 256 scenarios. For the large instance (R07), we considered scenario trees with 256, 1024, and 2048 scenarios.

In the construction of the refinement chain \eqref{eq:refinement_chain_omega}, we considered three different approaches. The first approach is based on a disjoint sequential partition, in which scenarios are grouped sequentially (e.g., for pairs of scenarios, groups are composed of scenarios 1--2, 3--4, and so on). The second and third approaches are based on the optimal grouping methodology proposed in \cite{MagCavFac2026}. In this case, we considered both disjoint partitions and fixed-scenario groupings, where each group is constructed with a fixed reference scenario, identified by the optimal grouping procedure itself.

Regarding the parameters of the objective function, we considered $\beta=0.5$, i.e., mean-risk formulation with equal weights assigned to expectation and CVaR, and $\alpha=0.95$ as the confidence level of the CVaR. For the small and medium instances, we imposed a time limit of 86400s (24 hours). For the large instance, preliminary experiments showed that neither the multi-cut method, the single-cut method, nor the proposed refinement chain cuts method were able to reach optimality within this limit. Therefore, we reduced the time limit to 10800s (3 hours) and focused the analysis on the final relative optimality gap. In Algorithm \ref{algo:refchaincuts_algo}, the absolute and relative tolerances were set to $\varepsilon_a=\varepsilon_r=10^{-6}$.

Both the implementation for a fixed refinement level and Algorithm \ref{algo:refchaincuts_algo} between consecutive refinement levels were developed in Python 3.12.2 using Gurobi Optimizer v13.0.1 with a MIPGap equal to $10^{-5}$. All experiments were performed on the Galileo100 HPC cluster at CINECA using 1 node, 8 CPUs per task, and 64 GB of RAM.

\subsection{Results and discussion} \label{subsec_results_discussion}

\begin{table}[htbp]

\centering
\caption{Summary of the best computational results, obtained by solving model \eqref{eq:master_problem_omega_i_j} at different levels of the refinement chain.}
\label{tab:R04R05R07_summary}
\setlength{\tabcolsep}{5pt}
\begin{tabular}{llllllll}
\toprule
\multirow{2}{*}{Instance} & \multirow{2}{*}{$|\mathcal{S}|$} & Grouping & Group & Time & Time & Relative & Gap\\
& & configuration & size & (s) & red. (\%) & gap & red. (\%)
\\
\midrule
\multirow{5}{*}{R04} & \multirow{5}{*}{256} & Multi-cut & 1 & 876.38 & 97.41 & $<10^{-6}$ & -- \\
&& Disjoint seq. & 4 & 904.86 & 97.33 & $<10^{-6}$ & -- \\
&& Disjoint opt. & 2 & \textbf{\underline{505.35}} & \textbf{\underline{98.51}} & $<10^{-6}$ & -- \\
&& Fixed-scenario opt. & 4 & 1053.57 & 96.89 & $<10^{-6}$ & -- \\
&& Single-cut & 256 & 33880.20 & 0.00 & $<10^{-6}$ & -- \\
\midrule
\multirow{5}{*}{R05}& \multirow{5}{*}{256} & Multi-cut & 1 & 3268.19 & 68.02 & $<10^{-6}$ & -- \\
&& Disjoint seq. & 32 & 3807.63 & 62.74 & $<10^{-6}$ & -- \\
&& Disjoint opt. & 32 & \underline{3232.68} & \underline{68.37} & $<10^{-6}$ & -- \\
&& Fixed-scenario opt. & 16 & \textbf{\underline{3210.15}} & \textbf{\underline{68.59}} & $<10^{-6}$ & -- \\
&& Single-cut & 256 & 10220.09 & 0.00 & $<10^{-6}$ & -- \\
\midrule
\multirow{10}{*}{R07} & \multirow{4}{*}{256} & Multi-cut & 1 & 10800.00* & -- &  $1.81 \times 10^{-2}$ & 28.16 \\
&& Disjoint opt. & 2 & 10800.00* & -- &  $\mathbf{\underline{8.60 \times 10^{-3}}}$ & $\mathbf{\underline{65.78}}$ \\
&& Fixed-scenario opt. & 4 & 10800.00* & -- & $\underline{1.42 \times 10^{-2}}$ & \underline{43.56} \\
&& Single-cut & 256 & 10800.00* & -- &  $2.51 \times 10^{-2}$ & 0.00 \\
\cmidrule{2-8}
& \multirow{3}{*}{1024} & Multi-cut & 1 & 10800.00* & -- &  $1.13 \times 10^{-2}$ & 75.57 \\
&& Disjoint opt. & 2 & 10800.00* & -- &  $\mathbf{\underline{9.69 \times 10^{-3}}}$ & $\mathbf{\underline{79.10}}$ \\
&& Single-cut & 1024 & 10800.00* & -- &  $4.64 \times 10^{-2}$ & 0.00 \\
\cmidrule{2-8}
& \multirow{3}{*}{2048} & Multi-cut & 1 & 10800.00* & -- &  $1.44 \times 10^{-2}$ & 72.57 \\
&& Disjoint opt. & 16 & 10800.00* & -- &  $\mathbf{\underline{1.18 \times 10^{-2}}}$ & $\mathbf{\underline{77.53}}$ \\
&& Single-cut & 2048 & 10800.00* & -- &  $5.25 \times 10^{-2}$ & 0.00 \\
\bottomrule
\end{tabular}
\end{table}

Table \ref{tab:R04R05R07_summary} reports the best computational results obtained when solving model \eqref{eq:master_problem_omega_i_j} at a fixed refinement level. The reductions are computed with respect to the single-cut formulation. For each instance, the best reduction is reported in bold, while results that also outperform the multi-cut formulation are underlined. The asterisk indicates that the corresponding time limit has been reached. Detailed results with all the grouping configurations are reported in Appendix \ref{sec_app_C3} (see Tables \ref{tab:R04R05_beta05}--\ref{tab:R07_10242048}).

For the R04 and R05 instances, all approaches achieve the optimal solution within the time limit. Performance can therefore be compared in terms of computational time. In both cases, intermediate levels of the refinement chain significantly outperform the single-cut formulation. For the R04 instance, the best result is obtained by the optimal disjoint grouping with group size 2, yielding a time reduction of 98.51\% with respect to the single-cut formulation. Moreover, this configuration also outperforms the multi-cut formulation, indicating that a limited degree of scenario aggregation can provide a more favorable balance between the number of generated cuts and the complexity of the resulting master problem. A similar behavior is observed for the R05 instance, where the best performance is achieved by the fixed-scenario optimal grouping with group size 16, corresponding to a reduction of 68.59\%.

The results for the large R07 instance confirm the trends observed for the R04 and R05 instances. In this case, since all runs reached the prescribed time limit, performance is assessed in terms of the final relative gap, computed as in line \ref{alg_ln24} of Algorithm \ref{algo:refchaincuts_algo}, where the lower bound and upper bound are evaluated for the corresponding fixed level $j$ of the refinement chain. For the instance with $|\mathcal{S}|=256$, the best result is obtained by the optimal disjoint grouping with group size 2, which reduces the final relative gap by 65.78\% with respect to the single-cut formulation. Moreover, this configuration also outperforms the multi-cut formulation, whose gap reduction is limited to 28.16\%.

To further assess the effectiveness of the proposed approach on larger-scale instances, we increased the number of scenarios to $|\mathcal S|=1024$ and $|\mathcal S|=2048$ when using disjoint optimal partitions. The results reported in Table \ref{tab:R04R05R07_summary} show that the benefits of scenario aggregation remain significant as the scenario set grows. In particular, the best-performing configurations achieve gap reductions of 79.10\% and 77.53\% for $|\mathcal{S}|=1024$ and $|\mathcal{S}|=2048$, respectively. In both cases, the best refinement level also outperforms the multi-cut formulation, indicating that the proposed approach remains effective even when the number of scenarios becomes very large.

A more detailed analysis of the results reported in Tables \ref{tab:R07_beta05} and \ref{tab:R07_10242048} reveals a consistent pattern across all considered scenario set sizes. The results highlight the existence of an effective intermediate level of aggregation. Moderately aggregated partitions consistently yield substantial reductions of the final relative gap with respect to the single-cut formulation, whereas performance deteriorates as the group size increases. In particular, very coarse aggregations provide only limited improvements and, in some cases, approach the behavior of the single-cut formulation. This trend suggests that excessive scenario aggregation weakens the approximation of the recourse function, leading to less informative Benders cuts and consequently slower convergence.

A comparison between sequential and optimal grouping strategies shows that the benefits of optimization-driven partitions are strongly instance dependent. For the R04 instance, the optimal grouping strategy provides the best performance at group sizes 2 and 128, yielding the largest time reductions among all tested configurations. It also outperforms sequential grouping at group size 8, although the improvement is less pronounced. For the R05 instance, optimal grouping is particularly effective for intermediate aggregation levels and provides the best results for group sizes 4, 16, and 32. However, the superiority of optimal grouping is not systematic, as sequential grouping remains competitive and occasionally superior for other aggregation levels. These observations indicate that the effectiveness of a partition depends on the interplay between the scenario structure, the aggregation scheme, and the resulting strength of the generated Benders cuts.

\begin{table}[htbp]

\centering
\caption{Summary of the best computational results, obtained by applying Algorithm \ref{algo:refchaincuts_algo} between consecutive levels of the refinement chain.}
\label{tab:R04R05R07_ALGO_summary}
\setlength{\tabcolsep}{4pt}
\begin{tabular}{llllllll}
\toprule
\multirow{2}{*}{Instance} & \multirow{2}{*}{$|\mathcal{S}|$} & Grouping & Consecutive & Time & Time & Relative & Gap\\
& & configuration & group size & (s) & red. (\%) & gap & red. (\%)
\\
\midrule
\multirow{2}{*}{R04} & \multirow{2}{*}{256} & Disjoint seq. & 1--2 & \textbf{\underline{407.83}} & $\textbf{\underline{98.61}}$ & $<10^{-6}$ & -- \\
& & Disjoint opt. & 1--2 & \underline{645.71} & \underline{98.09} & $<10^{-6}$ & -- \\
\midrule
\multirow{2}{*}{R05} & \multirow{2}{*}{256} & Disjoint seq. & 16--32 & \textbf{3854.77} & \textbf{62.28} & $<10^{-6}$ & -- \\
& & Disjoint opt. & 8--16 & 5595.39 & 45.25 & $<10^{-6}$ & -- \\
\midrule
\multirow{4}{*}{R07} & 256 & Disjoint opt. & 1--2 & 10800.00* & -- & $\mathbf{\underline{9.69 \times 10^{-3}}}$ & $\mathbf{\underline{61.44}}$ \\
\cmidrule{2-8}
& 1024 & Disjoint opt. & 4--8 & 10800.00* & -- & $\mathbf{1.29 \times 10^{-2}}$ & $\mathbf{72.25}$ \\
\cmidrule{2-8}
& 2048 & Disjoint opt. & 1--2 & 10800.00* & -- & $\mathbf{\underline{1.01 \times 10^{-2}}}$ & $\mathbf{\underline{80.75}}$ \\
\bottomrule
\end{tabular}
\end{table}

Table \ref{tab:R04R05R07_ALGO_summary} evaluates the iterative refinement procedure proposed in Algorithm \ref{algo:refchaincuts_algo}, where information generated at one refinement level is transferred to the next level. Detailed results with all the grouping configurations are reported in Appendix \ref{sec_app_C3} (see Tables \ref{tab:R04R05_ALGO_beta05}--\ref{tab:R07_ALGO_beta05}). For the R04 instance, Algorithm \ref{algo:refchaincuts_algo} is particularly effective when moving between the finest refinement levels. The best result is obtained for the transition from group size 1 to group size 2, achieving a reduction of 98.61\% relative to the single-cut formulation, which improves upon the corresponding fixed-level approach. Similar behavior is observed for the transition from group size 2 to group size 4 (see Table \ref{tab:R04R05_ALGO_beta05}). For the R05 instance, the gains obtained through the iterative procedure are less pronounced. The best performance is observed for the transition between levels 16 and 32, with a reduction of 62.28\%. Although Algorithm \ref{algo:refchaincuts_algo} remains competitive, the improvements are generally smaller than those observed for R04.

The results for the R07 instance further highlight the benefits of transferring information between consecutive refinement levels. For $|\mathcal{S}|=256$, the best performance is obtained for the transition from group size 1 to group size 2, achieving a gap reduction of 61.44\% with respect to the single-cut formulation. Similar results are observed for the larger scenario sets, where the best reductions are 72.25\% and 80.75\% for $|\mathcal{S}|=1024$ and $|\mathcal{S}|=2048$, respectively. Transitions involving coarser aggregation levels generally lead to larger relative gaps (see Table \ref{tab:R07_ALGO_beta05}). This behavior indicates that the Benders cuts inherited from one refinement level provide increasingly less informative approximations of the recourse function when transferred to refinement levels characterized by a significantly larger number of scenario groups.

\begin{figure}[h]
    \centering

    \begin{subfigure}{0.7\textwidth}
        \centering
        \includegraphics[width=\textwidth]{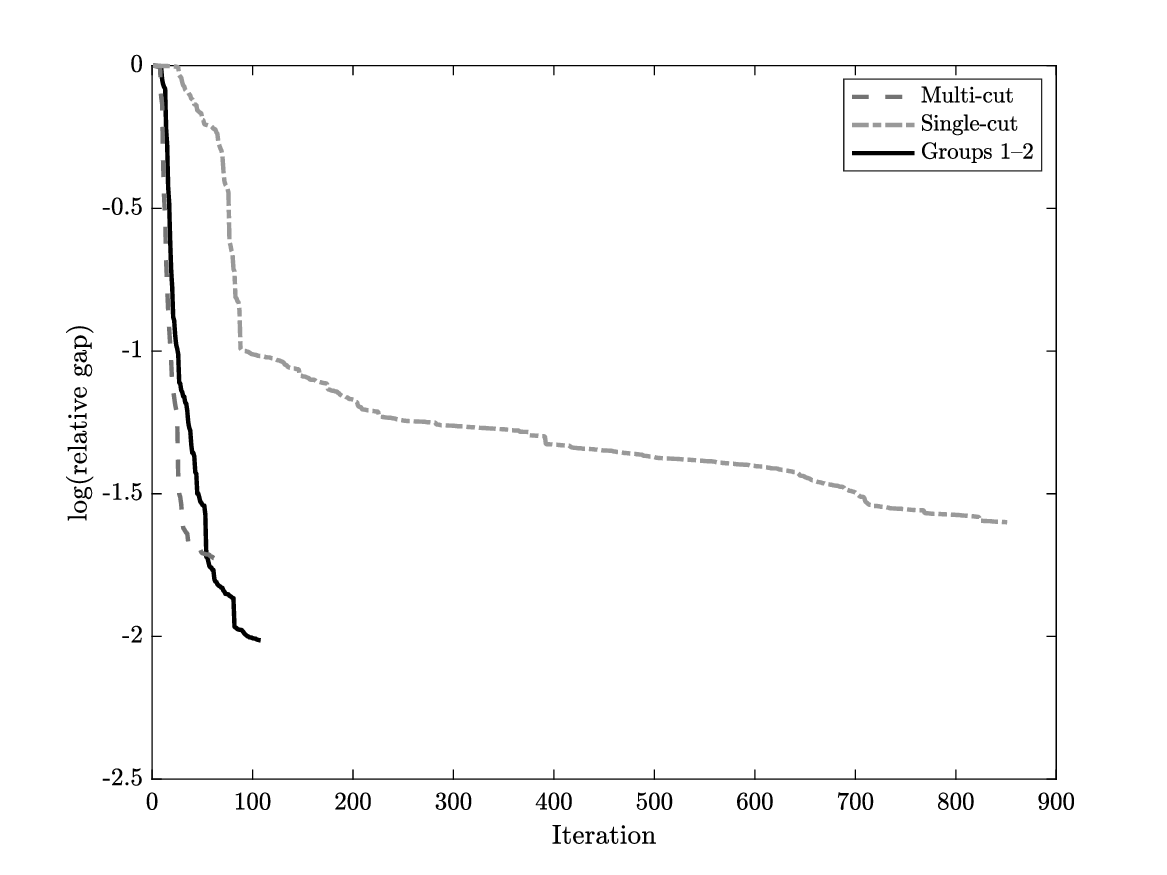}
        \caption{Relative optimality gap.}
        \label{B_fig_relgap_R07_gap}
    \end{subfigure}

    \vspace{0.3cm}

    \begin{subfigure}{0.48\textwidth}
        \centering
        \includegraphics[width=\textwidth]{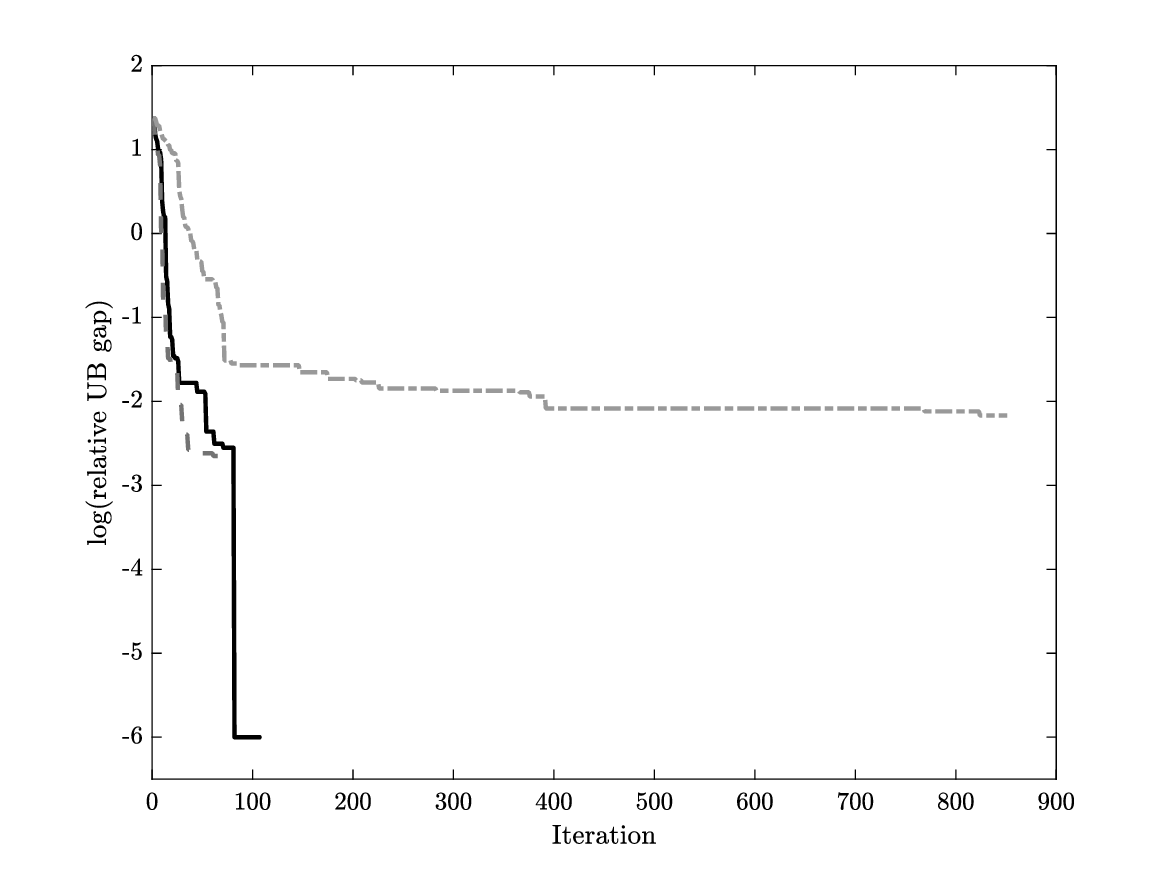}
        \caption{Relative gap from the best UB.}
        \label{B_fig_relgap_R07_UB}
    \end{subfigure}
    \hfill
    \begin{subfigure}{0.48\textwidth}
        \centering
        \includegraphics[width=\textwidth]{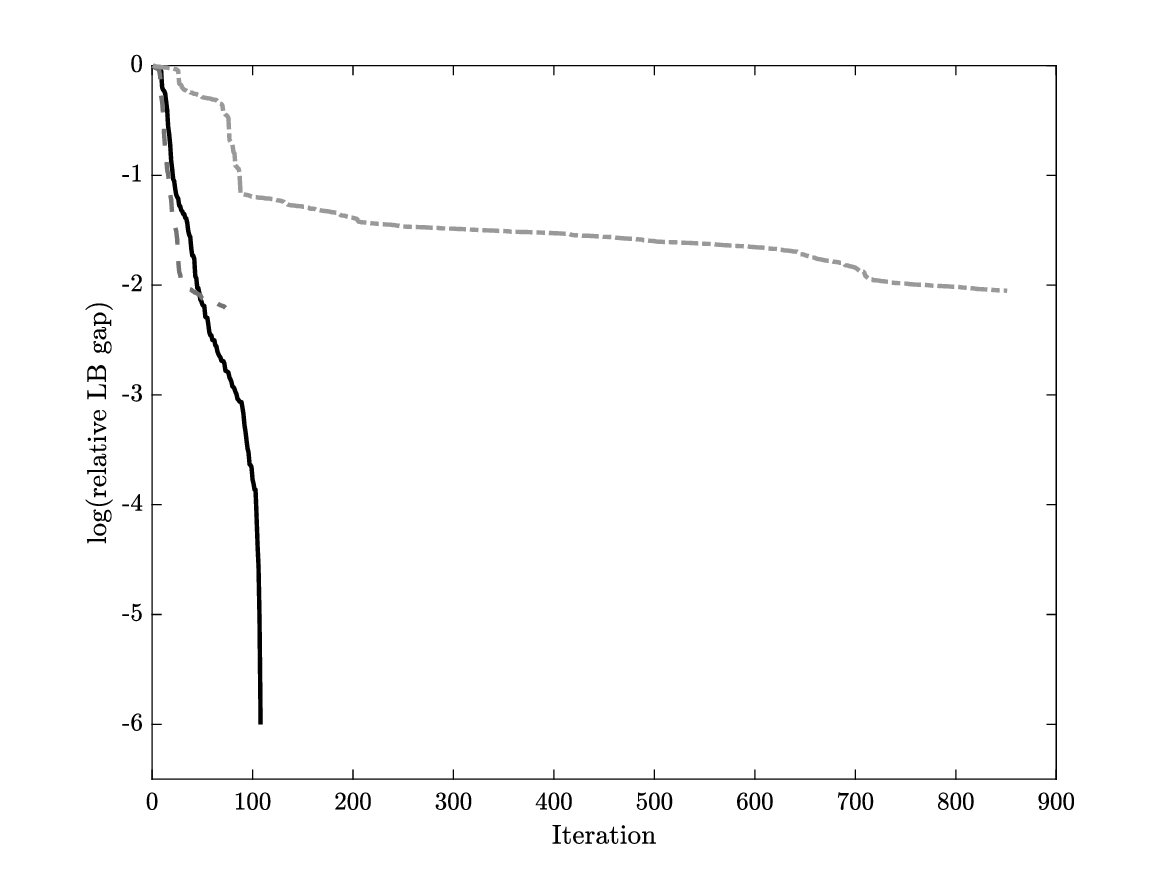}
        \caption{Relative gap from the best LB.}
        \label{B_fig_relgap_R07_LB}
    \end{subfigure}

    \caption{Relative gaps as a function of the number of Benders iterations for the R07 instance with $|\mathcal{S}|=256$. Comparison between the multi-cut formulation, the single-cut formulation, and the refinement-based Algorithm \ref{algo:refchaincuts_algo} between group sizes 1 and 2.}
    \label{B_fig_relgap_R07}
\end{figure}

To further investigate the performance of Algorithm \ref{algo:refchaincuts_algo} when solving the large-scale R07 instance with $|\mathcal{S}|=256$, Figure \ref{B_fig_relgap_R07} and Table \ref{tab:R07_convergence_statistics} provide a more detailed analysis of the convergence behavior of the different formulations. In particular, Figure \ref{B_fig_relgap_R07} illustrates the evolution of the optimality gap and of the lower and upper bounds over the Benders iterations, while Table \ref{tab:R07_convergence_statistics} reports complementary statistics on the number of iterations, the average time per iteration, the number of optimality cuts, and the attainment of the best upper bound. For ease of visualization, Figure \ref{B_fig_relgap_R07} compares the classical multi-cut and single-cut formulations with Algorithm \ref{algo:refchaincuts_algo} applied between group sizes 1--2, as this configuration dominates all the other refinement-based formulations. Table \ref{tab:R07_convergence_statistics}, instead, provides a broader comparison by considering different pairs of consecutive levels of the refinement chain, from group sizes 1--2 to 16--32.

Figure \ref{B_fig_relgap_R07_gap} reports the logarithm of the relative optimality gap $gap_r$ (see line \ref{alg_ln24} of Algorithm \ref{algo:refchaincuts_algo}) as a function of the number of Benders iterations. Figures \ref{B_fig_relgap_R07_UB} and \ref{B_fig_relgap_R07_LB} complement this analysis by reporting the logarithm of the relative gaps between the upper and lower bounds obtained at each iteration and the best upper and lower bounds, respectively. These best bounds are attained by the 1--2 configuration. For visualization on the logarithmic scale, zero gaps are replaced by $10^{-6}$.

The evolution of the relative optimality gap in Figure \ref{B_fig_relgap_R07_gap} shows that the single-cut formulation consistently exhibits the slowest convergence rate, maintaining the largest relative gap throughout the entire solution process. This confirms the weaker approximation of the recourse function provided by a highly aggregated representation of the scenario set. The comparison between the multi-cut formulation and Algorithm \ref{algo:refchaincuts_algo} highlights an interesting trade-off. During the initial phase of the solution process, up to approximately iteration 57, the multi-cut formulation achieves smaller relative gaps than the proposed algorithm between group sizes 1--2. This indicates that the larger number of cuts generated at each iteration provides a stronger approximation of the recourse function. However, after this point the 1--2 configuration overtakes the multi-cut formulation and achieves smaller relative gaps. This behavior suggests that the refinement chain strategy is able to preserve much of the strength of the multi-cut approach while maintaining a significantly smaller master problem, ultimately leading to better overall convergence within the available computational time.

Figures \ref{B_fig_relgap_R07_UB} and \ref{B_fig_relgap_R07_LB} provide further insight into the evolution of the primal and dual bounds. In the L-shaped method, convergence depends on both identifying high-quality primal solutions, reflected in the upper bound, and generating optimality cuts that progressively strengthen the master problem, reflected in the lower bound. Regarding the former (see Figure \ref{B_fig_relgap_R07_UB}), the multi-cut formulation exhibits faster initial progress and attains its best upper bound at iteration 61 (see also the fifth column of Table \ref{tab:R07_convergence_statistics}). The 1--2 configuration initially improves the upper bound more slowly, but continues to generate high-quality primal solutions and eventually attains the best upper bound among all the considered methods at iteration 82. A different behavior emerges for the lower bound (see Figure \ref{B_fig_relgap_R07_LB}). Although the multi-cut formulation initially provides a strong approximation of the recourse function, the 1--2 configuration progressively produces tighter lower bounds and ultimately attains the best lower bound. Hence, the advantage of the 1--2 configuration cannot be attributed solely to solving a smaller master problem more quickly. The aggregated cuts generated along the refinement chain lead to a different Benders trajectory, which combines continued improvement of the primal bound with effective strengthening of the dual bound. Together with the lower computational cost of each iteration, this allows the 1--2 configuration to achieve the smallest relative optimality gap within the available computational time.

Table \ref{tab:R07_convergence_statistics} highlights the computational trade-off associated with the different levels of scenario aggregation. At one extreme, the multi-cut formulation generates one optimality cut for each scenario at every iteration, resulting in 15995 optimality cuts and an average iteration time of 147.95 seconds. Consequently, only 73 iterations are completed within the three-hour time limit. At the other extreme, the single-cut formulation generates only one cut per iteration, resulting in an average iteration time of 12.69 seconds and allowing 851 iterations to be completed. The refinement-based configurations lie between these two extremes. As the group size increases, the number of generated cuts decreases, reducing the average computational cost per iteration and allowing a larger number of iterations to be performed. For instance, the 1--2 configuration completes 108 iterations with an average iteration time of 100.00 seconds and 13824 generated optimality cuts, whereas the 16--32 configuration completes 440 iterations with an average iteration time of 24.55 seconds and 3520 optimality cuts.

The larger number of iterations enabled by stronger scenario aggregation, however, does not necessarily translate into better convergence. As reported in the last column of Table \ref{tab:R07_convergence_statistics}, the 1--2 configuration achieves the smallest final relative gap, equal to $9.69\times10^{-3}$, compared with $1.81\times10^{-2}$ for the multi-cut formulation and $2.51\times10^{-2}$ for the single-cut formulation. The intermediate configurations also show that progressively reducing the computational cost per iteration does not lead to a monotonic improvement in the final gap. These results indicate that the computational advantage of the refinement-based approach stems from balancing the strength of the recourse function approximation with the complexity of the resulting master problem, rather than simply maximizing the number of iterations performed.

\begin{table}[htbp]

\centering
\caption{Computational results for the R07 instance with $|\mathcal{S}|=256$, comparing the multi-cut and single-cut formulations with Algorithm \ref{algo:refchaincuts_algo} between consecutive levels.}
\label{tab:R07_convergence_statistics}

\setlength{\tabcolsep}{4pt}
\begin{tabular}{lllllll}
\toprule
Grouping & Consecutive & Number of & Avg. iteration & Best UB & Number of & Final\\
configuration & group size & iterations & time (s) & iteration & opt. cuts & rel. gap\\
\midrule
Multi-cut & -- & 73 & 147.95 & 61 & 15995 & $1.81 \times 10^{-2}$ \\
\midrule
Disjoint opt. & 1--2 & 108 & 100.00 & 82 & 13824 & $9.69 \times 10^{-3}$ \\
Disjoint opt. & 2--4 & 150 & 72.00 & 145 & 9600 & $1.17 \times 10^{-2}$ \\
Disjoint opt. & 4--8 & 229 & 47.16 & 177 & 7328 & $2.13 \times 10^{-2}$ \\
Disjoint opt. & 8--16 & 296 & 36.49 & 254 & 4736 & $1.65 \times 10^{-2}$ \\
Disjoint opt. & 16--32 & 440 & 24.55 & 224 & 3520 & $2.51 \times 10^{-2}$ \\
\midrule
Single-cut & -- & 851 & 12.69 & 824 & 851 & $2.51 \times 10^{-2}$ \\
\bottomrule
\end{tabular}%

\end{table}

Overall, the computational study highlights the existence of an intermediate region of the refinement chain that provides a favorable trade-off between the strength of the generated cuts and the complexity of the master problem. The proposed refinement chain cuts framework generalizes both the single-cut and multi-cut formulations, while providing the flexibility to select intermediate aggregation levels that can improve computational performance. A particularly relevant observation is that the best-performing configurations are consistently found near the multi-cut level of the refinement chain. At these levels, scenario aggregation retains much of the strength associated with a disaggregated representation while reducing the number of cuts and the computational cost of each iteration. This finding suggests that a limited amount of aggregation can reduce the burden associated with managing a large number of cuts while preserving a strong approximation of the recourse function. Furthermore, the iterative refinement-based procedure outlined in Algorithm \ref{algo:refchaincuts_algo} proves particularly effective when applied between finely spaced refinement levels.

\section{Conclusions} \label{sec_conclusions}
In this paper, we introduced a novel \textit{L-shaped refinement chain cuts method} for solving two-stage stochastic programs by integrating the refinement chain of scenarios within the classical L-shaped decomposition framework. We proposed a generalized decomposition approach in which one dual subproblem is solved for each subgroup of scenarios at every refinement level, thereby encompassing both the classical multi-cut and single-cut L-shaped methods as special cases. We also established theoretical convergence properties for every level of the refinement chain and characterized the relationships between consecutive refinement levels in terms of feasibility and optimality cuts. These results enabled the development of an iterative refinement-based solution algorithm capable of exploiting information generated across different levels of the refinement chain. Finally, we demonstrated the effectiveness of the proposed framework on a two-stage stochastic fixed-charge multicommodity network design problem under a mean-risk formulation, showing competitive computational performance and highlighting the applicability of the proposed approach to large-scale risk-averse stochastic optimization problems.

Building on these results, there are several promising directions for future research, both from computational and methodological perspectives. First, advanced refinement-based strategies across multiple levels of the refinement chain could be developed, allowing the derivation of an algorithm that adaptively refines and explores selected subsets of scenarios in greater detail. This could further enhance both the efficiency and the convergence behavior of the proposed framework. Second, the development of strategies capable of identifying \textit{a priori} the most promising aggregation level and scenario grouping strategy represents an interesting research direction. The computational results suggest that the best performance is often achieved at intermediate levels of the refinement chain and that the effectiveness of a given partition is strongly instance dependent. Therefore, deriving indicators or learning-based approaches able to predict suitable aggregation structures before the optimization process could significantly improve the practical applicability of the proposed framework. Third, acceleration and stabilization techniques for Benders decomposition, such as those proposed in \cite{rahmaniani2018accelerating,yuan2009enhanced}, could be incorporated to further improve the computational performance of the proposed method. The integration of such techniques could reduce the number of iterations and improve the convergence of the algorithm on large-scale instances.

From a methodological perspective, fourth, the proposed framework could be extended to multi-stage stochastic programming problems in a nested Benders decomposition fashion. Such an extension would allow the refinement chain structure to be exploited across multiple stages of the stochastic process. Finally, alternative approaches for constructing the refinement chain of scenarios could be investigated and integrated within the proposed decomposition framework. This would provide greater flexibility in the definition of scenario subgroups and could potentially improve computational tractability.

\backmatter

\bmhead{Acknowledgements}

The authors acknowledge the University of Bergamo and the agreement with CINECA for the availability of high performance computing resources and support. Finally, Francesca Maggioni and Andrea Spinelli acknowledge the support received from Gruppo Nazionale per il Calcolo Scientifico (GNCS-INdAM).

\section*{Declarations}

\bmhead{Competing Interests} The authors declare that they have no conflict of interest.

\bmhead{Code availability} The codes developed for this work are
made publicly available on GitHub (\url{https://github.com/aspinellibg/RefinementLshaped}).

\color{black}
\clearpage
\newpage
\bibliography{bibliography_BoundsBenders}

\clearpage
\newpage

\begin{appendices}

\section{Proofs}\label{sec_app_A}

In this appendix, we provide the proofs of the technical results omitted from the main text.

\subsection{Proof of Lemma \ref{lemma_relation_phi_f_disjoint}}\label{sec_app_A1}

By applying properties (P2)--(P3) at refinement level $j$, we obtain:

\begin{equation} \label{lemma1_proof_1}
    \mathbb{P}
    = \sum_{i=1}^{m_j} \pi_{i}^{(j)} \mathbb{P}_{i}^{(j)} 
    = \sum_{i=1}^{m_j}\pi_{i}^{(j)} \sum_{\ell=1}^{N_i} \pi_{k_\ell,i}^{(j-1,j)} \mathbb{P}_{k_\ell}^{(j-1)} 
    = \sum_{i=1}^{m_j}\sum_{\ell=1}^{N_i} \pi_{i}^{(j)} \pi_{k_\ell,i}^{(j-1,j)} \mathbb{P}_{k_\ell}^{(j-1)}.
\end{equation}

Moreover, for refinement chains composed either of disjoint partitions or of fixed-scenario subsets, it holds that:
\begin{equation*}
    \forall i=1,\ldots,m_{j-1}, \quad \exists!\,\tilde{i}\in\{1,\ldots,m_j\}: \Omega_{\tilde{i}}^{(j)} \supseteq \Omega_i^{(j-1)}.
\end{equation*}

Since property (P2) holds at every level of the refinement chain, it also holds at level $j-1$. Moreover, the above observation implies that each subset $\Omega_i^{(j-1)}$ appears exactly once among the subsets $\Omega_{k_\ell}^{(j-1)}$, $\ell=1,\ldots,N_i$, associated with level $j$. Hence:

\begin{equation} \label{lemma1_proof_2}
    \mathbb{P}
    = \sum_{i=1}^{m_{j-1}} \pi_i^{(j-1)} \mathbb{P}_i^{(j-1)} 
    = \sum_{i=1}^{m_j}\sum_{\ell=1}^{N_i} \pi_{k_\ell}^{(j-1)} \mathbb{P}_{k_\ell}^{(j-1)}.
\end{equation}

Comparing the last expressions in \eqref{lemma1_proof_1} and \eqref{lemma1_proof_2}, the thesis follows.

\hfill {\Large$\square$}

\subsection{Proof of Lemma \ref{lemma_relations_intralevelwights_consecutivelevels}}\label{sec_app_A2}

    We discuss the two cases separately. On the one hand, if all subsets $\Omega_{k_\ell}^{(j-1)}$ have the first $f\geq 1$ scenarios fixed, then, according to \eqref{eq:weights_scenario_ffixed}, we have that:
    \small 
    \begin{equation*}
        \sum_{\ell=1}^{N_i}\pi_{k_\ell}^{(j-1)} = \sum_{\ell=1}^{N_i} \displaystyle \frac{\sum_{s=f+1}^{|\Omega_{k_{\ell}}^{(j)}|} p^{s}}{1-\sum_{s=1}^{f}p^{s}} 
        = \frac{1}{1-\sum_{s=1}^{f}p^{s}}\sum_{\ell=1}^{N_i} \sum_{s=f+1}^{|\Omega_{k_{\ell}}^{(j)}|} p^{s}
        = \frac{1}{1-\sum_{s=1}^{f}p^{s}}\sum_{s=f+1}^{|\Omega_{i}^{(j)}|} p^{s}
        = \pi_{i}^{(j)},
    \end{equation*}
    \normalsize
    where, in the third equality, we use the fact that each non-fixed scenario belongs to exactly one subset $\Omega_{k_\ell}^{(j-1)}$.

    On the other hand, if all subsets $\Omega_{k_\ell}^{(j-1)}$ are disjoint, then, according to \eqref{eq:weights_scenario_disjoint}, it holds that:
    \begin{equation*}
        \begin{split}
            \pi_{i}^{(j)}=\sum_{s: \bm \omega^s \in \Omega_{i}^{(j)}}p^s
             = \sum_{s: \bm \omega^s \in \bigcup_{\ell=1}^{N_i} \Omega_{k_\ell}^{(j-1)}} p^s 
             &= \sum_{s: \bm \omega^s \in \Omega_{k_1}^{(j-1)}}p^s + \ldots + \sum_{s: \bm \omega^s \in \Omega_{k_{N_i}}^{(j-1)}}p^s\\
             &= \pi_{k_1}^{(j-1)} + \ldots + \pi_{k_{N_i}}^{(j-1)} 
             = \sum_{\ell=1}^{N_i}\pi_{k_\ell}^{(j-1)},
        \end{split}
    \end{equation*}
    where, in the third equality, similarly as before, we rely on the fact that each scenario in $\Omega_{i}^{(j)}$ belongs to exactly one subset $\Omega_{k_\ell}^{(j-1)}$. \hfill {\Large$\square$}

\subsection{Proof of Lemma \ref{lemma:feasibility_convex_comb}}\label{sec_app_A3}

    Let $\tilde{\bm \lambda}^{(j)}_{i}:=\sum_{\ell=1}^{N_i}\pi_{k_\ell,i}^{(j-1,j)}\bm \lambda_{k_\ell}^{(j-1)}$. We need to show that $\tilde{\bm \lambda}^{(j)}_{i}$ satisfies condition \eqref{eq:constr1_scenario_subproblem_DUAL_omega_i_j}.
    Consider scenario $\bm \omega^s\in\Omega^{(j)}_i$. We have that:
    \begin{equation*}
            \bm W^\top \big(\tilde{\bm \lambda}^{(j)}_{i}\big)^s = \bm W ^\top \sum_{\ell=1}^{N_i}\pi_{k_\ell,i}^{(j-1,j)}\big(\bm \lambda_{k_\ell}^{(j-1)}\big)^s
            =\sum_{\ell=1}^{N_i}\pi_{k_\ell,i}^{(j-1,j)}\bigg(\bm W^\top \big(\bm \lambda_{k_\ell}^{(j-1)}\big)^s\bigg).
    \end{equation*}
    Since by hypothesis $\bm \lambda_{k_\ell}^{(j-1)}\in\Lambda^{(j-1)}_{k_\ell}$, each $\big(\bm \lambda_{k_\ell}^{(j-1)}\big)^s$ satisfies condition \eqref{eq:constr1_scenario_subproblem_DUAL_omega_i_j}. Hence, we get:
    \begin{equation*}
        \begin{split}
             \bm W^\top \big(\tilde{\bm \lambda}^{(j)}_{i}\big)^s \leq \sum_{\ell=1}^{N_i}\pi_{k_\ell,i}^{(j-1,j)}p^s_{\Omega_{k_\ell}^{(j-1)}}\bm q^s
             = \bm q^s \sum_{\ell=1}^{N_i}\pi_{k_\ell,i}^{(j-1,j)}p^s_{\Omega_{k_\ell}^{(j-1)}}
             = p^s_{\Omega_{i}^{(j)}}\bm q^s,
        \end{split}
    \end{equation*}
    where the last equality follows from property (P3) of the refinement chain of probabilities. Therefore, $\tilde{\bm \lambda}^{(j)}_{i}$ satisfies condition \eqref{eq:constr1_scenario_subproblem_DUAL_omega_i_j}, which proves the lemma. \hfill {\Large$\square$}

\subsection{Proof of Lemma \ref{lemma:suff_cond}} \label{sec_app_A4}

    We discuss the two cases separately. In the first case, suppose that subsets $\{\Omega_{k_\ell}^{(j-1)}\}_{\ell=1}^{N_i}$ are disjoint. This implies that each scenario $\bm \omega^s\in\Omega_{i}^{(j)}$ appears in one and only one subset $\Omega_{k_{\overline{\ell}(s)}}^{(j-1)}$ at level $j-1$, with $\overline{\ell}(s)\in\{1,\ldots,N_i\}$. For each $\bm \omega^s\in\Omega_{i}^{(j)}$, we can consider:
    \begin{equation}\label{eq_lemma_disjoint}
        (\bm \lambda_{k_{{\ell}}}^{(j-1)})^s:= \begin{cases}
        \displaystyle\frac{1}{\pi_{k_{\ell},i}^{(j-1,j)}}(\bm \lambda^{(j)}_{i})^s &  \text{if }\ell = \overline{\ell}(s)  \\
    0 & \text{if }\ell \neq \overline{\ell}(s).
    \end{cases} 
        \end{equation}
    First, we need to show that the collection of points defined according \eqref{eq_lemma_disjoint} satisfies $\bm \lambda_{k_{{\ell}}}^{(j-1)}\in \Lambda_{k_{{\ell}}}^{(j-1)}$ for all $\ell=1,\ldots,N_i$. Consider a scenario $\bm \omega^s\in\Omega_{i}^{(j)}$ with corresponding unique subset $\Omega_{k_{\overline{\ell}(s)}}^{(j-1)}\ni\bm \omega^s$. For all subsets $\Omega_{k_{\ell}}^{(j-1)}$ with $\ell\neq \overline{\ell}(s)$, it follows that $(\bm \lambda_{k_{\ell}}^{(j-1)})^s= 0$, and therefore condition \eqref{eq:constr1_scenario_subproblem_DUAL_omega_i_j} is automatically satisfied by setting $p^s_{\Omega_{k_\ell}^{(j-1)}}=0$. Conversely, for the subset $\Omega_{k_{\overline{\ell}(s)}}^{(j-1)}$, we have that:
        \begin{equation} \label{eq:dimlemma}
                \bm W^\top (\bm \lambda_{k_{\overline{\ell}(s)}}^{(j-1)})^s=\frac{1}{\pi_{k_{\overline{\ell}(s)},i}^{(j-1,j)}}\bm W^\top(\bm \lambda^{(j)}_{i})^s\leq \frac{1}{\pi_{k_{\overline{\ell}(s)},i}^{(j-1,j)}} p^{s}_{\Omega_{i}^{(j)}}\bm q^s,
        \end{equation}
        where the last inequality holds because $(\bm \lambda_i^{(j)})^s\in(\Lambda^{(j)}_i)^s$. Moreover, combining \eqref{eq:relation_phi_f_disjoint} with \eqref{eq:prob_scenario_disjoint}--\eqref{eq:weights_scenario_disjoint}, the right-hand side of \eqref{eq:dimlemma} can be reformulated as:
        \begin{equation*}
                \frac{1}{\pi_{k_{\overline{\ell}(s)},i}^{(j-1,j)}} p^{s}_{\Omega_{i}^{(j)}}\bm q^s   =\frac{\pi_{i}^{(j)}}{\pi_{k_{\overline{\ell}(s)}}^{(j-1)}}\cdot\frac{1}{\pi_{i}^{(j)}}p^s\bm q^s
                =\frac{1}{\pi_{k_{\overline{\ell}(s)}}^{(j-1)}}\cdot p^s_{\Omega_{k_{\overline{\ell}(s)}}^{(j-1)}}\pi_{k_{\overline{\ell}(s)}}^{(j-1)}\bm q^s
                =p^s_{\Omega_{k_{\overline{\ell}(s)}}^{(j-1)}}\bm q^s.
        \end{equation*}
        Therefore, $\bm W^\top (\bm \lambda_{k_{\overline{\ell}(s)}}^{(j-1)})^s\leq p^s_{\Omega_{k_{\overline{\ell}(s)}}^{(j-1)}}\bm q^s$, or, equivalently, $(\bm \lambda_{k_{\overline{\ell}(s)}}^{(j-1)})^s\in (\Lambda_{k_{\overline{\ell}(s)}}^{(j-1)})^s$.
        Second, it remains to prove that the collection $\{\bm \lambda_{k_{\ell}^{(j-1)}}\}_{\ell=1}^{N_i}$ satisfies \eqref{eq_thesis_lemma}. To verify this, note that for all scenarios $\bm \omega^s\in\Omega_{i}^{(j-1)}$, it holds that:
            \begin{equation*}
        \sum_{\ell=1}^{N_i}\pi_{k_\ell,i}^{(j-1,j)}(\bm \lambda_{k_\ell}^{(j-1)})^s=\pi_{k_{\overline{\ell}(s)},i}^{(j-1,j)}(\bm \lambda^{(j-1)}_{k_{\overline{\ell}(s)}})^s=\pi_{k_{\overline{\ell}(s)},i}^{(j-1,j)}\cdot\frac{1}{\pi_{k_{\overline{\ell}(s)},i}^{(j-1,j)}}(\bm \lambda^{(j)}_{i})^s=(\bm \lambda^{(j)}_{i})^s.
    \end{equation*}

    In the second case, suppose that all subsets $\{\Omega_{k_\ell}^{(j-1)}\}_{\ell=1}^{N_i}$ have the first $f\geq 1$ scenarios fixed, appearing in all of them. By property (R2), these $f$ scenarios are also fixed in $\Omega^{(j)}_i$, since $\Omega^{(j)}_i$ is the union of the subsets $\{\Omega_{k_\ell}^{(j-1)}\}_{\ell=1}^{N_i}$. Therefore, for all $s=1,\ldots,f$, it is sufficient to set:
    \begin{equation*}
        (\bm \lambda_{k_\ell}^{(j-1)})^s:=(\bm \lambda^{(j)}_{i})^s \qquad \forall \ell=1,\ldots,N_i.
    \end{equation*}
    By construction, all remaining and non-fixed scenarios appear in one and only one subset $\Omega_{k_{\overline{\ell}(s)}}^{(j-1)}$, with $\overline{\ell}(s)\in\{1,\ldots,N_i\}$. Therefore, for all $s=f+1,\ldots,|\Omega_{i}^{(j)}|$, we can define $(\bm \lambda_{k_{{\ell}}}^{(j-1)})^s$ as in \eqref{eq_lemma_disjoint}.

    For the fixed scenarios, condition \eqref{eq:constr1_scenario_subproblem_DUAL_omega_i_j} is satisfied, because probability $p^s_{\Omega_{i}^{(j)}}=p^s$ for all $s=1,\ldots,f$ (see \eqref{eq:prob_scenario_ffixed}). For the non-fixed scenario, the proof is analogous to the case of disjoint subsets discussed above. Regarding condition \eqref{eq_thesis_lemma}, for all fixed scenarios $s=1,\ldots,f$, it holds that:
    \begin{equation*}
        \sum_{\ell=1}^{N_i}\pi_{k_\ell,i}^{(j-1,j)}(\bm \lambda_{k_\ell}^{(j-1)})^s=\sum_{\ell=1}^{N_i}\pi_{k_\ell,i}^{(j-1,j)}(\bm \lambda^{(j)}_{i})^s=(\bm \lambda^{(j)}_{i})^s\sum_{\ell=1}^{N_i}\pi_{k_\ell,i}^{(j-1,j)}=(\bm \lambda^{(j)}_{i})^s,
    \end{equation*}
    where the last equality follows from property (P2) on the intra-level weights. Finally, for all remaining scenarios, we have that:
    \begin{equation*}
        \sum_{\ell=1}^{N_i}\pi_{k_\ell,i}^{(j-1,j)}(\bm \lambda_{k_\ell}^{(j-1)})^s=\pi_{k_{\overline{\ell}},i}^{(j-1,j)}(\bm \lambda^{(j-1)}_{k_{\overline{\ell}}})^s=\pi_{k_{\overline{\ell}},i}^{(j-1,j)}\cdot\frac{1}{\pi_{k_{\overline{\ell}},i}^{(j-1,j)}}(\bm \lambda^{(j)}_{i})^s=(\bm \lambda^{(j)}_{i})^s.
    \end{equation*}
    This concludes the proof. \hfill {\Large$\square$}

\subsection{Proof of Proposition \ref{prop:extremepoints_omegaij}} \label{sec_app_A5}

 By Lemma \ref{lemma:feasibility_convex_comb}, the point $\bm {e}^{(j)}_{i}$ belongs to $\Lambda^{(j)}_i$, since it is defined as a convex combination of points in $\{\Lambda^{(j-1)}_{k_\ell}\}_{\ell=1}^{N_i}$ with weights $\{\pi_{k_\ell,i}^{(j-1,j)}\}_{\ell=1}^{N_i}$. We first prove the forward implication. Suppose that each $\bm e_{k_\ell}^{(j-1)}$ is an extreme point of $\Lambda_{k_\ell}^{(j-1)}$ for all $\ell=1,\ldots,N_i$. Assume, by contradiction, that $\bm e_i^{(j)}$ is not an extreme point of $\Lambda_i^{(j)}$. Then, there exist distinct points $\tilde{\bm e}_i^{(j)},\hat{\bm e}_i^{(j)}\in \Lambda^{(j)}_i$, with $\tilde{\bm e}_i^{(j)},\hat{\bm e}_i^{(j)}\neq \bm e^{(j)}_i$ and $\alpha\in(0,1)$ such that $\bm e_i^{(j)}=\alpha\tilde{\bm e}_i^{(j)}+(1-\alpha)\hat{\bm e}_i^{(j)}$. By Lemma \ref{lemma:suff_cond}, since $\tilde{\bm e}_i^{(j)},\hat{\bm e}_i^{(j)}\in \Lambda^{(j)}_i$, there exist two collections $\{\tilde{\bm e}_{k_\ell}^{(j-1)}\}_{\ell=1}^{N_i}$, $\{\hat{\bm e}_{k_\ell}^{(j-1)}\}_{\ell=1}^{N_i}$ with $\tilde{\bm e}_{k_\ell}^{(j-1)},\hat{\bm e}_{k_\ell}^{(j-1)}\in\Lambda^{(j)}_{k_\ell}$ for all $\ell=1,\ldots,N_i$ such that:
    \begin{equation*}
        \tilde{\bm e}_i^{(j)}=\sum_{\ell=1}^{N_i}\pi_{k_\ell,i}^{(j-1,j)}\tilde{\bm e}_{k_\ell}^{(j-1)} \qquad \text{and} \qquad \hat{\bm e}_i^{(j)}=\sum_{\ell=1}^{N_i}\pi_{k_\ell,i}^{(j-1,j)}\hat{\bm e}_{k_\ell}^{(j-1)}.
    \end{equation*}
    Therefore, we can re-write $\bm e_{i}^{(j)}$ as:
    \begin{equation*}
    \begin{split}
        \bm e_i^{(j)}&=\alpha\sum_{\ell=1}^{N_i}\pi_{k_\ell,i}^{(j-1,j)}\tilde{\bm e}_{k_\ell}^{(j-1)}+(1-\alpha)\sum_{\ell=1}^{N_i}\pi_{k_\ell,i}^{(j-1,j)}\hat{\bm e}_{k_\ell}^{(j-1)}\\
        &=\sum_{\ell=1}^{N_i}\pi_{k_\ell,i}^{(j-1,j)}[\alpha\tilde{\bm e}_{k_\ell}^{(j-1)}+(1-\alpha)\hat{\bm e}_{k_\ell}^{(j-1)}].
        \end{split}
    \end{equation*}
    Comparing this expression with \eqref{eq:extr_point} we obtain $\bm e_{k_\ell}^{(j-1)}=\alpha\tilde{\bm e}_{k_\ell}^{(j-1)}+(1-\alpha)\hat{\bm e}_{k_\ell}^{(j-1)}$ for all $\ell=1,\ldots,N_i$. This contradicts the assumption that each $\bm e_{k_\ell}^{(j-1)}$ is an extreme point of $\Lambda_{k_\ell}^{(j-1)}$.
    
    We now prove the converse implication. Suppose that $\bm e_i^{(j)}$ is an extreme point of $\mathcal{E}(\Lambda_{i}^{(j)})$. Fix $\ell\in\{1,\ldots,N_i\}$ and assume, by contradiction, that $\bm e_{k_\ell}^{(j-1)}$ is not an extreme point of $\Lambda_{k_\ell}^{(j-1)}$. Then, there exist distinct points $\tilde{\bm e}_{k_\ell}^{(j-1)}$ and $\overline{\bm e}_{k_\ell}^{(j-1)}$ in $\Lambda_{k_{\ell}}^{(j-1)}$ with $\tilde{\bm e}_{k_\ell}^{(j-1)},\overline{\bm e}_{k_\ell}^{(j-1)}\neq \bm e_{k_\ell}^{(j-1)}$ and $\alpha\in(0,1)$ such that $\bm e_{k_\ell}^{(j-1)}=\alpha\tilde{\bm e}_{k_\ell}^{(j-1)}+(1-\alpha)\overline{\bm e}_{k_\ell}^{(j-1)}$. Therefore, substituting into \eqref{eq:extr_point} yields:
    \begin{equation*}
        \begin{split}
            \bm e_i^{(j)}&=\sum_{\ell=1}^{N_i}\pi_{k_\ell,i}^{(j-1,j)}\bm e_{k_\ell}^{(j-1)}\\
            &=\sum_{\ell=1}^{N_i}\pi_{k_\ell,i}^{(j-1,j)}[\alpha\tilde{\bm e}_{k_\ell}^{(j-1)}+(1-\alpha)\overline{\bm e}_{k_\ell}^{(j-1)}]\\
            &=\alpha \sum_{\ell=1}^{N_i}\pi_{k_\ell,i}^{(j-1,j)}\tilde{\bm e}_{k_\ell}^{(j-1)}+(1-\alpha)\sum_{\ell=1}^{N_i}\pi_{k_\ell,i}^{(j-1,j)}\overline{\bm e}_{k_\ell}^{(j-1)}.
        \end{split}
    \end{equation*}
    By Lemma \ref{lemma:feasibility_convex_comb}, each summation on the last right-hand side defines a point in $\Lambda_{i}^{j}$. Consequently, $\bm e_i^{(j)}$ is not an extreme point of $\Lambda_{i}^{j}$, which contradicts the assumption.
    
    Finally, we evaluate the objective function \eqref{eq:obj_scenario_subproblem_DUAL_omega_i_j} at $\bm e_i^{(j)}$:
    \begin{equation*}
        \begin{split}
            z_i^{(j)} & = \sum_{s: \bm \omega^s \in \Omega_i^{(j)}}(\bm h^s-\bm T^s\bm x)^\top (\bm e_{i}^{(j)})^s\\
            & = \sum_{s: \bm \omega^s \in \Omega_i^{(j)}}(\bm h^s-\bm T^s\bm x)^\top \sum_{\ell=1}^{N_i}\pi_{k_\ell,i}^{(j-1,j)} (\bm e_{k_\ell}^{(j-1)})^s\\
            & = \sum_{\ell=1}^{N_i}\pi_{k_\ell,i}^{(j-1,j)} \sum_{s: \bm \omega^s \in \Omega_i^{(j)}}(\bm h^s-\bm T^s\bm x)^\top (\bm e_{k_\ell}^{(j-1)})^s\\
            & = \sum_{\ell=1}^{N_i}\pi_{k_\ell,i}^{(j-1,j)} \sum_{s: \bm \omega^s \in \Omega_{k_\ell}^{(j-1)}}(\bm h^s-\bm T^s\bm x)^\top (\bm e_{k_\ell}^{(j-1)})^s\\
            & = \sum_{\ell=1}^{N_i}\pi_{k_\ell,i}^{(j-1,j)} z_{k_\ell}^{(j-1)},
        \end{split}
    \end{equation*}
    where the fourth equality follows from property (R2) of the refinement chain. This concludes the proof.    \hfill {\Large$\square$}

\subsection{Proof of Proposition \ref{prop:optimal_extreme_points}}\label{sec_app_A6}

By Proposition \ref{prop:extremepoints_omegaij}, the point $\bm e_{i}^{(j)}$ defined in \eqref{eq:extr_opt_point} is an extreme point of the dual feasible region $\Lambda_i^{(j)}$ at level $j$. Let $z_i^{(j)}$ denote its associated objective function value.
    
    We first prove the forward implication. Suppose that each $\bm e_{k_\ell}^{(j-1)}$ is an optimal extreme point of $\Lambda_{k_\ell}^{(j-1)}$ for all $\ell=1,\ldots,N_i$. Assume, by contradiction, that there exists another extreme point $\tilde{\bm e}_{i}^{(j)} \in \Lambda_i^{(j)}$ with objective function value $\tilde{z}_i^{(j)}>z_i^{(j)}$. By Proposition \ref{prop:extremepoints_omegaij}, there exists a collection $\{\tilde{\bm e}_{k_\ell}^{(j-1)}\}_{\ell=1}^{N_i}$, with $\tilde{\bm e}_{k_\ell}^{(j-1)}\in \Lambda_{k_\ell}^{(j-1)}$, for all $\ell=1,\ldots,N_i$, such that $\tilde{\bm e}_{i}^{(j)}=\sum_{\ell=1}^{N_i}\pi_{k_\ell,i}^{(j-1,j)}\tilde{\bm e}_{k_\ell}^{(j-1)}$, with corresponding objective function values $\tilde{z}_{k_\ell}^{(j-1)}$ satisfying $\tilde{z}_i^{(j)} = \sum_{\ell=1}^{N_i}\pi_{k_\ell,i}^{(j-1,j)} \tilde{z}_{k_\ell}^{(j-1)}$.
    Hence, by the assumption of contradiction, it holds that:
    \begin{equation*}
        \tilde{z}_i^{*(j)}-z_{i}^{(j)} = \sum_{\ell=1}^{N_i}\pi_{k_\ell,i}^{(j-1,j)} [\tilde{z}_{k_\ell}^{(j-1)}-z_{k_\ell}^{*(j-1)}]>0.
    \end{equation*}
    This is not possible since all inter-level weights $\pi_{k_\ell,i}^{(j-1,j)}$ are non-negative and each $z_{k_\ell}^{(j-1)}$ is assumed to be optimal, implying $z_{k_\ell}^{(j-1)}>\tilde{z}_{k_\ell}^{(j-1)}$ for all $\ell=1,\ldots,N_i$.
    
    We now prove the reverse implication. Suppose that $\bm e_i^{(j)}$ is an optimal extreme point of problem \eqref{eq:scenario_subproblem_DUAL_omega_i_j} at level $j$, with objective function value $z_i^{(j)}$. Assume by contradiction that there exists an index $\bar{\ell}\in\{1,\ldots,N_i\}$ and an extreme point $\tilde{\bm e}^{(j-1)}_{k_{\bar{\ell}}}$ such that $\tilde z^{(j-1)}_{k_{\bar{\ell}}} > z^{(j-1)}_{k_{\bar{\ell}}}$. Then, defining $\tilde{\bm e}_{i}^{(j)}$ as:
    \begin{equation*}
        \tilde{\bm e}_{i}^{(j)}:= \pi^{(j-1,j)}_{k_{\bar{\ell}},i} \tilde{\bm e}^{(j-1,j)}_{k_{\bar{\ell}},i} + \sum_{\ell=1,\ell\neq \bar{\ell}}^{N_i} \pi^{(j-1,j)}_{k_\ell,i} \bm e^{(j-1)}_{k_\ell},
    \end{equation*}
    the associated objective function value $\tilde{z}_{i}^{(j)}$ satisfies:
    \small
   \begin{equation*}
        \tilde{z}_{i}^{(j)} \!=\pi^{(j-1,j)}_{k_{\bar{\ell}},i} \tilde{z}^{(j-1,j)}_{k_{\bar{\ell}},i} + \!\!\!\sum_{\ell=1,\ell\neq \bar{\ell}}^{N_i} \pi^{(j-1,j)}_{k_\ell,i} z^{(j-1)}_{k_\ell} > \pi^{(j-1,j)}_{k_{\bar{\ell}},i} {z}^{(j-1,j)}_{k_{\bar{\ell}},i} + \!\!\!\sum_{\ell=1,\ell\neq \bar{\ell}}^{N_i} \!\!\pi^{(j-1,j)}_{k_\ell,i} z^{(j-1)}_{k_\ell} \!= z_i^{(j)},
   \end{equation*}
   \normalsize
   contradicting the optimality of $e_i^{(j)}$. This concludes the proof. \hfill {\Large$\square$}

\section{A numerical example of the L-shaped refinement chain cuts method}\label{sec_app_B}
\color{black}
    Consider the following two-stage stochastic optimization problem: 
\begin{subequations}
\begin{align}
    \min_{x,y_1^s,y_2^s} & \quad -x+\sum_{s=1}^5p^s(-2y^s_1-y_2^s) \\
    \text{s.t.} & \quad y_1^1+2y_2^1=1+2x \\
    & \quad y_1^2+2y_2^2=6-x \\
    & \quad y_1^3+2y_2^3=2+x \\
    & \quad y_1^4+2y_2^4=3+x \\
    & \quad y_1^5+2y_2^5=1-2x \\
    & \quad 0\leq x\leq 10 \\
    & \quad y_1^s,y_2^s\geq 0 & \forall s=1,\ldots,5,
\end{align}
\end{subequations}
 where $\Omega=\{\bm \omega^{1},\bm \omega^{2},\bm \omega^{3},\bm \omega^{4},\bm \omega^{5}\}$, with probabilities $p^{1}=\frac{3}{5},p^{2}=\frac{1}{10},p^{3}=\frac{1}{20},p^{4}=\frac{3}{20},p^{5}=\frac{1}{10}$. The optimal solution is $x^*=\frac{1}{2}$, with optimal objective value $z^*=-\frac{53}{10}$.
 
Assume to construct a refinement chain of $\Omega$ with three levels, where scenario $\bm \omega^1$ is fixed across all subsets:
    \begin{equation*}
\begin{array}{ccc}
     \text{Level } j & \big\{\Omega_{i}^{(j)}\big\}_{i=1}^{m_j} \\[1.5ex]
     3 & \Omega_1^{(3)}=\Omega & \big(\{\bm \omega^{1},\bm \omega^{2},\bm \omega^{3},\bm \omega^{4},\bm \omega^{5}\}\big)\\[1.5ex]
     2 & \big(\Omega_{1}^{(2)},\Omega_{2}^{(2)}\big) & \big(\{\bm \omega^{1},\bm \omega^{2},\bm \omega^{3}\},\{\bm \omega^{1},\bm \omega^{4},\bm \omega^{5}\}\big)\\[1.5ex]
    1 & \big(\Omega_{1}^{(1)},\Omega_{2}^{(1)},\Omega_{3}^{(1)},\Omega_{4}^{(1)}\big) & \big(\{\bm \omega^{1},\bm \omega^{2}\},\{\bm \omega^{1},\bm \omega^{3}\},\{\bm \omega^{1},\bm \omega^{4}\},\{\bm \omega^{1},\bm \omega^{5}\}\big)
\end{array}
\end{equation*}

In the following we consider each refinement level separately and formulate the corresponding Benders master problem \eqref{eq:master_problem_omega_i_j}.

\textbf{Level 1}
\begin{subequations}\label{eq:MP_level1}
\begin{align} 
     \min_{x,\theta^{(1)}_1,\theta^{(1)}_2,\theta^{(1)}_3,\theta^{(1)}_4} & \quad -x+\frac{1}{4}\theta^{(1)}_1+\frac{1}{8}\theta^{(1)}_2+\frac{3}{8}\theta^{(1)}_3+\frac{1}{4}\theta^{(1)}_4 \\
    \text{s.t.} & \quad x \leq 6, \quad x \geq -\frac{1}{2}, \quad x \leq \frac{1}{2} \label{eq:feasibility_level1} \\
    & \quad -6-\frac{8}{5}x \leq \theta_1^{(1)}, \quad -\frac{14}{5}-\frac{8}{5}x \leq \theta_2^{(1)} \label{eq:optimality_12_level1}\\
    & \quad -\frac{18}{5}-\frac{8}{5}x \leq \theta_3^{(1)}, \quad -2-\frac{4}{5}x \leq \theta_4^{(1)} \label{eq:optimality_34_level1}\\
    & \quad 0 \leq x \leq 10\\
    & \quad \theta_1^{(1)},\theta_2^{(1)},\theta_3^{(1)},\theta_4^{(1)} \in \mathbb{R}.
\end{align}
\end{subequations}

\textbf{Level 2}
\begin{subequations}\label{eq:MP_level2}
\begin{align} 
     \min_{x,\theta^{(2)}_1,\theta^{(2)}_2} & \quad -x+\frac{3}{8}\theta^{(2)}_1+\frac{5}{8}\theta^{(2)}_2 \\
    \text{s.t.} & \quad x \leq 6, \quad x \geq -\frac{1}{2}, \quad x \leq \frac{1}{2}, \quad x \geq -2, \quad x \geq -3 \label{eq:feasibility_level2} \\
    & \quad -\frac{74}{15}-\frac{32}{15}x \leq \theta_1^{(2)} \label{eq:optimality_1_level2}\\
    & \quad -\frac{74}{25}-\frac{56}{25}x \leq \theta_2^{(2)} \label{eq:optimality_2_level2}\\
    & \quad 0 \leq x \leq 10\\
    & \quad \theta_1^{(2)},\theta_2^{(2)} \in \mathbb{R}.
\end{align}
\end{subequations}

\textbf{Level 3}
\begin{subequations} \label{eq:MP_level3}
\begin{align} 
     \min_{x,\theta^{(3)}_1} & \quad -x+\theta^{(3)}_1 \\
    \text{s.t.} & \quad x \leq 6, \quad x \geq -\frac{1}{2}, \quad x \leq \frac{1}{2}, \quad x \geq -2, \quad x \geq -3 \label{eq:feasibility_level3} \\
    & \quad -\frac{37}{10}-\frac{11}{5}x \leq \theta_1^{(3)} \label{eq:optimality_1_level3}\\
    & \quad 0 \leq x \leq 10\\
    & \quad \theta_1^{(3)} \in \mathbb{R}.
\end{align}
\end{subequations}

Constraints \eqref{eq:feasibility_level1}, \eqref{eq:feasibility_level2}, and \eqref{eq:feasibility_level3} correspond to the feasibility cuts at each refinement level, whereas constraints \eqref{eq:optimality_12_level1}--\eqref{eq:optimality_34_level1}, \eqref{eq:optimality_1_level2}--\eqref{eq:optimality_2_level2}, and \eqref{eq:optimality_1_level3} are the corresponding optimality cuts. Problems \eqref{eq:MP_level1}, \eqref{eq:MP_level2}, and \eqref{eq:MP_level3} all admit the same optimal solution $x^*$ and optimal objective value $z^*$ as the original two-stage stochastic program, in agreement with Theorem \ref{teo_convergence}.

We now illustrate the relationship between two consecutive refinement levels by considering levels 1 and 2. In particular, we focus on subsets $\Omega_{1}^{(1)}$ and $\Omega_{2}^{(1)}$ at level 1, and on subset $\Omega_{1}^{(1)} \cup \Omega_{2}^{(1)}=\Omega_{1}^{(2)}$ at level 2. According to \eqref{eq:relation_phi_f_disjoint}, the corresponding inter-level weights are $\pi_{1,1}^{(1,2)}=\frac{2}{3}$ and $\pi_{2,1}^{(1,2)} = \frac{1}{3}$.

The extreme points of the dual feasible regions $\Lambda_{1}^{(1)}$, $\Lambda_{2}^{(1)}$, and $\Lambda_{1}^{(2)}$ are:
\begin{displaymath}
    \bm e_{1}^{(1)}=\left[
    \begin{array}{c} 
          \displaystyle -\frac{6}{5}\\[10pt]
          \displaystyle -\frac{4}{5}\\[10pt]
          0
    \end{array} \right] \qquad \bm e_{2}^{(1)}=\left[
    \begin{array}{c} 
          \displaystyle -\frac{6}{5}\\[10pt]
           0\\[10pt]
          \displaystyle -\frac{4}{5}
    \end{array} \right] \qquad \bm e_{1}^{(2)}=\left[
    \begin{array}{c} 
         \displaystyle -\frac{6}{5}\\[10pt]
         \displaystyle -\frac{8}{15}\\[10pt]
         \displaystyle -\frac{4}{15}
    \end{array} \right].
\end{displaymath}

By Proposition \ref{prop:extremepoints_omegaij}, it follows that:
\begin{displaymath}
    \bm e_{1}^{(2)}= \left[ \begin{array}{c}
         \displaystyle -\frac{6}{5}\\[10pt]
         \displaystyle -\frac{8}{15}\\[10pt]
         \displaystyle -\frac{4}{15}
    \end{array} \right] = \frac{2}{3}\left[ \begin{array}{c}
         \displaystyle-\frac{6}{5}  \\[10pt]
         \displaystyle-\frac{4}{5}\\[10pt]
         0
    \end{array}\right] + \frac{1}{3}\left[ \begin{array}{c}
         \displaystyle-\frac{6}{5}  \\[10pt]
         0\\[10pt]
         \displaystyle-\frac{4}{5}
    \end{array}\right] = \pi_{1,1}^{(1,2)}\bm e_{1}^{(1)}+\pi_{1,2}^{(1,2)}\bm e_{2}^{(1)}.
\end{displaymath}

Applying Proposition \ref{prop:optimality_cuts_prop} to the optimality cut associated with $\Omega_{1}^{(2)}$ yields:
\begin{equation*}
    \pi_{1,1}^{(1,2)}\bigg(-6-\frac{8}{5}x\bigg)+\pi_{1,2}^{(1,2)}\bigg(-\frac{14}{5}-\frac{16}{5}x\bigg) = -\frac{74}{15}-\frac{32}{15}x \leq \theta_1^{(2)}.
\end{equation*}

Regarding the feasibility cuts, the dual feasible region $\Lambda^s$ has a single extreme ray given by $r=-1$ for each $s$.

For the subset $\Omega_{1}^{(2)}=\{\bm \omega^1,\bm \omega^2,\bm \omega^3\}$, the extreme rays of the corresponding dual feasible region $\Lambda_{1}^{(2)}$ are:
\begin{displaymath}
    \bm r_{1}=\left[
    \begin{array}{c} 
          -1\\
          0\\
          0
    \end{array} \right] \qquad \bm r_{2}=\left[
    \begin{array}{c} 
          0\\
          -1\\
          0
    \end{array} \right] \qquad \bm r_{3}=\left[
    \begin{array}{c} 
          0\\
          0\\
          -1
    \end{array} \right],
\end{displaymath}
which is consistent with Proposition \ref{prop:feasibility_cuts_prop}. The same construction can be carried out for the remaining subsets of the refinement chain.
\color{black}

\section{Supplementary material for the multicommodity network design problem}\label{sec_app_C}

In this appendix, we provide supplementary material related to the two-stage stochastic fixed-charge multicommodity network design problem considered in Section \ref{sec_numericalresults}.

\subsection{Illustrative example}\label{sec_app_C1}

For the sake of illustration, Figure \ref{fig:network_design_example} reports an example of a network with two commodities, denoted by 1 and 2, which must be shipped from node 4 to node 7 and from node 1 to node 10, respectively. Origin and destination nodes associated with commodities 1 and 2 are represented by double-circle and dashed double-circle markers, respectively. The left panel, Figure \ref{fig:subfig_a}, shows the design arcs selected in the first stage. The right panel, Figure \ref{fig:subfig_b}, illustrates the second-stage decisions after the realization of the uncertain demands $d^1_7$ and $d^2_{10}$, equal to 12 and 21, respectively, as well as the realization of the arc capacities $u_{ij}$. At this stage, the flow decisions $y_{ij}^{1}$ and $y_{ij}^{2}$ are determined. Each design arc is associated with a label of the form $(u_{ij},(y_{ij}^{1},y_{ij}^{2}))$, where $u_{ij}$ denotes the realized arc capacity and $(y_{ij}^{1},y_{ij}^{2})$ the corresponding routed flows. Since the realized capacities of the design arcs are not sufficient to satisfy the demand of both commodities, the dummy arcs $(4,7)$ and $(1,10)$, represented by dashed lines, are activated for commodities 1 and 2, respectively. The labels on the dummy arcs indicate the amount of unmet demand routed through the corresponding recourse arcs.

\begin{figure}[H]
\centering

\begin{subfigure}[t]{0.40\linewidth}
\centering

\resizebox{\linewidth}{!}{%
\begin{tikzpicture}[
    >=Latex,
    standardarc/.style={-{Latex[length=1.8mm, width=1.8mm]}, thick},
    selectedarc/.style={-{Latex[length=1.8mm, width=1.8mm]}, very thick},
    externalarc/.style={-{Latex[length=1.8mm, width=1.8mm]}, thick, dashed},
    circnode/.style={
        draw, thick, circle,
        inner sep=0pt,
        minimum size=8mm
    },
    specialnode/.style={
        draw, thick, circle,
        inner sep=0pt,
        minimum size=8mm,
        double,
        double distance=1pt
    },
    dasheddoublenode/.style={
        draw,
        thick,
        circle,
        double,
        double distance=1pt,
        dashed,
        inner sep=0pt,
        minimum size=8mm
    }
]

\node[dasheddoublenode] (n1)  at (0,0) {1};
\node[above=2pt] at (n1.north) {\scriptsize $on^2$};
\node[circnode]   (n2)  at (2,1.6) {2};
\node[circnode]   (n3)  at (-3,-2.8) {3};
\node[specialnode](n4)  at (2.5,-2.8) {4};
\node[right=2pt] at (n4.east) {\scriptsize $on^1$};
\node[circnode]   (n5)  at (0,-2.8) {5};
\node[circnode]   (n6)  at (1,2.8) {6};
\node[specialnode](n7)  at (3.4,2.8) {7};
\node[above=2pt] at (n7.north) {\scriptsize $dn^1$};
\node[circnode]   (n8)  at (-0.6,2.8) {8};
\node[circnode]   (n9)  at (-1.6,0.9) {9};
\node[dasheddoublenode] (n10) at (-3,0) {10};
\node[above=2pt] at (n10.north) {\scriptsize $dn^2$};

\draw[standardarc, very thick] (n1) -- (n2);
\draw[standardarc, very thick] (n2) -- (n6);
\draw[standardarc, very thick] (n6) -- (n7);
\draw[standardarc, very thick] (n2) -- (n4);
\draw[standardarc, very thick] (n4) -- (n5);

\draw[selectedarc, bend right=25, very thick] (n1) to (n5);
\draw[selectedarc, bend right=25, very thick] (n5) to (n1);

\draw[standardarc, very thick] (n5) -- (n3);
\draw[standardarc, very thick] (n3) -- (n10);

\end{tikzpicture}
}

\caption{First-stage decisions.} \label{fig:subfig_a}
\end{subfigure}\hfill
\begin{subfigure}[t]{0.53\linewidth}
\centering

\resizebox{\linewidth}{!}{%
\begin{tikzpicture}[
    >=Latex,
    standardarc/.style={-{Latex[length=1.8mm, width=1.8mm]}, thick},
    selectedarc/.style={-{Latex[length=1.8mm, width=1.8mm]}, very thick},
    externalarc/.style={-{Latex[length=1.8mm, width=1.8mm]}, thick, dashed},
    circnode/.style={
        draw, thick, circle,
        inner sep=0pt,
        minimum size=8mm
    },
    specialnode/.style={
        draw, thick, circle,
        inner sep=0pt,
        minimum size=8mm,
        double,
        double distance=1pt
    },
    dasheddoublenode/.style={
        draw,
        thick,
        circle,
        double,
        double distance=1pt,
        dashed,
        inner sep=0pt,
        minimum size=8mm
    }
]

\node[dasheddoublenode] (n1)  at (0,0) {1};
\node[circnode]   (n2)  at (2,1.6) {2};
\node[circnode]   (n3)  at (-3,-2.8) {3};
\node[specialnode](n4)  at (2.5,-2.8) {4};
\node[circnode]   (n5)  at (0,-2.8) {5};
\node[circnode]   (n6)  at (1,2.8) {6};
\node[specialnode](n7)  at (3.4,2.8) {7};
\node[above=2pt] at (n7.north) {\scriptsize $d_{7}^{1}=12$};
\node[circnode]   (n8)  at (-0.6,2.8) {8};
\node[circnode]   (n9)  at (-1.6,0.9) {9};
\node[dasheddoublenode] (n10) at (-3,0) {10};
\node[above=2pt] at (n10.north) {\scriptsize $d_{10}^{2}=21$};

\draw[standardarc, very thick]
(n1) -- node[midway, above left=-3pt] {\scriptsize $(11,(7,4))$} (n2);

\draw[standardarc, very thick]
(n2) -- node[midway, below left=1pt] {\scriptsize $(7,(7,0))$} (n6);

\draw[standardarc, very thick]
(n6) -- node[midway, above] {\scriptsize $(10,(7,0))$} (n7);

\draw[standardarc, very thick]
(n2) -- node[midway, below right=3pt] {\scriptsize $(5,(0,4))$} (n4);

\draw[standardarc, very thick]
(n4) -- node[midway, below=3pt] {\scriptsize $(12,(7,4))$} (n5);

\draw[selectedarc, bend right=25, very thick]
(n1) to node[midway, right=24pt] {\scriptsize $(9,(7,0))$} (n5);

\draw[selectedarc, bend right=25, very thick]
(n5) to node[midway, left=24pt] {\scriptsize $(5,(0,5))$} (n1);

\draw[externalarc, very thick]
(n1) -- node[midway, below] {\scriptsize $12$} (n10);

\draw[externalarc, bend right=65, very thick]
(n4) to node[midway, left=3pt] {\scriptsize $5$} (n7);

\draw[standardarc, very thick]
(n5) -- node[midway, below] {\scriptsize $(13,(0,9))$} (n3);

\draw[standardarc, very thick]
(n3) -- node[midway, left] {\scriptsize $(9,(0,9))$} (n10);

\end{tikzpicture}
}

\caption{Second-stage decisions.} \label{fig:subfig_b}
\end{subfigure}

\caption{Illustrative example of a network design problem with two commodities: first-stage design decisions (left panel) and second-stage flow decisions after the realization of demands and arc capacities (right panel).}
\label{fig:network_design_example}

\end{figure}

\subsection{Benders reformulation}\label{sec_app_C2}

In this section, we derive the Benders reformulation of model \eqref{eq:network_meanrisk_pb_DEP_approach1} associated with the proposed L-shaped refinement chain cuts method. Introducing a discrete scenario set $\mathcal{S}$, model \eqref{eq:network_meanrisk_pb_DEP_approach1} can be reformulated as the following linear program (see \cite{RocUry2000} for details):
\begin{subequations} \label{eq:network_meanrisk_pb_DEP_approach1_scenarios}
\begin{align}
    \min_{\bm x,\bm y,\bm{\widetilde{y}},\bm z,\tau} & \quad \sum_{(i,j)\in\mathcal{A}^A}f_{ij}x_{ij} + (1-\beta)\sum_{s \in \mathcal{S}}p^s\bigg[\sum_{k\in\mathcal{K}} \sum_{(i,j)\in\mathcal{A}^A}c_{ij}^ky_{ij}^{ks}+\sum_{k\in\mathcal{K}} \sum_{(i,j)\in\mathcal{A}^D}\widetilde{c}_{ij}^k\widetilde{y}_{ij}^{ks} \bigg]+\nonumber\\
    & +\beta\bigg[\tau + \frac{1}{1-\alpha}\sum_{s \in \mathcal{S}}p^sz^s \bigg] \label{eq:network_meanrisk_objfunc_DEP_approach1_scenarios}\\
    \text{s.t.} & \quad \sum_{j \in \mathcal{N}^+(i)}\left(y_{ij}^{ks}+\widetilde{y}_{ij}^{ks}\right)-\sum_{j \in \mathcal{N}^-(i)}\left(y_{ij}^{ks}+\widetilde{y}_{ij}^{ks}\right)=d_{i}^{ks} \qquad \forall i \in \mathcal{N},k \in \mathcal{K},s\in\mathcal{S}\label{eq:network_meanrisk_constr1_DEP_approach1_scenarios}   \\
    & \quad \sum_{k\in\mathcal{K}} y_{ij}^{ks} \leq u_{ij}^sx_{ij} \qquad \forall (i,j)\in\mathcal{A}^A,s \in \mathcal{S} \label{eq:network_meanrisk_constr2_DEP_approach1_scenarios} \\
    & \quad \sum_{k\in\mathcal{K}} \sum_{(i,j)\in\mathcal{A}^A}c_{ij}^ky_{ij}^{ks}+\sum_{k\in\mathcal{K}} \sum_{(i,j)\in\mathcal{A}^D}\widetilde{c}_{ij}^k\widetilde{y}_{ij}^{ks} \leq \tau + z^s \qquad \forall s \in \mathcal{S} \label{eq:network_meanrisk_constr3_DEP_approach1_scenarios}\\
    & \quad x_{ij}\in\{0,1\} \qquad  \forall (i,j)\in\mathcal{A}^A \label{eq:network_meanrisk_constr4_DEP_approach1_scenarios} \\
    & \quad y_{ij}^{ks} \geq 0 \qquad \forall (i,j)\in\mathcal{A}^A,k\in\mathcal{K},s\in\mathcal{S} \label{eq:network_meanrisk_constr5_DEP_approach1_scenarios}\\
    & \quad \widetilde{y}_{ij}^{ks} \geq 0 \qquad \forall (i,j)\in\mathcal{A}^D,k\in\mathcal{K},s\in\mathcal{S} \label{eq:network_meanrisk_constr6_DEP_approach1_scenarios}\\
    & \quad z^s \geq 0 \qquad \forall s \in \mathcal{S} \label{eq:network_meanrisk_constr7_DEP_approach1_scenarios}\\
    & \quad \tau \in \mathbb{R} \label{eq:network_meanrisk_constr8_DEP_approach1_scenarios}.
\end{align}
\end{subequations}

Consider a refinement chain of the scenario set $\mathcal{S}$, and let $j\in{1,\ldots,J}$ denote a given refinement level. The Benders master problem \eqref{eq:master_problem_omega_i_j} at level $j$ is given by:
\begin{subequations} \label{eq:network_meanrisk_pb_DEP_approach1_MASTER_fixlevel}
    \begin{align}
    \min_{\bm x,\tau,\theta_i^{(j)}} & \quad \sum_{(i,j)\in\mathcal{A}^A}f_{ij}x_{ij}+\beta\tau + \sum_{i=1}^{m_j}\pi_i^{(j)} \theta_i^{(j)} \\
    \text{s.t.} & \quad \mathcal{Q}(\bm x,\tau,\bm \xi_i^{(j)}) \leq \theta_i^{(j)} \qquad \forall i=1,\ldots,m_j \\
    & \quad \theta_i^{(j)} \in \mathbb{R} \qquad \forall i=1,\ldots,m_j \\
    & \quad x_{ij}\in\{0,1\} \qquad  \forall (i,j)\in\mathcal{A}^A\\
    & \quad \tau \in \mathbb{R}.
    \end{align}
\end{subequations}

For fixed first-stage variables $\bm x \in \{0,1\}^{|\mathcal{A}|^A}$ and $\tau \in \mathbb{R}$, the recourse function $\mathcal{Q}(\bm x,\tau,\bm \xi_i^{(j)})$, defined as in \eqref{eq:scenario_subproblem_omega_i_j}, is obtained by minimizing:
\begin{equation*}
    \sum_{s:\bm \omega^s \in \Omega_i^{(j)}}p^s_{\Omega_{i}^{(j)}} \bigg\{(1-\beta)\bigg[\sum_{k\in\mathcal{K}} \sum_{(i,j)\in\mathcal{A}^A}c_{ij}^ky_{ij}^{ks}+\sum_{k\in\mathcal{K}} \sum_{(i,j)\in\mathcal{A}^D}\widetilde{c}_{ij}^k\widetilde{y}_{ij}^{ks}\bigg]+\frac{\beta}{1-\alpha}z^s\bigg\},
\end{equation*}
with respect to the second-stage decision variables $\bm y^s$, $\widetilde{\bm y}^s$, and $z^s$, for all scenarios $s\in\mathcal{S}$ such that $\bm\omega^s \in \Omega_i^{(j)}$. 


Therefore, the dual recourse subproblem $\mathcal{Q}_{dual}(\bm x,\tau,\bm \xi_i^{(j)})$ associated with subgroup $\Omega_i^{(j)}$ is given by:
\small
\begin{subequations} \label{eq:second_stage_subproblem_approach1_dual_fixlevel}
\begin{align}
    \max_{\bm \lambda^s,\bm \mu^s,\gamma^s:\bm \omega^s \in \Omega_{i}^{(j)}} & \quad \sum_{s: \bm \omega^s \in \Omega_i^{(j)}} \bigg[ \sum_{k\in\mathcal{K}} v^{ks}\left(\lambda_{on^k}^{ks}-\lambda_{dn^k}^{ks}\right)-\sum_{(i,j)\in\mathcal{A}^A}u_{ij}^sx_{ij}\mu_{ij}^s - \tau \gamma^s \bigg] \\
    \text{s.t.} & \quad \lambda_i^{ks}-\!\lambda_{j}^{ks}-\!\mu_{ij}^s-\!c_{ij}^k\gamma^s \!\leq p^s_{\Omega_{i}^{(j)}}(1-\beta)c_{ij}^k \hspace*{0.2cm} \forall (i,j) \in \mathcal{A}^A,k \in \mathcal{K},s:\bm \omega^s\in\Omega_i^{(j)}   \\
    & \quad \lambda_i^{ks}-\lambda_{j}^{ks}-\widetilde{c}_{ij}^k \gamma^s\leq p^s_{\Omega_{i}^{(j)}}(1-\beta)\widetilde{c}^k_{ij} \qquad \forall (i,j) \in \mathcal{A}^D,k \in \mathcal{K}, \forall s:\bm \omega^s\in\Omega_i^{(j)}   \\
    & \quad (1-\alpha)\gamma^s \leq p^s_{\Omega_{i}^{(j)}}\beta \qquad \forall s:\bm \omega^s\in\Omega_i^{(j)} \\
    & \quad \lambda_{i}^{ks} \in \mathbb{R} \qquad \forall i\in\mathcal{N},k\in\mathcal{K},s:\bm \omega^s\in\Omega_i^{(j)}\\
    & \quad \mu_{ij}^s \geq 0 \qquad \forall (i,j)\in\mathcal{A}^A,s:\bm \omega^s\in\Omega_i^{(j)}\\
    & \quad \gamma^s \geq 0 \qquad \forall s:\bm \omega^s\in\Omega_i^{(j)}.
\end{align}
\end{subequations}
\normalsize

The optimal solutions of \eqref{eq:second_stage_subproblem_approach1_dual_fixlevel} are then used to generate the optimality cuts \eqref{eq:optimality_cuts_omega_i_j}.

\subsection{Supplementary numerical results}\label{sec_app_C3}

\begin{sidewaystable}
\caption{Detailed results for the R04 and R05 instances, obtained by solving model \eqref{eq:master_problem_omega_i_j} at different levels of the refinement chain. The time reduction is computed with respect to the single-cut formulation. For each instance, the best reduction is reported in bold, while results that also outperform the multi-cut formulation are underlined.}\label{tab:R04R05_beta05}
\begin{tabular*}{\textheight}{@{\extracolsep\fill} l l ll ll ll ll ll}
\toprule%
& & \multicolumn{6}{@{}l@{}}{Disjoint partitions}& \multicolumn{4}{@{}l@{}}{Fixed-scenario}
\\
\cmidrule{3-8} \cmidrule{9-12}%
& & & & \multicolumn{2}{@{}l@{}}{Sequential grouping} & \multicolumn{2}{@{}l@{}}{Optimal grouping} & & & \multicolumn{2}{@{}l@{}}{Optimal grouping}
\\
\cmidrule{5-6} \cmidrule{7-8} \cmidrule{11-12}%
\multirow{2}{*}{Instance} & \multirow{2}{*}{$|\mathcal{S}|$} & Group & No. of & Time & Time  & Time & Time & Group & No. of & Time & Time
\\
& & size & groups & (s) & red. (\%) & (s) & red. (\%) & size & groups & (s) & red. (\%)
\\
\midrule
\multirow{9}{*}{R04} & \multirow{9}{*}{256} & 1 (multi-cut)   & 256 & 876.38   & 97.41\% & 876.38   & 97.41\% & $-$ & $-$ & $-$      & $-$ 
\\
& & 2   & 128 & 1005.75  & 97.03\% & \textbf{\underline{505.35}}   & $\mathbf{\underline{98.51\%}}$ & 2   & 255 & 1135.79  & 96.65\% 
\\
& & 4   & 64  & 904.86   & 97.33\% & 1759.61  & 94.81\% & 4   & 85  & 1053.57  & 96.89\% 
\\
& & 8   & 32  & 1554.00  & 95.41\% & 1414.93  & 95.82\% & $-$ & $-$ & $-$      & $-$  
\\
& & 16  & 16  & 3480.50  & 89.73\% & 3700.02  & 89.08\% & 16  & 17  & 1690.47  & 95.01\% 
\\
& & 32  & 8   & 3832.67  & 88.69\% & 4209.92  & 87.57\% & $-$ & $-$ & $-$      & $-$     
\\
& & 64  & 4   & 15427.53 & 54.46\% & 31593.89 & 6.75\%  & $-$ & $-$ & $-$      & $-$     
\\
& & 128 & 2   & 19290.63 & 43.06\% & 15035.36 & 55.62\% & $-$ & $-$ & $-$      & $-$     
\\
& & 256 (single-cut) & 1   & 33880.20 & 0.00\%  & 33880.20 & 0.00\%  & 256 (single-cut) & 1   & 33880.20 & 0.00\%  
\\
\midrule
\multirow{9}{*}{R05} & \multirow{9}{*}{256} & 1 (multi-cut)   & 256 & 3268.19  & 68.02\% & 3268.19  & 68.02\% & $-$ & $-$ & $-$     & $-$
\\
& & 2   & 128 & 4677.10  & 53.36\% & 5597.87  & 45.23\% & 2   & 255 & 5060.38 & 50.49\% 
\\
& & 4   & 64  & 5687.08  & 44.35\% & 3727.93  & 63.52\% & 4   & 85  & 3276.99 & 67.94\% 
\\
& & 8   & 32  & 3864.34  & 62.19\% & 9010.64  & 11.83\% & $-$ & $-$ & $-$     & $-$      
\\
& & 16  & 16  & 5243.62  & 48.69\% & 3453.01  & 66.21\% & 16  & 17  & \textbf{\underline{3210.15}} & $\mathbf{\underline{68.59}}$\% 
\\
& & 32  & 8   & 3807.63  & 62.74\% & \underline{3232.68}  & \underline{68.37\%} & $-$ & $-$ & $-$     & $-$      
\\
& & 64  & 4   & 4141.67  & 59.48\% & 5695.52  & 44.27\% & $-$ & $-$ & $-$     & $-$      
\\
& & 128 & 2   & 7899.18  & 22.71\% & 8286.60  & 18.92\% & $-$ & $-$ & $-$     & $-$      
\\
& & 256 (single-cut) & 1   & 10220.09 & 0.00\%  & 10220.09 & 0.00\%  & 256 (single-cut) & 1   & 10220.09 & 0.00\%  
\\
\botrule
\end{tabular*}
\end{sidewaystable}

\begin{sidewaystable}
\caption{Detailed results for the R07 instances, obtained by solving model \eqref{eq:master_problem_omega_i_j} at different levels of the refinement chain with $|\mathcal{S}|=256$. The relative gap reduction is computed with respect to the single-cut formulation. The best reduction is reported in bold, while results that also outperform the multi-cut formulation are underlined.}\label{tab:R07_beta05}
\begin{tabular*}{\textheight}{@{\extracolsep\fill} l l ll ll ll ll}
\toprule%
& & \multicolumn{4}{@{}l@{}}{Disjoint partitions}& \multicolumn{4}{@{}l@{}}{Fixed-scenario}
\\
\cmidrule{3-6} \cmidrule{7-10}%
& & & & \multicolumn{2}{@{}l@{}}{Optimal grouping} & & & \multicolumn{2}{@{}l@{}}{Optimal grouping}
\\
\cmidrule{5-6}\cmidrule{9-10}%
\multirow{2}{*}{Instance} & \multirow{2}{*}{$|\mathcal{S}|$} & Group & No. of & Relative & Relative & Group & No. of & Relative & Relative
\\
& & size & groups & gap & gap red. (\%) & size & groups & gap & gap red. (\%)
\\
\midrule
\multirow{9}{*}{R07} & \multirow{9}{*}{256} & 1 (multi-cut)   & 256 & $1.81 \times 10^{-2}$ & 28.16\% & $-$ & $-$ & $-$ & $-$  
\\
& & 2   & 128 & $\mathbf{\underline{8.60 \times 10^{-3}}}$ & $\mathbf{\underline{65.78\%}}$ & 2 & 255 & $\underline{1.50 \times 10^{-2}}$ & $\underline{40.25\%}$ 
\\
&& 4   & 64  & $\underline{1.14 \times 10^{-2}}$ & $\underline{54.50\%}$ & 4 & 85  & $\underline{1.42 \times 10^{-2}}$ & $\underline{43.56\%}$
\\
&& 8   & 32  & $\underline{1.42 \times 10^{-2}}$ & $\underline{43.62\%}$ & $-$ & $-$ & $-$ & $-$   
\\
&& 16  & 16  & $\underline{1.47 \times 10^{-2}}$ & $\underline{41.70\%}$ & 16 & 17 & $\underline{1.66 \times 10^{-2}}$ & $\underline{33.89\%}$ 
\\
&& 32  & 8   & $\underline{1.79 \times 10^{-2}}$ & $\underline{28.87\%}$ & $-$ & $-$ & $-$ & $-$    
\\
&& 64  & 4   & $2.74 \times 10^{-2}$ & -9.13\% & $-$ & $-$ & $-$ & $-$      
\\
&& 128 & 2   & $2.70 \times 10^{-2}$ & -7.21\% & $-$ & $-$ & $-$ & $-$      
\\
&& 256 (single-cut) & 1   & $2.51 \times 10^{-2}$ & 0.00\% & 256 & 1 & $2.51 \times 10^{-2}$ & 0.00\% 
\\
\botrule
\end{tabular*}
\end{sidewaystable}

\begin{table}[htbp]
\centering
\caption{Detailed results for the R07 instances, obtained by solving model \eqref{eq:master_problem_omega_i_j} at different levels of the refinement chain with $|\mathcal{S}|\in\{1024,2048\}$. The relative gap reduction is computed with respect to the single-cut formulation. For each instance, the best reduction is reported in bold, while results that also outperform the multi-cut formulation are underlined.}
\label{tab:R07_10242048}
\begin{tabular}{l l l l l l}
\toprule
& &  & & \multicolumn{2}{@{}l@{}}{Disjoint partitions}
\\
\cmidrule{5-6}
& & & & \multicolumn{2}{@{}l@{}}{Optimal grouping}
\\
\cmidrule{5-6}
\multirow{2}{*}{Instance} & \multirow{2}{*}{$|\mathcal{S}|$} & Group & No. of & Relative & Relative \\
& & size & groups & gap & gap red. (\%)\\
\midrule
\multirow{22}{*}{R07} & \multirow{10}{*}{1024} & 1 (multi-cut) & 1024 & $1.13 \times 10^{-2}$ & 75.57 \\
& & 2 & 512 & $\mathbf{\underline{9.69 \times 10^{-3}}}$ & $\mathbf{\underline{79.10}}$ \\
& & 4 & 256 & $2.05 \times 10^{-2}$ & 55.88 \\
& & 8 & 128 & $1.35 \times 10^{-2}$ & 70.88 \\
& & 16 & 64 & $1.35 \times 10^{-2}$ & 70.78 \\
& & 32 & 32 & $1.41 \times 10^{-2}$ & 69.57 \\
& & 64 & 16 & $1.47 \times 10^{-2}$ & 68.24 \\
& & 128 & 8 & $2.96 \times 10^{-2}$ & 36.19 \\
& & 256 & 4 & $2.61 \times 10^{-2}$ & 43.73 \\
& & 512 & 2 & $4.61 \times 10^{-2}$ & 0.64 \\
& & 1024 (single-cut) & 1 & $4.64 \times 10^{-2}$ & 0.00 \\
\cmidrule{2-6}
 & \multirow{12}{*}{2048}
& 1 (multi-cut) & 2048 & $1.44 \times 10^{-2}$ & 72.57 \\
& & 2 & 1024 & $\underline{1.23 \times 10^{-2}}$ & \underline{76.46} \\
& & 4 & 512 & $1.64 \times 10^{-2}$ & 68.73 \\
& & 8 & 256 & $1.75 \times 10^{-2}$ & 66.70 \\
& & 16 & 128 & $\mathbf{\underline{1.18 \times 10^{-2}}}$ & $\mathbf{\underline{77.53}}$ \\
& & 32 & 64 & $1.53 \times 10^{-2}$ & 70.83 \\
& & 64 & 32 & $1.69 \times 10^{-2}$ & 67.73 \\
& & 128 & 16 & $2.25 \times 10^{-2}$ & 57.03 \\
& & 256 & 8 & $4.20 \times 10^{-2}$ & 19.86 \\
& & 512 & 4 & $4.64 \times 10^{-2}$ & 11.50 \\
& & 1024 & 2 & $3.88 \times 10^{-2}$ & 25.98 \\
& & 2048 (single-cut) & 1 & $5.25 \times 10^{-2}$ & 0.00 \\
\bottomrule
\end{tabular}
\end{table}

\begin{table}
\caption{Detailed results for the R04 and R05 instances, obtained by applying Algorithm \ref{algo:refchaincuts_algo} between consecutive levels of the refinement chain. The time reduction is computed with respect to the single-cut formulation. For each instance, the best reduction is reported in bold, while results that also outperform the multi-cut formulation are underlined.}\label{tab:R04R05_ALGO_beta05}
\begin{tabular*}{\textwidth}{@{\extracolsep\fill} l l l ll ll}
\toprule%
& & & \multicolumn{4}{@{}l@{}}{Disjoint partitions}
\\
\cmidrule{4-7}
& & & \multicolumn{2}{@{}l@{}}{Sequential grouping} & \multicolumn{2}{@{}l@{}}{Optimal grouping}
\\
\cmidrule{4-5}\cmidrule{6-7}%
\multirow{2}{*}{Instance} & \multirow{2}{*}{$|\mathcal{S}|$} & Consecutive & Time & Time & Time & Time
\\
& & group sizes & (s) & red. (\%) & (s) & red. (\%)
\\
\midrule
\multirow{8}{*}{R04} &\multirow{8}{*}{256} &  1--2     & \textbf{\underline{470.83}}   & $\mathbf{\underline{98.61}}$ & \underline{645.71}    & $\underline{98.09}$
\\
&& 2--4     & \underline{843.95}   & $\underline{97.51}$ & 1993.86   & 94.11
\\
&& 4--8     & 2264.57  & 93.32 & 3446.07   & 89.83
\\
&& 8--16    & 2762.38  & 91.85 & 3947.77   & 88.35
\\
&& 16--32   & 12356.87 & 63.53 & 4502.34   & 86.71
\\
&& 32--64   & 5509.98  & 83.74 & 24865.04  & 26.61
\\
&& 64--128  & 15189.61 & 55.17 & 29964.42  & 11.56
\\
&& 128--256 & 14792.84 & 56.34 & 48151.83  & 1.33
\\
\midrule
\multirow{8}{*}{R05} & \multirow{8}{*}{256} & 1--2     & 7991.12 & 21.81 & 7833.45  & 23.35
\\
&& 2--4     & 6193.63 & 39.40 & 26194.87 & -156.31
\\
&& 4--8     & 4242.10 & 58.49 & 10108.77 & 1.09
\\
&& 8--16    & 4676.97 & 54.24 & 5595.39  & 45.25
\\
&& 16--32   & \textbf{3854.77} & $\mathbf{62.28}$ & 6720.51  & 34.24
\\
&& 32--64   & 7484.40 & 26.77 & 11073.77 & -8.35
\\
&& 64--128  & 7654.11 & 25.11 & 10247.13 & -0.26
\\
&& 128--256 & 7419.94 & 27.40 & 8230.61  & 19.47
\\
\botrule
\end{tabular*}
\end{table}

\begin{table}
\caption{Detailed results for the R07 instances, obtained by applying Algorithm \ref{algo:refchaincuts_algo} between consecutive levels of the refinement chain. The relative gap reduction is computed with respect to the single-cut formulation. For each instance, the best reduction is reported in bold, while results that also outperform the multi-cut formulation are underlined.}\label{tab:R07_ALGO_beta05}
\begin{tabular*}{\textwidth}{@{\extracolsep\fill} l l l ll}
\toprule%
& & & \multicolumn{2}{@{}l@{}}{Disjoint partitions}
\\
\cmidrule{4-5}
& & & \multicolumn{2}{@{}l@{}}{Optimal grouping}
\\
\cmidrule{4-5}
\multirow{2}{*}{Instance} & \multirow{2}{*}{$|\mathcal{S}|$} & Consecutive & Relative & Relative
\\
& & group sizes & gap & gap red. (\%)
\\
\midrule
\multirow{27}{*}{R07} & \multirow{8}{*}{256} & 1--2     & $\mathbf{\underline{9.69 \times 10^{-3}}}$ & $\mathbf{\underline{61.44}}$
\\
&& 2--4      & $\underline{1.17 \times 10^{-2}}$ & \underline{$53.51$}
\\
&& 4--8      & $2.13 \times 10^{-2}$ & 15.35
\\
&& 8--16    & $\underline{1.65 \times 10^{-2}}$ & \underline{$34.40$}
\\
&& 16--32    & $2.51 \times 10^{-2}$ & 0.24
\\
&& 32--64   & $3.53 \times 10^{-2}$ & -40.38
\\
&& 64--128   & $6.99 \times 10^{-2}$ & -178.14
\\
&& 128--256 & $2.55 \times 10^{-2}$ & -1.47
\\
\cmidrule{2-5}
& \multirow{9}{*}{1024} &1--2       & $1.29 \times 10^{-2}$ & 72.13 \\
&&2--4       & $1.35 \times 10^{-2}$ & 70.84 \\
&&4--8       & $\mathbf{1.29 \times 10^{-2}}$ & $\mathbf{72.25}$ \\
&&8--16      & $2.27 \times 10^{-2}$ & 51.11 \\
&&16--32     & $3.00 \times 10^{-2}$ & 35.37 \\
&&32--64     & $1.71 \times 10^{-2}$ & 63.19 \\
&&64--128    & $1.62 \times 10^{-2}$ & 65.12 \\
&&128--256   & $2.10 \times 10^{-2}$ & 54.67 \\
&&256--512   & $2.74 \times 10^{-2}$ & 40.83 \\
&&512--1024  & $3.94 \times 10^{-2}$ & 15.04 \\
\cmidrule{2-5}
& \multirow{10}{*}{2048} & 1--2 & $\mathbf{\underline{1.01 \times 10^{-2}}}$ & $\mathbf{\underline{80.75}}$ \\
&&2--4 & $\underline{1.25 \times 10^{-2}}$ & \underline{76.23} \\
&&4--8 & $1.70 \times 10^{-2}$ & 67.65 \\
&&8--16 & $1.88 \times 10^{-2}$ & 64.18 \\
&&16--32 & $1.64 \times 10^{-2}$ & 68.81 \\
&&32--64 & $2.46 \times 10^{-2}$ & 53.20 \\
&&64--128 & $5.00 \times 10^{-2}$ & 4.60 \\
&&128--256 & $4.82 \times 10^{-2}$ & 8.19 \\
&&256--512 & $4.56 \times 10^{-2}$ & 13.14 \\
&&512--1024 & $7.65 \times 10^{-2}$ & -45.89 \\
&&1024--2048 & $1.06 \times 10^{-1}$ & -101.79 \\
\botrule
\end{tabular*}
\end{table}

\end{appendices}

\end{document}